\documentclass[12pt,a4paper,english,reqno]{amsart}

\usepackage{amsmath,amsfonts,amssymb,mathrsfs}
\usepackage[english]{babel}
\usepackage[utf8x]{inputenc}
\usepackage[T1]{fontenc}
\usepackage{graphicx}
\usepackage{tabularx}
\usepackage{hyperref}
\usepackage{mathtools}
\usepackage[margin=1in]{geometry}
\usepackage{cleveref}
\usepackage{comment}

\newcommand{\nocontentsline}[3]{}
\newcommand{\tocless}[2]{\bgroup\let\addcontentsline=\nocontentsline#1{#2}\egroup}
\DeclareUnicodeCharacter{0177}{\^u}

\DeclarePairedDelimiter\floor{\lfloor}{\rfloor}
\DeclareMathOperator{\rank}{rank}
\DeclareMathOperator{\End}{End}
\DeclareMathOperator{\Hom}{Hom}
\DeclareMathOperator{\Lie}{Lie}
\DeclareMathOperator{\SL}{SL}

\DeclareMathOperator{\SO}{SO}
\DeclareMathOperator{\spec}{Spec}
\DeclareMathOperator{\GU}{GU}
\DeclareMathOperator{\GSp}{GSp}
\DeclareMathOperator{\GSpin}{GSpin}

\DeclareMathOperator{\Id}{Id}
\DeclareMathOperator{\Vol}{Vol}
\DeclareMathOperator{\tr}{tr}
\DeclareMathOperator{\Gal}{Gal}
\DeclareMathOperator{\CH}{CH} 
\DeclareMathOperator{\Pic}{Pic} 
\DeclareMathOperator{\Disc}{Disc}
\DeclareMathOperator{\Frob}{Frob}
\DeclareMathOperator{\re}{Re}
\DeclareMathOperator{\spn}{span}
\DeclareMathOperator{\length}{length}
\DeclareMathOperator{\val}{val}
\DeclareMathOperator{\disc}{disc}
\DeclareMathOperator{\Sh}{Sh}

\numberwithin{equation}{section}

\theoremstyle{plain}

\newtheorem{theorem}{Theorem}[section]
\newtheorem*{thm}{Theorem}
\newtheorem{proposition}[theorem]{Proposition}
\newtheorem{lemme}[theorem]{Lemma}
\newtheorem{claim}[theorem]{Claim}
\newtheorem{lemma}[theorem]{Lemma}

\newtheorem{corollaire}[theorem]{Corollary}
\newtheorem{conjecture}[theorem]{Conjecture}
\theoremstyle{definition}
\newtheorem{definition}[theorem]{Definition}
\newtheorem{assumption}[theorem]{Assumption}
\theoremstyle{remark}
\newtheorem{remarque}[theorem]{Remark}

\newcommand{\R}{\mathbb{R}}
\newcommand{\K}{\mathbb{K}}
\newcommand{\Z}{\mathbb{Z}}
\renewcommand{\P}{\mathbb{P}}
\newcommand{\Q}{\mathbb{Q}}

\newcommand{\N}{\mathbb{N}}

\newcommand{\A}{\mathbb{A}}
\newcommand{\C}{\mathbb{C}}
\newcommand{\cC}{\mathcal{C}}

\newcommand{\G}{\mathrm{G}}

\newcommand{\Ok}{\mathcal{O}}
\newcommand{\cP}{\mathcal{P}}
\newcommand{\cA}{\mathcal{A}}
\newcommand{\B}{\mathcal{B}}

\newcommand{\cO}{\mathcal{O}}
\newcommand{\cM}{\mathcal{M}}
\newcommand{\cX}{\mathcal{X}} 
\newcommand{\cY}{\mathcal{Y}}
\newcommand{\cZ}{\mathcal{Z}}
\newcommand{\cW}{\mathcal{W}}

\newcommand{\cS}{\mathcal{S}}
\newcommand{\cF}{\mathcal{F}}
\newcommand{\cN}{\mathcal{N}}

\newcommand{\bF}{\mathbb{F}}

\newcommand{\bH}{\mathbb{H}}

\newcommand{\fP}{\mathfrak{P}}
\newcommand{\fe}{\mathfrak{e}}

\newcommand{\bfN}{\mathbf{N}}
\newcommand{\bfV}{\mathbf{V}}
\newcommand{\bfH}{\mathbf{H}}
\newcommand{\bfQ}{\mathbf{Q}}
\newcommand{\bfx}{\mathbf{x}}
\newcommand{\bfone}{\mathbf{1}}
\newcommand{\bomg}{\boldsymbol{\omega}}

\newcommand{\zero}{\mathrm{zero}}
\newcommand{\bad}{\mathrm{bad}}
\newcommand{\good}{\mathrm{good}}
\newcommand{\mt}{\mathrm{mt}}
\newcommand{\er}{\mathrm{er}}
\newcommand{\univ}{\mathrm{univ}}
\newcommand{\ks}{\mathrm{KS}}
\newcommand{\nr}{\mathrm{nr}}

\newcommand{\inj}{\hookrightarrow}

\usepackage[usenames,dvipsnames]{color}  

\newcommand{\salim}[1]{{\color{Purple} \sf  Salim: [#1]}}
\newcommand{\ananth}[1]{{\color{Red} \sf Ananth: [#1]}}

\newcommand{\yunqing}[1]{{\color{Blue} \sf  Yunqing: [#1]}}

\title[]{Exceptional jumps of Picard ranks of reductions of K3 surfaces over number fields}

\author{Ananth N. Shankar}
\address{The department of mathematics, MIT, 182 memorial drive, Cambridge MA 02139.}
\email{ananths@mit.edu}
\author{Arul Shankar}
\address{215 Huron st, Toronto, ON M5T 1R2.}
\email{ashankar@math.utoronto.ca}
\author{Yunqing Tang}
\address{Department of Mathematics, University of California, Berkeley, Evans Hall, Berkeley CA 94720, USA.}
\email{yunqing.tang@berkeley.edu}
\author{Salim Tayou}
\address{45, Rue d'Ulm, 75230, Paris.}
\email{salim.tayou@ens.psl.eu}
\date\today

\begin{document}

\begin{abstract}
Given a K3 surface $X$ over a number field $K$ with potentially good reduction everywhere, we prove that the set of primes of $K$ where the geometric Picard rank jumps is infinite. As a corollary, we prove that either $X_{\overline{K}}$ has infinitely many rational curves or $X$ has infinitely many unirational specializations. 

Our result on Picard ranks is a special case of more general results on exceptional classes for K3 type motives associated to GSpin Shimura varieties. These general results have several other applications. For instance, we prove that an abelian surface over a number field $K$ with potentially good reduction everywhere is isogenous to a product of elliptic curves modulo infinitely many primes of $K$. 

\end{abstract}

\thanks{}
\maketitle
\setcounter{tocdepth}{1}
\tableofcontents

\section{Introduction}

\subsection{Picard rank jumps for K3 surfaces}
Let $X$ be a K3 surface over a number field $K$. Let $\mathcal{X}\xrightarrow{\pi} S$ be a smooth and projective model of $X$ over an open sub-scheme $S$ of $\mathrm{Spec}(\mathcal{O}_K)$, the spectrum of the ring of integers $\mathcal{O}_K$ of $K$. For every place $\mathfrak{P}$ of $K$ in $S$, let $\mathcal{X}_{\overline{\mathfrak{P}}}$ be the geometric fiber of $\pi$ at $\mathfrak{P}$. There is an injective specialization map between Picard groups (see \cite[Chap.17 Prop.2.10]{huybrechts}): 
$$\mathrm{sp}_{\mathfrak{P}}:\mathrm{Pic}(X_{\overline{K}})\hookrightarrow \mathrm{Pic}(\mathcal{X}_{\overline{\mathfrak{P}}}),$$ which implies the inequality between Picard ranks $$\rho(\mathcal{X}_{\overline{\mathfrak{P}}}):=\rank_{\Z}\Pic(\cX_{\overline{\fP}})\geq\rho(X_{\overline{K}}):=\rank_{\Z}\Pic(X_{\overline{K}}).$$  

In this paper, our main result implies the following theorem.
\begin{theorem}\label{k3ar}  
Let $X$ be a K3 surface over a number field $K$ and assume that $X$ admits, up to a finite extension of $K$, a projective smooth model $\mathcal{X}\rightarrow \mathrm{Spec}(\mathcal{O}_K)$. Then there are infinitely many finite places $\fP$ of $K$ such that $\rho(\cX_{\overline{\fP}})> \rho(X_{\overline{K}})$.
\end{theorem}
This question has been raised by Charles \cite{charles} inspired by the work of Bogomolov--Hassett--Tschinkel \cite{bht} and Li--Liedtke \cite{ll} (see also \cite{costatschinkel}, \cite{costajahnel}). 
By \cite[Theorem 1]{charles}, up to a finite extension of $K$, the Picard rank $\rho(\cX_{\overline{\fP}})$ for a density one set of primes $\fP$ is completely determined by $\rho({X}_{\overline{K}})$ and the endomorphism field $E$ of the sub-Hodge structure $T(X_\sigma)$ of $H^2(X_\sigma^{an}(\C),\Q))$ given by the orthogonal complement of $\Pic(X_{\overline{K}})$ with respect to the intersection form, for any embedding $\sigma$ of $K$ in $\C$. For instance, if $E$ is CM or $\dim_E T(X_\sigma)$ is even, then $\rho(\cX_{\overline{\fP}})=\rho(X_{\overline{K}})$ for a density one set of primes $\fP$. In this situation, our theorem proves that the density zero set where $\rho(\cX_{\overline{\fP}})>\rho(X_{\overline{K}})$ is in fact infinite.

\subsection{Rational curves on K3 surfaces}
As an application of the above theorem, let us first recall the following conjecture.
\begin{conjecture}
Let $X$ be a K3 surface over an algebraically closed field $k$. Then $X$ contains infinitely many rational curves. 
\end{conjecture}
The first result towards this conjecture is the one attributed to Bogomolov and Mumford and appearing in \cite[Appendix]{morimukai} which states that every K3 surface over $\C$ contains a rational curve. This conjecture has been settled in recent years in many cases thanks to the work of many people \cite{bht,ll,charles,bt1,tayouelliptic,chen,chenlewis}. In characteristic zero, this conjecture has been solved in full generality in \cite[Theorem A]{chengoun}. Our theorem imply the following alternative for K3 surfaces over number fields admitting everywhere potentially good reduction.
\begin{corollaire}\label{rational}
Let $X$ be a K3 surface over a number field $K$ and assume that $X$ has potentially good reduction everywhere. Then either:
\begin{enumerate}
    \item $X_{\overline{K}}$ contains infinitely many rational curves;
    \item $X$ has infinitely many unirational and hence supersingular specializations.
\end{enumerate}
\end{corollaire}



\subsection{Exceptional splitting of abelian varieties}
Let $A$ denote a geometrically simple abelian variety over a number field $K$. Assuming the Mumford--Tate conjecture for $A$, Zwyina \cite[Corollary 1.3]{zywina} proved that the mod $\fP$ reduction $A_{\fP}$ is geometrically simple for a density one set of primes of $K$ (up to replacing $K$ by a finite extension) if and only if $\End(A_{\overline{K}})$ is commutative. As an application of the proof of Theorem \ref{k3ar} (more precisely, \Cref{main_sp_end}), we prove that the density zero set of primes with $A_{\fP}$ geometrically non simple is infinite for certain classes of abelian varieties $A$ which are closely related to Kuga--Satake abelian varieties. We note that the Mumford--Tate conjecture is known (by work of Tankeev \cite{tankeev,tankeev2} and Vasiu \cite{vasiu1}) for the classes of abelian varieties that we treat.  

As a first example, we observe that the moduli space of principally polarized abelian surfaces can be realized as a GSpin Shimura variety of dimension $3$ and in this case, the associated Kuga--Satake abelian varieties are isogenous to powers of abelian surfaces. We therefore obtain:

\begin{theorem}\label{thm_absurf}
Every $2$-dimensional abelian scheme over $\cO_K$ admits infinitely many places of geometrically non simple reduction.
\end{theorem}

More generally, consider the setting of $(V,Q)$, a $(b+2)$-dimensional quadratic space over $\Q$ with signature $(b,2)$. 

\begin{assumption}\label{ass_KSsplit}

\end{assumption}
Such a quadratic space (and its Clifford algebra) corresponds to a family of abelian varieties, called Kuga--Satake abelian varieties (see \S\ref{sec_KS}). Every such abelian variety $A$ has a splitting of the form $A = A^+\times A^-$, induced by the grading of the Clifford algebra. By the Kuga--Satake construction, it follows that $A^+$ is isogenous to $B^{2^{n}}$ for some lower-dimensional abelian variety $B$. Our next result concerns places of split reduction of $B$ when $A$ is defined over some number field $K$. Generically, $\End(B_{\bar{K}})=\Z$ (see \S\ref{sec_KSsplit}) and hence the set of places of geometrically split reductions has density zero by \cite{zywina}, as the Mumford--Tate conjecture is known for $A$, and therefore for $B$. We prove then the following result:
\begin{theorem}\label{cor_gspin}
Consider the above setting, with the assumption that $B$ extends to an abelian scheme $\mathcal{B}\rightarrow\mathrm{Spec}(\cO_K)$ (and therefore, $A$ also extends to an abelian scheme $\mathcal{A} \rightarrow\mathrm{Spec}(\cO_K)$). Then there are infinitely many finite places $\fP$ of $K$ such that $\mathcal{B}_{\fP}$ is geometrically non simple.
\end{theorem}


We also have similar results for abelian varieties parameterized by Shimura varieties associated to the unitary similitude group $\mathrm{GU}(r,1)$, $r\geq1$, see \S\ref{sec_unitary} for the precise definitions. 
\begin{corollaire}\label{cor_unitary}
Let $E$ be an imaginary quadratic field and let $\cA$ be a principally polarized abelian scheme over $\cO_K$. Suppose that there is an embedding $\cO_E\subset \End(\cA)$ which is compatible with the polarization on $\cA$, and that the action of $\cO_E$ on $\Lie \cA_K$ has signature $(r,1)$. 
Then there are infinitely many finite places $\fP$ of $K$ such that $\cA_{\fP}$ admits a geometric isogeny factor which is an elliptic curve CM by $E$.
\end{corollaire}

We should mention that the assumption of potentially good reduction in the above statements has been removed recently by one of the authors in \cite{tayou-bad-reduction}, so that all the above results hold unconditionally.

\subsection{GSpin Shimura varieties.}

The above theorems can be reformulated within  the more general framework of intersections of a (non-special) arithmetic $1$-cycle and special divisors in GSpin Shimura varieties as follows. Let $(L,Q)$ be an integral quadratic even lattice of signature $(b,2)$ with $b\geq 3$.\footnote{In this paper, we focus on $b\geq 3$ case since the essential cases when $b=1,2$ have been treated in \cite{charles1,st}.}
Assume that $L$ is a maximal lattice in $V:=L\otimes_{\Z}\Q$ over which $Q$ is $\Z$-valued. Associated to this data is a {\it GSpin Shimura variety} $M$, which is a Deligne--Mumford stack over $\Q$, see \S\ref{gspin_Q}. It is a Shimura variety of Hodge type which (by the work of Andreatta--Goren--Howard--Madapusi-Pera \cite{agmp2}) admits a normal flat integral model $\mathcal{M}$ over $\Z$. This model is smooth at primes $p$ that do not divide $\Disc(Q)$. Moreover, there is a family of the so called \emph{Kuga--Satake abelian scheme} $\mathcal{A}^\univ\rightarrow \mathcal{M}$, see \S\S\ref{sec_KS},\ref{integral}. For every $m\in \Z_{>0}$, a {\it special divisor} $\mathcal{Z}(m)\rightarrow \mathcal{M}$ is constructed in \cite{agmp2}\footnote{also called Heegner divisor in the literature.}  parameterizing Kuga-Satake abelian varieties which admit \emph{special} endomorphisms $s$ such that $s\circ s = [m]$ (see \S\ref{special}). In particular, the moduli space of polarized K3 surfaces can be embedded in a GSpin Shimura variety (see \S\ref{sec_K3}), and special divisors parameterize K3 surfaces with Picard rank greater than that of the generic K3 surface. Our main theorem is the following.
\begin{theorem}\label{main}
Let $K$ be a number field and let $\cY\in \mathcal{M}(\mathcal{O}_K)$. Assume that $\cY_K\in M(K)$ is Hodge-generic. Then there exist infinitely many finite places $\fP$ of $K$ modulo which $\cY$ lies in the image of $\mathcal{Z}(m)\rightarrow \mathcal{M}$ for some $m\in \Z_{>0}$ (here $m$ depends on $\fP$).


\end{theorem}
Here we say that $x\in M(K)$ is \emph{Hodge-generic} if for one embedding (equivalently any) $\sigma:K\hookrightarrow \C$, the point $x^\sigma\in M(\C)$ does not lie on any divisor $\mathcal{Z}(m)(\C)$. This is a harmless assumption since all $\cZ(m)_K$ are union of GSpin Shimura varieties associated to rational quadratic spaces having signature $(b-1,2)$. Hence, we may and will always work with the smallest GSpin sub Shimura variety of $M$ containing $\cY_K$.\footnote{Note that this use of the term Hodge-generic is \emph{not standard} -- we do not assume that the Mumford--Tate group associated to $\cY_K$ is equal to the group of Spinor similitudes associated to $(V,Q)$.}\\


Theorem \ref{main} is a generalization of the main results of \cite{charles1,st}. 
In the complex setting, the analogous results are well-understood. More precisely, the Noether-Lefschetz locus of a non-trivial variation of Hodge structures of weight $2$ with $h^{2,0}=1$ over a complex quasi-projective curve is dense for the analytic topology by a well-known result of Green \cite[Prop.~17.20]{voisin}. When this variation is of K3 type, the main result of \cite{tayouequi} shows in fact that this locus is equidistributed with respect to a natural measure. In the global function field setting, the main result of \cite{mst} shows that given a non-isotrivial ordinary  abelian surface over a projective curve $C$ over $\overline{\mathbb{F}}_p$, there are infinitely many $\overline{\mathbb{F}}_p$-points in $C$ such that the corresponding abelian surface is not simple. This result is analogous to \Cref{main} in the function field setting when $b=2,3$.

We now say a word about the potentially good reduction hypothesis (i.e., the fact that we require $\cY$ to be an $\cO_K$-point of $\mathcal{M}$, up to a finite extension of the base field, as opposed to a $\cO_K[1/N]$-point). The boundary components in the Satake compactification\footnote{Which is a projective variety.} of $M$ have dimension either $0$ or $1$, whereas the Shimura variety itself is $b$-dimensional. As the boundary has large codimension in the ambient Shimura variety, it follows that ``most'' points have potentially good reduction, and so our good reduction hypothesis is not an especially stringent condition. A large family of points with potentially good reduction everywhere  (in the case $b = 2c$) can be obtained as follows: consider a real quadratic field $F/\Q$, and a $(c+1)$-dimensional orthogonal space $(V',Q')$ over $F$ with real signatures $(c+1,0)$ at one archimedean place and $(c-1,2)$ at the other. Then, the associated $(c-1)$-dimensional Shimura variety of Hodge type is compact, and embeds inside the $b$-dimensional Shimura variety associated to the $\Q$-rational $(b+2)$-dimensional quadratic space obtained by treating $V'$ as a $\Q$-vector space, equipped with the quadratic form $\tr_{F/\Q}(Q')$.

\subsection{Strategy of the proof}

The proof of Theorem \ref{main} follows the lines of \cite{charles1} and relies on Arakelov intersection theory on the integral model $\mathcal{M}$ of the GSpin Shimura variety $M$. For every positive integer $m$, the special divisor $\mathcal{Z}(m)$ parameterizes points of $\mathcal{M}$ for which the associated Kuga--Satake abelian variety admits an extra special endomorphism $s$ that satisfies $s\circ s=[m]$, see \S\ref{special divisors}. By the work of Bruinier \cite{bruinier}, this divisor can be endowed with a Green function $\Phi_{m}$ which is constructed using theta lift of non-holomorphic Eisenstein series of negative weight and thus yields an arithmetic divisor $\widehat{\mathcal{Z}}(m)=(\mathcal{Z}(m),\Phi_{m})$ in the first arithmetic Chow group $\widehat{\CH^1}(\mathcal{M})$ of $\mathcal{M}$. By assumption, we have an abelian scheme $\mathcal{A}_\mathcal{Y}\rightarrow \mathcal{Y}=\mathrm{Spec}(\mathcal{O}_K)$ and a map $\iota:\mathcal{Y}\rightarrow \mathcal{M}$. We can express the height $h_{\widehat{\mathcal{Z}}(m)}(\mathcal{Y})$ of $\mathcal{Y}$ with respect to the arithmetic divisor $\widehat{\mathcal{Z}}(m)$ as follows (see \S \ref{arithmetic}): 

\begin{align}\label{mainformula}
h_{\widehat{\cZ}(m)}(\cY)=\sum_{\sigma:K\hookrightarrow \C}\Phi_m(\cY^\sigma)+\sum_{\fP \text{ finite place}} (\cY. \cZ(m))_\fP \log |\cO_K/\fP|.
\end{align}

By definition, $(\cY,\cZ(m))_\fP\neq 0$ if and only if the Kuga--Satake abelian variety at $\cY_{\overline{\fP}}$ admits a special endomorphism $s$ with $s\circ s=[m]$. Therefore, to prove \Cref{main}, it suffices to show that for a fixed finite place $\fP$, for most positive integers $m$, we have  
\begin{align}\label{introeq}
   (\cY. \cZ(m))_\fP=o\left(h_{\widehat{\cZ}(m)}(\cY)-\sum_{\sigma:K\rightarrow \C}\Phi_m(\cY^\sigma)\right).  
\end{align}

Here are the ingredients of the proof. Let $m$ be a positive integer which is represented by the lattice $(L,Q)$.
\begin{enumerate}
\item Starting from an explicit expression of $\Phi_m$ given by Bruinier in \cite[\S 2.2]{bruinier}, we pick out the main term out of archimedean part of \Cref{mainformula}, which is a scalar multiple of $m^{\frac{b}{2}}\log m$, see \Cref{logterm}.

\item To treat the term $h_{\widehat{\cZ}(m)}(\cY)$, we use a theorem of Howard and Madapusi-Pera \cite[Theorem 9.4.1]{howardmadapusi} which asserts that the generating series of $\widehat{\cZ}(m)$ is a component of a vector valued modular form of weight $1+\frac{b}{2}$ with respect to the Weil representation associated to the lattice $(L,Q)$. As a consequence, we get $h_{\widehat{\cZ}(m)}(\cY)=O(m^{\frac{b}{2}})$. This modularity result was previously known over the complex fiber by the work of Borcherds \cite{borcherdszagier} and a cohomological version was given by Kudla--Millson \cite{kudlamillson}.

\item Based on Bruinier's explicit formula, we reduce the estimate of the remaining part of the archimedean term into a problem of counting lattice points with weight functions admitting logarithmic singularities, see \Cref{twoterms}.

\item The treatment of this lattice counting problem in \S 6 is one of the key novelties here compared to the treatment of the archimedean places in the previous works \cite{charles1,st}. We break the sum into two parts; the first part, which consists of lattice points which are not very close to the singularity of the weight function, is treated using the circle method and the rest is controlled by the so-called \emph{diophantine bound}. More precisely, the geometrical meaning of controlling the second part is to show that (away from a small set of $m$) $\cY(\C)$ is not very close to $\cZ(m)(\C)$. Roughly speaking, we prove that if $\cY(\C)$ is too close to too many special divisors with $m$ in a certain range, then $\cY(\C)$ must be close to a special divisor with much smaller $m$. This would violate the diophantine bound deduced from the height formula and the estimates in (1)-(3) above, see \Cref{summary}.
\item Now it remains to treat the finite contribution. This part can be translated into a lattice counting problem on a sequence of lattices $L_n, n\in \Z_{\geq 1}$, where $L_n$ is the lattice of special endomorphisms of the Kuga--Satake abelian variety over $\cY \bmod \fP^n$, see \Cref{loc_int_nb}. As in \cite{st}, we use Grothendieck--Messing theory and Serre--Tate theory to describe the asymptotic behavior of $L_n$. These results give adequate bounds for the main terms. In order to deal with the error terms, we use the diophantine bound (see \Cref{summary}) \textit{for individual $m$ to obtain better bounds on average}. This step is crucial for our proof. Indeed, we illustrate the necessity of using the height bound with an example of a transcendental point of $\cM$ in \S\ref{sec_trans_ex}, where the finite contribution can be arbitrarily large for an infinite sequence of $m$. We also remark that our idea of using the global height bound has other applications. Indeed, our idea is a crucial ingredient in proving intersection-theoretic results in characteristic $p$, as well as in proving the ordinary Hecke-orbit conjecture for GSpin Shimura varieties (see \cite{K3function}). It has also been used to extend the main theorem to the bad reduction situation, see \cite{tayou-bad-reduction}. We note, however, that our theorem applies to K3 surfaces with \emph{potentially} good reduction everywhere. There are very few K3 surfaces with good reduction everywhere, but the condition of potentially good reduction is far less restrictive. Indeed, the moduli spaces of polarized K3 surfaces contains several compact Shimura varieties, whose points parameterize K3 surfaces with potentially good reduction everywhere and our result applies to them. 
\item We briefly describe how we use the diophantine bound for \emph{each} $m$ to obtain stronger bounds on the local contribution from finite places on average. The arguments in (1)-(3) actually prove that the quantity inside the little $o$ in the right hand side of Equation \eqref{introeq} is bounded by $m^{\frac{b}{2}}\log m$, as $m\rightarrow \infty$ (although we prove that it is also $\gg m^{\frac{b}{2}}\log m$ in (4) for most $m$). Since each term $(\cY. \cZ(m))_\fP$ is nonnegative, we have that for any $\fP$, the local contribution $(\cY. \cZ(m))_{\fP}=O(m^{\frac{b}{2}}\log m)$ -- this is the diophantine bound that we refer to. By \Cref{loc_int_nb}, this diophantine bound implies that for $n\gg m^{\frac{b}{2} +\epsilon}, \epsilon>0$, any non-zero special endomorphism $s$ in the lattice $L_n$ must satisfy $s\circ s=[m']$ with  $m'\geq m$. Then a geometry-of-numbers argument suffices to conclude the proof of our theorem. 
\end{enumerate}



\subsection{Organization of the paper}
In \S\ref{gspin} we recall the construction of the GSpin Shimura variety associated to the lattice $(L,Q)$ following \cite{agmp1,agmp2,howardmadapusi}, as well as the construction of its integral model and the construction of the special divisors using the notion of special endomorphisms,  then we give a reformulation of \Cref{main}. In \S\ref{harmonic} we recall how to associate Green functions to the special divisors  on $\mathcal{M}$ and we state Borcherds--Howard--Madapusi-Pera's modularity result from \cite{howardmadapusi}, then we derive consequences on the growth of the global height of an $\mathcal{O}_K$-point on $\mathcal{M}$. In \S\ref{quadraticgeneral} we collect some general results on quadratic forms which will be used in the following sections. In \S\ref{mainproof}, we give a first step estimate on the growth of the archimedean terms and we derive the diophantine bounds on archimedean and non-archimedean contributions. The second step in estimating the archimedean contributions is performed in \S\ref{sec_arch}, while the non-archimedean contributions are treated in \S\ref{sec_finite}. In \S\ref{finalmvt} we put together all the ingredients to prove \Cref{main_sp_end} and hence \Cref{main}. Finally, in \S\ref{applicationsch} we prove the applications to K3 surfaces. Subsequently, we prove the applications to Kuga--Satake abelian varieties and abelian varieties parametrized by unitary Shimura varieties. 
 
\subsection{Acknowledgements}
We are very grateful to Fabrizio Andreatta, Olivier Benoist, Laurent Clozel, Edgar Costa, Quentin Guignard, Jonathan Hanke, Benjamin Howard, Christian Liedtke, Yifeng Liu, Chao Li, Davesh Maulik, and Jacob Tsimerman for many helpful conversations. Part of this work has been done in S.T's PhD thesis and he is particularly grateful to François Charles. We're also very grateful to the referee for several valuable suggestions that have improved this paper. 

A.N.S. is partially supported by the NSF grant DMS-2100436. A.S. is supported by an NSERC Discovery grant and a Sloan fellowship. Y.T. is partially supported by the NSF grant DMS-1801237. S.T has received funding from the European Research Council (ERC) under the European Union’s Horizon 2020 research and innovation programme (grant agreement No 715747).
\subsection{Notations} If $f,g:\N\rightarrow \R$ are  real functions and $g$ does not vanish, then: 
\begin{enumerate}
\item $f=O(g)$, or $f\ll g$, if there exists an integer $n_0\in \N$, a positive constant $C_{0}>0$ such that $$\forall n\geq n_0,\, |f(n)|\leq C_{0}|g_{}(n)|.$$
\item $f\asymp h$ if $f=O(h)$ and $h=O(f)$.
\item $f=o(g)$ if for every $\epsilon>0$, there exists $n_\epsilon$ such that for every $n\geq n_\epsilon$
$$|f(n)|\leq \epsilon |g(n)|.$$
\item For $p$ a prime number, $\val_p$ denotes  the $p$-adic valuation on $\Q$.
\item For $s\in \C$, $\mathrm{Re}(s)$ is the real part of $s$. 
\end{enumerate}

\section{The GSpin Shimura varieties and their special divisors}\label{gspin}
Let $(L,Q)$ be an integral quadratic even lattice of signature $(b,2)$, $b\geq1$, with associated bilinear form defined by $$(x.y)=Q(x+y)-Q(x)-Q(y),$$ for $x,y\in L$. Let $V:=L\otimes_{\Z}\Q$ and assume that $L$ is a maximal lattice in $V$ over which $Q$ is $\Z$-valued. We recall in this section the theory of GSpin Shimura varieties associated with $(L,Q)$. Our main references are \cite[Section 2]{agmp1}, \cite[Section 4]{agmp2} and \cite[Section 3]{madapusiintegral}.
\subsection{The GSpin Shimura variety}\label{gspin_Q}
For a commutative ring $R$, let $L_{R}$ denote $L\otimes_{\Z}R$ and the quadratic form $Q$ on $L$ induces a quadratic form $Q$ on $L_R$. The \emph{Clifford algebra} $C(L_R)$ of $(L_R, Q)$ is the $R$-algebra defined as the quotient of the tensor algebra $\bigotimes L_R$ by the ideal generated by $\{(x\otimes x)-Q(x),\, x\in L_R\}$. It has a $\Z/2\Z$ grading $C(L_R)=C(L_R)^+\oplus C(L_R)^-$ induced by the grading on $\bigotimes L_R$. When $R$ is a $\Q$-algebra, we also denote $C(L_R)$ (resp. $C^\pm(L_R)$) by $C(V_R)$ (resp. $C^\pm(V_R)$) and note that $C(L)$ is a lattice in $C(V)$.
 
Let $G:=\mathrm{GSpin}(V)$ be the group of spinor similitudes of $V$. 
It is the reductive algebraic group over $\Q$ such that $$G(R)=\{g\in C^{+}(V_R)^\times,\, gV_R g^{-1}=V_{R}\}$$ for any $\Q$-algebra $R$. We denote by $\nu:G\rightarrow \mathbb{G}_m$ the spinor similitude factor as defined in \cite[Section 3]{bass}. The group $G$ acts on $V$ via $g\bullet v=gvg^{-1}$ for $v\in V_R$ and $g\in G(R)$. Moreover, there is an exact sequence of algebraic groups $$1\rightarrow \mathbb{G}_{m}\rightarrow G\xrightarrow{g\mapsto g\bullet}\mathrm{SO}(V)\rightarrow 1.$$ 

Let $D_L$ be the period domain associated to $(L,Q)$ defined by\footnote{Here we pick $z\in V_{\C}$ such that $[z]=x$ and $(\overline{x}.x)>0$ means that $(\overline{z},z)>0$ and $(x.x)=0$ means $(z.z)=0$; both conditions are independent of the choice of $z$.} $$D_L=\{x \in \mathbb{P}(V_{\C})\mid (\overline{x}.x)<0,(x.x)=0\}.$$
It is a hermitian symmetric domain and the group $G(\R)$ acts transitively on $D_L$. As in \cite[\S 4.1]{agmp2}, $(G,D_L)$ defines a Shimura datum as follows: for any class $[z]\in D_{L}$ with $z\in V_\C$,  there is a morphism of algebraic groups over $\R$ $$h_{[z]}:\mathbb{S}=\mathrm{Res}_{\C/\R}\mathbb{G}_m\rightarrow G_{\R}$$ such that the induced Hodge decomposition on $V_\C$ is given by $$V^{1,-1}=\C z, V^{-1,1}=\C \overline{z}, V_{\C}^{0,0}=(\C z\oplus \C \overline{z})^{\bot}.$$
Indeed, choose a representative $z=u+iw$ where $u,w\in V_{\R}$ are orthogonal and $Q(u)=Q(w)=-1$, then $h_{[z]}$ is the morphism such that $h_{[z]}(i)=uw\in G(\R)\subset C^{+}(V_\R)^{\times}$. Hence $D_L$ is identified with a $G(\R)$-conjugacy class in $\mathrm{Hom}(\mathrm{Res}_{\C/\R}\mathbb{G}_m,G_{\R})$. The reflex field of $(G,D_L)$ is equal to $\Q$ by \cite[Appendix 1]{andretate}.  

Let  $\K\subset G(\mathbb{A}_{f})$ be the compact open subgroup $$\K=\G(\A_f)\cap C(\widehat{L})^{\times},$$
where $\widehat{L}=L\otimes_{\Z} \widehat{\Z}$. By \cite[Lemma 2.6]{madapusiintegral}, the image of $\K$ in $\mathrm{SO}(\widehat{L})$ is the subgroup of elements acting trivially on $L^{\vee}/L$, where $L^{\vee}$ is the dual lattice of $L$ defined by $$L^{\vee}:=\{x\in V\mid \forall y\in L,\, (x.y)\in \Z\}.$$
By the theory of canonical models, we get a $b$-dimensional Deligne--Mumford stack $M$ over $\Q$, the \emph{GSpin Shimura variety associated with $L$}, such that   
$$M(\C)=G(\Q)\backslash D_L\times G(\mathbb{A}_f)/\K.$$

\subsection{The Kuga--Satake construction and K3 type motives in characteristic $0$}\label{sec_KS}
The Kuga--Satake construction was first considered in \cite{kugasatake} and later in \cite{deligne} and \cite{deligneshimura2}. We follow here the exposition of \cite[Section 3]{madapusiintegral}.
 
Let $G\rightarrow \mathrm{Aut}(N)$ be an algebraic representation of $G$ on a $\Q$-vector space $N$, and let $N_{\widehat{\Z}}\subset N_{\A_f}$ be a $\K$-stable lattice. Then one can construct a local system $\mathbf{N}_{B}$ on $M(\C)$ whose fiber at a point $[[z],g]$ is identified with $N\cap g(N_{\widehat{\Z}})$. The corresponding vector bundle $\mathbf{N}_{dR,M(\C)}=\mathcal{O}_{M(\C)}\otimes \mathbf{N}_{B}$ is equipped with a holomorphic filtration $\mathcal{F}^{\bullet}\mathbf{N}_{dR,M(\C)}$ which at every point $[[z],g]$ equips the fiber with the Hodge structure determined by the cocharacter $h_{[z]}$. Hence we obtain a functor 
\begin{align}\label{functor}
(N,N_{\widehat{\Z}})\mapsto \left(\mathbf{N}_{B},\mathcal{F}^{\bullet}\mathbf{N}_{dR,M(\C))}\right)
\end{align}
from the category of algebraic $\Q$-representations of $G$ with a $\K$-stable lattice to variations of $\Z$-Hodge structures over $M(\C)$.
Applying this functor to $(V,\widehat{L})$, we obtain a variation of $\Z$-Hodge structures $\{\mathbf{V}_{B},\mathcal{F}^{\bullet}\mathbf{V}_{dR,M(\C)}\}$ of weight $0$ over $M(\C)$. The quadratic form $Q$ gives a polarization on $(\mathbf{V}_{B},\mathcal{F}^{\bullet}\mathbf{V}_{dR,M(\C)})$ and hence by \cite[1.1.15]{deligneshimura2}, all $(\mathbf{N}_{B},\mathcal{F}^{\bullet}\mathbf{N}_{dR,M(\C))})$ are polarizable.

Also, if we denote by $H$ the representation of the group $G$ on $C(V)$ by left multiplication and $H_{\widehat{\Z}}=C(L)_{\widehat{\Z}}$, then applying the functor (\ref{functor}) to the pair $(H,H_{\widehat{\Z}})$, we obtain a polarizable variation of $\Z$-Hodge structures $(\mathbf{H}_{B},\mathbf{H}_{dR,M(\C)})$ of type $(-1,0),(0,-1)$ with a right $C(V)$-action.
Therefore, there is a family of abelian schemes $A^\univ\rightarrow M$ of relative dimension $2^{b+1}$, the \emph{Kuga--Satake abelian scheme}, such that the homology of the family $A^{\univ,an}(\C)\rightarrow M^{an}(\C)$ is precisely $(\mathbf{H}_{B},\mathbf{H}_{dR,M(\C)})$. It is equipped with a right $C(L)$-action and a compatible $\Z/2\Z$-grading: $A^\univ=A^{\univ,+}\times A^{\univ,-}$, see \cite[3.5--3,7, 3.10]{madapusiintegral}.\footnote{Here we follow the convention in \cite{agmp2}, where $H$ is the homology of $A^\univ$. In \cite{madapusiintegral}, $H$ is the cohomology of $A^\univ$.}

Using $A^\univ$, one descends $\mathbf{H}_{dR,M(\C)}$ to a filtered vector bundle with an integrable connection $(\mathbf{H}_{dR},\mathcal{F}^{\bullet}\mathbf{H}_{dR})$ over $M$ as the first relative de Rham homology with the Gauss--Manin connection (\cite[3.10]{madapusiintegral}). For any prime $\ell$, the $\ell$-adic sheaf $\Z_\ell\otimes \mathbf{H}_B$ over $M(\C)$ descends also canonically to an $\ell$-adic étale sheaf $\mathbf{H}_{\ell,\textrm{ét}}$ over $M$, which is canonically isomorphic to the $\ell$-adic Tate module of $A^\univ$ (\cite[3.13]{madapusiintegral}). Moreover, by Deligne's theory of absolute Hodge cycles, one descends $\mathbf{V}_{dR,M(\C)}$ and $\Z_\ell \otimes \mathbf{V}_B$ to $(\mathbf{V}_{dR},\mathcal{F}^{\bullet}\mathbf{V}_{dR})$ and $\mathbf{V}_{\ell,\textrm{ét}}$ over $M$ (\cite[3.4, 3.10--3.12]{madapusiintegral}). More precisely, an idempotent $$\pi=(\pi_{B,\Q}, \pi_{dR, \Q}, \pi_{\ell,\Q})\in \End(\End(\bfH_{B}\otimes \Q))\times \End(\End(\bfH_{dR})) \times \End(\End(\bfH_{\ell, \textrm{\'et}}\otimes \Q_\ell))$$ is constructed in \emph{loc.~cit.} such that the fiber of $\pi$ at each closed point in $M$ is an absolute Hodge cycle and $(\bfV_{B}\otimes \Q, \bfV_{dR}, \bfV_{\ell, \textrm{\'et}}\otimes \Q_\ell)$ is the image of $\pi$. In particular, $(\bfV_{B}\otimes \Q, \bfV_{dR}, \bfV_{\ell, \textrm{\'et}}\otimes \Q_\ell)$ is a family of absolute Hodge motives over $M$ and we call each fiber a \emph{K3 type motive}. (For a reference on absolute Hodge motives, we refer to \cite[IV]{DMOS}.)

On the other hand, let $\End_{C(V)}(H)$ denote the endomorphism ring of $H$ as a $C(V)$-module with right $C(V)$-action. Then the action of $V$ on $H$ as left multiplication induces a $G$-equivariant embedding $V\hookrightarrow \mathrm{End}_{C(V)}(H)$, which maps $\widehat{L}\hookrightarrow H_{\widehat{\Z}}$. The functoriality of (\ref{functor}) induces embeddings \begin{align*}
\mathbf{V}_{B}\hookrightarrow \mathrm{End}_{C(L)}(\mathbf{H}_{B})\quad\textrm{and}\quad \mathbf{V}_{dR,M(\C)}\hookrightarrow \mathrm{End}_{C(V)}(\mathbf{H}_{dR,M(\C)}),
\end{align*}
the latter being compatible with filtration. By \cite[1.2, 1.4, 3.11]{madapusiintegral}, these embeddings are the same as the one induced by $\pi$ above with the natural forgetful map $\End_{C(V)}(\bfH)\rightarrow \End(\bfH)$. In particular, the embedding $\mathbf{V}_{dR}\hookrightarrow \End_{C(V)}(\bfH_{dR})$ is compatible with filtrations and connection and $V_{\ell, \textrm{\'et}}\hookrightarrow \End_{C(V)}(\bfH_{\ell, \textrm{\'et}})$ as $\Z_\ell$-lisse sheaves with compatible Galois action on each fiber.\footnote{The compatibility of $\Z_\ell$-structure can be deduced from Artin's comparison theorem and that we have the embedding $\mathbf{V}_{B}\hookrightarrow \mathrm{End}_{C(L)}(\mathbf{H}_{B})$ as $\Z$-local system over $M(\C)$.} 

Moreover, there is a canonical quadratic form $\mathbf{Q}:\mathbf{V}_{dR}\rightarrow \mathcal{O}_M$ given on sections by $v\circ v=\bfQ(v)\cdot\mathrm{Id}$ where the composition takes places in $\mathrm{End}_{C(V)}(\mathbf{H}_{dR})$. Similarly,
there is also a canonical quadratic form on $\mathbf{V}_{\ell,\textrm{ét}}$ induced by composition in $\End_{C(V)}(\bfH_{\ell, \textrm{\'et}})$ and valued in the constant sheaf $\underline{\Z_\ell}$. 

\subsection{Special divisors on $M$ over $\Q$}\label{special divisors}For any vector $\lambda\in L_{\R}$ such that $Q(\lambda)>0$, let $\lambda^{\bot}$  be the set of elements of $D_{L}$ orthogonal to $\lambda$. Let $\beta\in L^{\vee}/L$ and $m\in Q(\beta)+\Z$ with $m>0$ and define the complex orbifold  
\begin{align*}
Z(\beta,m)(\C):=\bigsqcup_{g\in G(\Q)\backslash G(\mathbb{A}_f)/\K}\Gamma_{g}\backslash\left(\bigsqcup_{\lambda\in \beta_g+L_g,\, Q(\lambda)=m}\lambda^{\bot}\right)
\end{align*} 
where $\Gamma_g=G(\Q)\cap g\K g^{-1}$, $L_g\subset V$ is the lattice determined by $\widehat{L_{g}}=g\bullet \widehat{L}$ and $\beta_g=g\bullet\beta\in L_{g}^{\vee}/L_g$.   
Then $Z(\beta,m)(\C)$ is the set of complex points of a disjoint union of Shimura varieties associated with orthogonal lattices of signature $(b-1,2)$ and it admits a canonical model $Z(\beta,m)$ over $\Q$ for $b\geq 2$.\footnote{When $b=1$, $Z(\beta,m)$ is $0$-dimensional and it is still naturally a Deligne--Mumford stack over $\Q$.} The natural map $Z(\beta,m)(\C)\rightarrow M(\C)$ descends to a finite unramified morphism $Z(\beta,m)\rightarrow M$. \'Etally locally on $Z(\beta,m)$, this map is a closed immersion defined by a single equation and hence its scheme theoretic image gives an effective Cartier divisor on $M$, which we will also denote by $Z(\beta,m)$.\footnote{We need to take the scheme theoretic image since the map $Z(\beta, m)\rightarrow M$ is not a closed immersion. See for instance \cite[Ch.~5, p.119]{bruinier} and \cite[Remark 9.3]{KRY04}.}

\subsection{Integral models}\label{integral}
We recall the construction of an integral model of $M$ from \cite[4.4]{agmp2}, see also \cite[6.2]{howardmadapusi} and the original work of Kisin \cite{kisin} and Madapusi Pera \cite{madapusiintegral}. Let $p$ be a prime number. We say that $L$ is \emph{almost self-dual} at $p$ if either $L$ is self-dual at $p$ or $p=2$, $\dim_{\Q}(V)$ is odd and $|L^{\vee}/L|$ is not divisible by $4$. The following proposition is parts of \cite[Proposition 4.4.1, Theorem 4.4.6]{agmp2}
and \cite[Remark 6.3.1]{howardmadapusi}. 

\begin{proposition}\label{arithmeticstack}
There exists a flat, normal Deligne--Mumford $\Z$-stack $\mathcal{M}$ with the following properties.
\begin{enumerate}
\item $\mathcal{M}_{\Z_{(p)}}$ is smooth over $\Z_{(p)}$ if $L$ is almost self-dual at $p$; 
\item the Kuga--Satake abelian scheme $A^\univ\rightarrow M$ extends to an abelian scheme $\mathcal{A}^\univ\rightarrow \mathcal{M}$ and the $C(L)$-action on $A^\univ$ also extends to a $C(L)$-action on $\cA^\univ$; 
\item the line bundle $\cF^{1}V_{dR}$ extends canonically to a line bundle $\boldsymbol{\omega}$ over $\mathcal{M}$.
\item the extension property: for $E/\Q_p$ finite, $t\in M(E)$ such that $A^\univ_t$ has potentially good reduction over $\cO_E$, then $t$ extends to a map $\spec(\cO_E)\rightarrow \cM$.
\end{enumerate}
\end{proposition}

For $p$ such that $L$ is self-dual at $p$, we now discuss the extensions of $\bfV_{dR}, \bfV_{\ell,\textrm{\'et}}, \ell\neq p$ over $\cM_{\Z_{(p)}}$ and recall the construction of $\bfV_{cris}$. For the ease of reading, we will denote the extensions by the same notation. We will use these notions to provide an \emph{ad hoc} definition of the reduction of the K3 type motives defined in \S\ref{sec_KS}.

By (2) in \Cref{arithmeticstack}, there are natural extensions of $\bfH_{dR}, \bfH_{\ell, \textrm{\'et}}$ as the first relative de Rham homology and the $\Z_\ell$-Tate module of $\cA^\univ$ and we define $\bfH_{cris}$ to be the first relative crystalline homology $$\mathrm{Hom}\left(R^{1}\pi_{cris,*}\mathcal{O}^{cris}_{\mathcal{A}_{\mathbb{F}_p}/\Z_p},\mathcal{O}^{cris}_{\mathcal{M}_{\mathbb{F}_{p}/\Z_p}}\right).$$ 

By \cite[Remark 4.2.3]{agmp2}, there exists a canonical extension of $\bfV_{\ell, \textrm{\'et}}\hookrightarrow \End_{C(L)}(\bfH_{\ell, \textrm{\'et}})$ over $\cM_{\Z_{(p)}}$. Note that $L_{\Z_{(p)}}$ is a $\Z_{(p)}$-representation of $\mathrm{GSpin}(L_{\Z_{(p)}},Q)$, then by \cite[Propositions 4.2.4, 4.2.5]{agmp2}, there is a vector bundle with integrable connection $\bfV_{dR}$ over $\cM_{\Z_{(p)}}$ and a canonical embedding into $\End_{C(L)}(\bfH_{dR})$ extending their counterparts over $M$; moreover, there is an $F$-crystal $\bfV_{cris}$ with a canonical embedding into $\mathrm{End}_{C(L)}(\mathbf{H}_{cris})$. Both embeddings realize $\bfV_{dR}$ and $\bfV_{cris}$ as local direct summands of $\End_{C(L)}(\bfH_{dR})$ and $\mathrm{End}_{C(L)}(\mathbf{H}_{cris})$ and these two embeddings are compatible with the canonical crystalline-de Rham comparison.

For a point $x\in \cM(\bF_q)$, where $q$ is a power of $p$, we consider the fiber of $(\bfV_{\ell, \textrm{\'et}}, \bfV_{dR},\bfV_{cris})$ at $x$ of the \emph{ad hoc} motive attached to $x$, where $q$-Frobenius $\Frob_x$ acts on $\bfV_{\ell, \textrm{\'et},x}$ and semi-linear crystalline Frobenius $\varphi_x$ acts on $\bfV_{cris,x}$. Although these realizations are not as closely related as the analogous characteristic $0$ situation, there is a good notion of algebraic cycles in $(\bfV_{\ell, \textrm{\'et}}, \bfV_{dR},\bfV_{cris})$, namely the special endomorphisms of the Kuga--Satake abelian varieties discussed in the following subsection.

\subsection{Special endomorphisms and integral models of special divisors}\label{special}
We recall the definition of special endomorphisms from \cite[\S\S4.3,4.5]{agmp2}. For an $\cM$-scheme $S$, we use $A_S$ to denote $\cA^\univ_S$, the pull-back of the universal Kuga--Satake abelian scheme to $S$. 

\begin{definition}\label{def_sp_end}
An endomorphism $v\in \End_{C(L)}(A_S)$ is \emph{special} if 
\begin{enumerate}
    \item for prime $p$ such that $L$ is self-dual at $p$, all homological realizations of $v$ lie in the image of $\bfV_{?}\hookrightarrow \End_{C(L)}(\bfH_{?})$ given in \S\S\ref{sec_KS},\ref{integral};\footnote{More precisely, if $p$ is invertible in $S$, we take $?=B, dR, (\ell,\textrm{\'et})$ and by the theory of absolute Hodge cycles, it is enough to just consider $?=B$; otherwise, we take $?=B, dR, cris, (\ell,\textrm{\'et}), \ell\neq p$ (and we drop $?=B$ if $S_{\Q}=\emptyset$).} and
    \item for $p$ such that $L$ is not self-dual, after choosing an auxiliary maximal lattice $L^\diamond$ of signature $(b^\diamond, 2)$ which is self-dual at $p$ and admits an isometric embedding $L\inj L^\diamond$, the image of $v$ under the canonical embedding $\End_{C(L)}(A_S)\inj \End_{C(L^\diamond)}(A^\diamond_S)$ has all its homological realizations lying in the image of $\bfV_{?}^\diamond\hookrightarrow \End_{C(L^\diamond)}(\bfH_{?}^\diamond)$.\footnote{Here $(-)^\diamond$ denotes the object defined using $L^\diamond$ in previous sections. The existence of the canonical embedding $\End_{C(L)}(A_S)\inj \End_{C(L^\diamond)}(A^\diamond_S)$ follows from \cite[Proposition 4.4.7 (2)]{agmp2}. By \cite[Proposition 4.5.1]{agmp2}, this definition is independent of the choice of $L^\diamond$.}
\end{enumerate}
We use $V(A_S)$ to denote the $\Z$-module of special endomorphisms of $A_S$.
\end{definition}

By \cite[Prop.4.5.4]{agmp2}, there is a positive definite quadratic form $Q:V(A_S)\rightarrow \Z$ such that  for each $v\in V(A_S)$, we have $v\circ v=Q(v)\cdot\mathrm{Id}_{A_S}$.\footnote{Here we use the same letter $Q$ for this quadratic form since if every point in $S$ is the reduction of some characteristic $0$ point in $S$, then this quadratic form is the restriction of $(\bfV_B,\bfQ)$ via the canonical embedding $V(A_S)\hookrightarrow \bfV_{B,S}$ (recall that $\bfQ$ is induced by $(L,Q)$ in \S\ref{sec_KS}).}

\begin{definition}\label{def_sp_pdiv}
When $S$ is a $\cM_{\Z_{(p)}}$-scheme, $v$ is a \emph{special endomorphism} of the $p$-divisible group $A_S[p^\infty]$ if $v\in \End_{C(L)}(A_S[p^\infty])$ and the crystalline realization of $x$ (resp. image of $x$ under the canonical embedding $\End_{C(L)}(A_S[p^\infty])\inj \End_{C(L^\diamond)}(A^\diamond_S[p^\infty])$) lies in $\bfV_{cris}$ (resp. $\bfV^\diamond_{cris}$) if $L$ is self-dual at $p$ (resp. otherwise).\footnote{In \cite[\S 4.5]{agmp2}, the definition of a special endomorphism of a $p$-divisible group also contains a condition on its $p$-adic \'etale realization over $S[p^{-1}]$. The proof of \cite[Lemma 5.13]{madapusiintegral} shows that this extra condition is implied by the crystalline condition.}
\end{definition}

For an odd prime $p$ such that $L$ is self-dual at $p$, for a point $x\in \cM(\bF_{p^r})$ by \cite[Theorem 6.4]{madapusiperatate},\footnote{Assumption 6.2 in \cite{madapusiperatate} follows immediately from \cite[Corollary (2.3.1)]{kisin17}} we have isometries
\[V(\cA^\univ_x)\otimes \Q_\ell \cong \lim_{n\rightarrow\infty}\bfV_{\ell,\textrm{\'et},x}^{\Frob_x^n=1},\ell\neq p, \quad V(\cA^\univ_x)\otimes \Q_p \cong \lim_{n\rightarrow\infty}(\Q_{p^{rn}}\otimes\bfV_{cris,x})^{\varphi_x=1}.\]
Therefore, we view special endomorphisms of $\cA^\univ_x$ as the algebraic cycles of the \emph{ad hoc} motive $(\bfV_{\ell, \textrm{\'et},x}, \bfV_{dR,x},\bfV_{cris,x})$.

For $m\in \Z_{>0}$, the \emph{special divisor} $\cZ(m)$ is defined as the Deligne--Mumford stack over $\cM$ with functor of points $\cZ(m)(S)=\{v\in V(A_S) \mid Q(v)=m\}$ for any $\cM$-scheme $S$. More generally, in \cite[\S 4.5]{agmp2}, for $\beta\in L^\vee/L, m\in Q(\beta)+\Z, m>0$, there is also a special cycle $\cZ(\beta,m)$ defined as a Deligne--Mumford stack over $\cM$ parameterizing points with certain special quasi-endomorphisms and $\cZ(m)=\cZ(0,m)$. By \cite[Proposition 4.5.8]{agmp2}, the generic fiber $\cZ(\beta,m)_{\Q}$ is equal to the divisor $Z(\beta,m)$ defined in \S\ref{special divisors}. Moreover,  étale locally on the source, $\cZ(\beta,m)$ is an effective Cartier divisor on $\cM$ and we will use the same notation for the Cartier divisor on $\cM$ defined by \'etale descent.


\subsection{Reformulation of \Cref{main}}
Using the notion of special endomorphisms, \Cref{main} is a direct consequence of the following theorem.
\begin{theorem}\label{main_sp_end}
Assume that $b\geq 3$. Let $K$ be a number field and let $D\in \Z_{>0}$ be a fixed integer represented by $(L,Q)$. Let $\cY\in \mathcal{M}(\mathcal{O}_K)$ and assume that $\cY_K\in M(K)$ is Hodge-generic. Then there are infinitely many places $\fP$ of $K$ such that $\cY_{\overline{\fP}}$ lies in the image of $\cZ(Dm^2)\rightarrow \cM$ for some $m\in \Z_{>0}$.\footnote{Here $m$ may vary as $\fP$ varies.} Equivalently, for a Kuga--Satake abelian variety $\cA$ over $\cO_K$ parameterized by $\cM$ such that $\cA_{\overline{K}}$ does not have any special endomorphisms, there are infinitely many $\fP$ such that $\cA_{\overline{\fP}}$ admits a special endomorphism $v$ such that $v\circ v=[Dm^2]$ for some $m\in \Z_{>0}$.
\end{theorem}


\section{The global height}\label{harmonic}

In this section, we begin the proof of \Cref{main_sp_end} (and in particular \Cref{main}) by studying the height of a given $\cO_K$-point with respect to a sequence of arithmetic special divisors. In \S \ref{arithmetic}, we follow \cite{bruinier}, \cite{borcherds} and endow the special divisors $\cZ(m)$, $m\in \Z_{>0}$ (defined in \S\S\ref{special divisors}, \ref{special}) with Green functions $\Phi_m$, thereby bestowing on them the structure of arithmetic divisors. In \S \ref{howardmadapbor} we recall the modularity theorem of the generating series of arithmetic special divisors $(\cZ(m),\Phi_m)$ proved by Howard--Madapusi-Pera \cite{howardmadapusi}, and in \S \ref{Eisenstein} we use this to deduce asymptotic estimates for the global height 
$h_{\widehat{\mathcal{Z}}(m)}(\mathcal{Y})$. 



For simplicity, we assume that $b\geq 3$ as in \cite{bruinier} and we refer the interested reader to \cite{bruinierintegrals,bruinierfunke,bruinieryang} for related results without this assumption.
Note that our quadratic form $Q$ differs from the one in \cite{bruinier,bruinierintegrals} by a factor of $-1$ and hence we shall replace the Weil representation there by its dual; the rest remains the same, namely we work with the same space of modular forms, harmonic Maass forms, and the same Eisenstein series.
    
\subsection{Arithmetic special divisors and heights.}\label{arithmetic}
Let $\rho_L:\mathrm{Mp}_{2}(\mathbb{Z})\rightarrow \mathrm{Aut_{\C}}(\C[L^\vee/L])$ denote the unitary Weil representation, where $\mathrm{Mp}_{2}(\mathbb{Z})$ is the metaplectic double cover of $\mathrm{SL}_{2}(\Z)$, see, for instance, \cite[Section 1.1]{bruinier}. 
Let $k=1+\frac{b}{2}$ and let $\mathrm{H}_{2-k}(\rho_{L}^\vee)$ be the $\C$-vector space of vector valued harmonic weak Maass forms of weight $2-k$ with respect to the dual $\rho_L^\vee$ of the Weil representation as defined in \cite[3.1]{bruinieryang}. 

For $\beta\in L^\vee/L$, $m\in \Z+Q(\beta)$ with $m>0$, let $F_{\beta,m} \in\mathrm{H}_{2-k}(\rho_{L}^\vee)$ denote the Hejhal--Poincaré harmonic Maass form defined in \cite[Def.~1.8]{bruinier}. In fact $F_{\beta,m}(\tau):=F_{\beta,-m}(\tau, \frac{1}{2}+\frac{b}{4})$ in \emph{loc.~cit.}, where $\tau$ lies in the Poincaré upper half plane $\bH$. Let $\Phi_{\beta,m}$ denote the regularized theta lifting of $F_{\beta,m}$ in the sense of Borcherds, see \cite[(2.16)]{bruinier} and \cite[\S 5.2]{bruinierfunke}. By \cite[\S 6, Thm.~13.3]{borcherds} and \cite[\S 2.2, eqn.~(3.40), Thm.~3.16]{bruinier}, $\Phi_{\beta,m}$ is a Green function\footnote{\emph{A priori}, the Green function is defined over $D_L$, but it descends to $M(\C)$; we use the same notation for both functions on $D_L$ and on $M(\C)$.} for the divisor $\cZ(\beta,m)$ and we use $\widehat{\cZ}(\beta,m)$ to denote the arithmetic divisor $(\cZ(\beta,m), \Phi_{\beta,m})$. Our main focus is the case when $\beta=0$ and we set $\Phi_m:=\Phi_{0,m}, \widehat{\cZ}(m):=\widehat{\cZ}(0,m)$ for $m\in\Z_{>0}$.

Let $\widehat{\CH}^1(\cM)_{\Q}$ denote the first arithmetic Chow group of Gillet--Soul\'e \cite{gilletsoule} as defined in \cite[\S 4.1]{agmp1}. Since $\cM$ is a normal Deligne--Mumford stack, we have a natural isomorphism, as in \cite[III.4]{soule}, $$\widehat{\mathrm{Pic}}(\mathcal{M})_{\Q}\otimes \Q\xrightarrow{\sim} \widehat{\mathrm{CH}^{1}}(\mathcal{M})_{\Q},$$ where $\widehat{\Pic}(\cM)$ denotes the group of isomorphism classes of metrized line bundles and $\widehat{\Pic}(\cM)_{\Q}:=\widehat{\Pic}(\cM)\otimes \Q$, see \cite[\S 5.1]{agmp1} for more details. Since $\cZ(\beta,m)$ is (étale locally) Cartier, then we view $\widehat{\cZ}(\beta,m)\in \widehat{\Pic}(\cM)_{\Q}$.

Moreover, the line bundle $\bomg$ from \Cref{arithmeticstack} is endowed with the Petersson metric defined as follows: the fiber of $\bomg$ at a complex point $[[z],g]\in M(\C)$ is identified with the isotropic line $\C z\subset V_\C$, then we set  $||z||^2=-\frac{(z.\overline{z})}{4\pi e^{\gamma}}$, where $\gamma=-\Gamma'(1)$ is the Euler--Mascheroni constant. Hence we get a metrized line bundle $\overline{\bomg}\in\widehat{\mathrm{Pic}}(\mathcal{M})$.

Recall that we have a map $\mathcal{Y} \rightarrow \cM $, where $\mathcal{Y} = \spec \cO_K$. We now define the notion of the height of $\mathcal{M}$ with respect to the arithmetic divisors $\widehat{\cZ}(m)$ and $\overline{\bomg}$. As in \cite[\S\S 5.1,5.2]{agmp1}, \cite[\S 6.4]{agmp2}, the height $h_{\widehat{\mathcal{Z}}(m)}(\mathcal{Y})$ of $\mathcal{Y}$ with respect to $\widehat{\mathcal{Z}}(m)$ (resp.~$\overline{\bomg}$) is defined as the image of $\widehat{\mathcal{Z}}(m)$ (resp.~$\overline{\bomg}$) under the composition 
$$\widehat{\mathrm{CH}^{1}}(\mathcal{M})_{\Q}\cong \widehat{\Pic}(\cM)_{\Q}\rightarrow \widehat{\Pic}(\mathcal{Y})_{\Q}\xrightarrow{\widehat{\deg}} \R,$$
where the middle map is the pull-back of metrized line bundles and the arithmetic degree map $\widehat{\deg}$ is the extension over $\Q$ of the one defined in \cite[6.4]{agmp2}. 

Since $\cY$ and $\cZ(m)$ intersect properly (recall that we assume $\cY_K$ is Hodge-generic), we have the following description of $h_{\widehat{\cZ}(m)}(\cY)$. Let $\cA$ denote $\cA^\univ_{\cY}$, where $\cA^\univ$ is the Kuga--Satake abelian scheme over $\cM$. 
Using the moduli definition of $\cZ(m)$ in \S\ref{special}, the $\cY$-scheme $\cY\times_{\cM}\cZ(m)$ is given by
\[\cY\times_{\cM}\cZ(m)(S)=\{v\in V(\cA_S)\mid v\circ v=[m]\},\]
for any $\cY$-scheme $S$.
Via the natural map $\cY\times_\cM \cZ(m)\rightarrow \cY$ and using \'etale descent, we view $\cY\times_{\cM} \cZ(m)$ as a $\Q$-Cartier divisor on $\cY$. Therefore,
\begin{align}\label{intersectionformula2}
h_{\widehat{\mathcal{Z}}(m)}(\mathcal{Y})=\sum_{\sigma:K\hookrightarrow\C}\Phi_{m}(\cY^\sigma)+\sum_{\mathfrak{P}}(\cY.\cZ(m))_{\fP}\log|\mathcal{O}_{K}/\mathfrak{P}|,
\end{align}
where for $\sigma:K\hookrightarrow \C$, we use $\cY^\sigma$ to denote the point in $M(\C)$ induced by $\spec(\C)\xrightarrow{\sigma} \spec{\cO_K}\xrightarrow{\cY} \cM$ and if we denote by $\cO_{\cY\times_{\cM}\cZ(m),v}$ the \'etale local ring of $\cY\times_{\cM}\cZ(m)$ at $v$,
\begin{align}\label{int_formula_finite1}
(\cY.\cZ(m))_{\fP}=\sum_{v\in\cY\times_{\cM}\mathcal{Z}(m)(\overline{\mathbb{F}}_\mathfrak{P})}\length(\cO_{\cY\times_{\cM}\cZ(m),v}),
\end{align}
where $\bF_\fP$ denotes the residue field of $\fP$. 

\subsection{Howard--Madapusi-Pera--Borcherds' modularity theorem}\label{howardmadapbor}
Let $M_{1+\frac{b}{2}}(\rho_L)$ denote the $\C$-vector space of $\C[L^\vee/L]$-valued modular forms of weight $1+\frac{b}{2}$ with respect to $\rho_L$ (see see \cite[Definition 1.2]{bruinier}). Let $(\fe_\beta)_{\beta\in L^\vee/L}$ denote the standard basis of $\C[L^\vee/L]$.
\begin{theorem}[{\cite[Theorem 9.4.1]{howardmadapusi}}]\label{modularity}
Assume $b\geq 3$ and let $q=e^{2\pi i \tau}$. The formal generating series 
$$\widehat{\Phi}_L=\overline{\bomg}^\vee\fe_0+\sum_{\underset{m> 0, m\in Q(\beta)+\Z}{\beta\in L^{\vee}/L}}\widehat{\mathcal{Z}}(\beta,m)\cdot q^m \fe_\beta$$
is an element of $M_{1+\frac{b}{2}}(\rho_L)\otimes \widehat{\mathrm{Pic}}(\mathcal{M})_{\Q}$. More precisely, for any $\Q$-linear map $\alpha:\widehat{\Pic}(\cM)_\Q\rightarrow\C$, we have $\alpha(\widehat{\Phi}_L)\in M_{1+\frac{b}{2}}(\rho_L)$.
\end{theorem}

\subsection{Asymptotic estimates for the global height}\label{Eisenstein}
In this subsection, we provide asymptotic estimates for the global height. First, we introduce an Eisenstein series $(\tau,s)\rightarrow E_{0}(\tau,s)$ for $\tau\in \bH$ and $s\in \C$ with $\re(s)>\frac{1}{2}-\frac{b}{4}=1-\frac{k}{2}$, which serves two purposes. First, the Fourier coefficients of its value at $s=0$ gives the main term in the Fourier coefficients of $\widehat{\Phi}_L$, see the proof of \Cref{eq1}. Second, we will use the Fourier coefficients of $E_0(\tau,s)$ to describe $\Phi_m$ explicitly in \S\ref{explicitPhi}.

Let $(\tau,s)\rightarrow E_{0}(\tau,s)$ denote the Eisentein series defined in \cite[Equation (1.4), (3.1) with $\beta=0, \kappa=1+\frac{b}{2}$]{bruinierintegrals}.   
It converges normally on $\mathbb{H}$ for $\mathrm{Re}(s)>1-\frac{k}{2}$ and defines a $\mathrm{Mp}_{2}(\Z)$-invariant real analytic function. 

For a fixed $s\in \C$ with $\mathrm{Re}(s)>1-\frac{k}{2}$, by \cite[Proposition 3.1]{bruinierintegrals}, the Eisenstein series $E_0(\cdot,s)$ has a Fourier expansion of the form 
\begin{align*}
E_{0}(\tau,s)=\sum_{\beta\in L^{\vee}/L}\sum_{m\in Q(\beta)+\Z}c_{0}(\beta,m,s,y)e^{2\pi i m x}\fe_{\beta},
\end{align*}
where we write $\tau=x+iy, x\in \R, y\in \R_{>0}$.
By \cite[Proposition 3.2]{bruinierintegrals}, the coefficients $c_0(\beta,m,s,y)$ can be decomposed, for $m\neq 0$, as 
\begin{align}\label{cfunction}
c_0(\beta,m,s,y)=C(\beta,m,s)\mathcal{W}_s(4\pi m y),
\end{align}
where the function $C(\beta,m,s)$ is independent of $y$ (see \cite[Equation (3.22)]{bruinierintegrals}) and $\cW_s$ is defined in \cite[(3.2)]{bruinierintegrals}.\medskip

By \cite[Proposition 3.1, (3.3)]{bruinierintegrals}, the value at $s=0$ of $E_0(\tau,s)$ is an element of $\mathrm{M}_{1+\frac{b}{2}}(\rho_L)$. For $\beta\in L^{\vee}/L$, $m\in Q(\beta)+\Z$ with $m\geq 0$, we denote by $c(\beta,m)$ its $(\beta,m)$-th Fourier coefficient and we can thus write  
\begin{align*}
 E_0(\tau):=E_0(\tau, 0)=2\fe_0+\sum_{\underset{m\in Q(\beta)+\Z,m>0}{\beta\in L^{\vee}/L}}c(\beta,m)q^m \fe_{\beta}, \text{ where }q=e^{2\pi i \tau}. 
\end{align*}
By definition and \emph{loc.~cit.}, we have $C(\beta,n,0)=c(\beta,n)$.
By \cite[Prop.4.8]{bruinierintegrals}, the coefficient $c(\beta,m)$ encodes the degree of the special divisor $Z(\beta,m)(\C)$. Moreover, \cite[Proposition 4, equation (19)]{bruinierkuss} gives explicit formulas for $c(\beta,m)$. By \cite[Proposition 14]{bruinierkuss}, $c(\beta,m)<0$ if $m\in Q(L+\beta)$ and $c(\beta,m)=0$ if $m\notin Q(L+\beta)$. By \cite[Example 2.3]{tayouequi}, we have that for $m\in Q(L+\beta)$,\footnote{recall that $b\geq 3$}
\begin{align}\label{asymp_fourier}
    |c(\beta,m)|=-c(\beta,m)\asymp m^{\frac{b}{2}}.
\end{align}

We will henceforth focus on the case where $\beta=0$ and we set $C(m,s):=C(0,m,s)$ and $c(m):=c(0,m)$. We are now ready to establish asymptotics for the global height in terms of the Fourier coefficients just defined.
\begin{proposition}\label{eq1}
For every $\epsilon >0$ and $m\in \Z_{>0}$, we have: 
\begin{align*}
h_{\widehat{\mathcal{Z}}(m)}(\mathcal{Y})=\frac{-c(m)}{2}h_{\overline{\bomg}}(\mathcal{Y})+O_{\epsilon}(m^{\frac{2+b}{4}+\epsilon}).
\end{align*}
In particular, we have $h_{\widehat{\mathcal{Z}}(m)}(\cY)=O(m^{\frac{b}{2}})$ as $m\rightarrow \infty$.
\end{proposition}

The second claim follows from the first claim and \eqref{asymp_fourier}.
\begin{proof}
The proof is similar to the one in \cite[Proposition 2.5]{tayouequi}.
For $\widehat{\cZ}\in \widehat{\CH^1}(\cM)_\Q \cong \widehat{\Pic}(\cM)_{\Q}$, the height $h_{\widehat{\cZ}}(\cY)$ defines a $\Q$-linear map $\widehat{\Pic}(\cM)_{\Q}\rightarrow \R$; 
by \Cref{modularity}, the following generating series 
$$-h_{\overline{\bomg}}(\cY)\fe_0+\sum_{\underset{m> 0, m\in Q(\beta)+\Z}{\beta\in L^{\vee}/L}}h_{\widehat{\mathcal{Z}}(\beta,m)}(\mathcal{Y})\cdot q^{m}\fe_\beta$$
is the Fourier expansion of an element in $M_{1+\frac{b}{2}}(\rho_L)$.
By \cite[p.27]{bruinier}, we write
$$-h_{\overline{\bomg}}(\cY)\fe_0+\sum_{\underset{m> 0, m\in Q(\beta)+\Z}{\beta\in L^{\vee}/L}}h_{\widehat{\mathcal{Z}}(\beta,m)}(\cY)\cdot q^{m}\fe_\beta=\frac{-h_{\overline{\bomg}}(\cY)}{2}E_0+g$$
where $E_0=E_0(\tau)$ is the Eisenstein series recalled in \S\ref{Eisenstein} and $g\in M_{1+\frac{b}{2}}(\rho_L)$ is a cusp form, see \cite[Def.1.2]{bruinier} for a definition.

For $m\in \Z_{>0}$, the equation for the $\fe_0$-component implies that $$h_{\widehat{\mathcal{Z}}(m)}(\cY)=\frac{-c(m)}{2}h_{\overline{\bomg}}(\cY)+g(m),$$
where $g(m)$ is the $m$-th Fourier coefficient of the $\fe_0$-component of $g$. We obtain the desired estimate by \cite[Prop. 1.5.5]{sarnak}, which implies that 
$$|g(m)|\leq C_{\epsilon,g} m^{\frac{2+b}{4}+\epsilon},$$
for all $\epsilon >0$, some constant $C_{\epsilon,g}>0$, and for all $m\in \mathbb{Z}_{>0}$. 
\end{proof}

\section{General results on quadratic forms}\label{quadraticgeneral}
In this section, we collect some general results on quadratic forms which will be used in \S\S\ref{mainproof}-\ref{sec_finite}. First, in \S\ref{sec_density}, we prove estimates on the number of local representations of integral quadratic forms which will be used is Section 5.1. Then in \S\ref{circle_method}, we state results due to Heath-Brown on the number of integral representations of integral quadratic forms which will be used in Sections 6 and 7. In \S\ref{Vol-Nie}, we apply Heath-Brown's results to the lattice $(L,Q)$ in \S\ref{gspin} and also recall the work of Niedermowwe, which could be viewed as a refinement of Heath-Brown's work - these results are used in Section 6. 
The reader may skip this section first and refer back later.

\subsection{Local estimates of representations by quadratic forms}\label{sec_density}

Recall that $(L,Q)$ is an even quadratic lattice of signature $(b,2)$ with $b\geq 3$ and $L$ is maximal in $V=L\otimes \Q$. Let $r=b+2$ denote the rank of $L$ and let $\det(L)$ denote the Gram determinant of $L$.
Let $p$ be a fixed prime, and let $\val_p$ denote the $p$-adic valuation.
For integers $m\in\Z$ and $n\in \Z_{\geq 0}$, we define the set $\mathcal{N}_{m}(p^n)$, its size $N_{m}(p^n)$, and the density $\mu_{p}(m,n)$ as follows:
\begin{equation}\label{eq:localdensity}
\begin{array}{rcl}
\displaystyle\mathcal{N}_{m}(p^n)&=&\displaystyle\{v\in L/p^{n}L \mid Q(v)\equiv m (\bmod p^n)\};\\[.1in]
\displaystyle N_{m}(p^n)&=&\displaystyle |\mathcal{N}_{m}(p^n)|;\\[.1in]
\displaystyle \mu_{p}(m,n)&=& \displaystyle p^{-n(r-1)}N_m(p^n).
\end{array}
\end{equation}

The goal of this subsection is to study the variation of the quantity $\mu_{p}(m,n)$. 
Define the quantity $w_p(m):=1+\val_p(m)$ for $p\neq 2$ and $w_2(m):=1+2\val_2(2m)$. Then we prove the following result.
\begin{proposition}\label{count}
If $m$ is an integer representable by $Q$ over $\Z$, then we have
\begin{align*}
\left|w_p(m)-\sum_{n=0}^{w_p(m)-1}\frac{\mu_{p}(m,n)}{\mu_{p}(m,w_p(m))}\right|\ll\frac{1}{p},
\end{align*} 
where the implied constant is independent of $p$ and $m$.
\end{proposition}

We use an inductive method, due to Hanke \cite{hanke}, to compute the quantities $\mu_{p}(m,n)$.
Let $L_{p}:=L\otimes \Z_p$ be the completion of $L$ at $p$. Since $L$ is maximal in $V$, it follows that $L_p$ is maximal in $V\otimes\Q_p$. Indeed, $L$ being maximal is equivalent to the fact that $L^\vee/L$ has no totally isotropic subgroup; thus its $p$-torsion, which is equal to $L_p^\vee/L_p$, has no totally isotropic subgroup implying that $L_p$ is maximal. It is well known that the quadratic lattice $(L_p,Q)$ admits an orthogonal decomposition 
\begin{equation}\label{eqOD}
(L_{p},Q)=\bigoplus_j (L_j,p^{\nu_j}Q_j)    
\end{equation} with $\nu_j\geq 0$, such that $(L_j,Q_j)$ is a $\Z_p$-unimodular quadratic lattice of dimension $1$ or $2$. (See, for example, \cite[(2.3),Lemma 2.1]{hanke}.) Moreover, when $p\neq 2$, then every $L_j$ is of dimension $1$. For every $v\in L_p$, we write $v=(v_j)_j$, and we have
$$Q(v)=\sum_{j}p^{\nu_j}Q_j(v_j).$$ 
Note that since $L_p$ is maximal, we have $\nu_j\leq 1$ for each $j$. For $i\in\{0,1\}$, let $S_i$ denote the set of indices $j$ with $\nu_j=i$, and let $s_i$ denote the size of $S_i$.


\vspace{.05in}

Following \cite[Definition 3.1]{hanke}, for $n\geq 1$, $v\in \mathcal{N}_{m}(p^n)$, we say that $v$ is of 
\begin{enumerate}
\item \emph{zero type} if $v\equiv 0 \pmod{p}$, 
\item \emph{good type} if there exists $j$ such that $v_j\not\equiv 0 \pmod{p}$ and $\nu_j=0$,
\item \emph{bad type} otherwise.   
\end{enumerate}
Let $\mathcal{N}_{m}^{\textrm{good}}(p^n)$, $\mathcal{N}_{m}^{\textrm{bad}}(p^n)$ and $\mathcal{N}_{m}^{\textrm{zero}}(p^n)$ be the set of good type, bad type and zero type solutions respectively and set $N_{m}^{?}(p^n)=|\mathcal{N}_{m}^{?}(p^n)|$ and $\mu_p^{?}(m,n)=p^{-n(r-1)}N_m^{?}(p^n)$, for $?=$ good, bad, or zero.
Note also from \cite[Remark 3.4.1]{hanke}, that we have $\cN_m(p^n)=\cN_m^\good(p^n)$ when $p\nmid m$; $\cN_m^{\zero}(p^n)=\emptyset$ when $p^2\nmid m$ and $n\geq 2$; and $\cN_m^{\bad}(p^n)=\emptyset$ when $p\nmid 2\det(L)$.

To state the inductive result on local densities, we need to introduce the auxiliary form $Q'$, where $Q'$ is obtained from the orthogonal decomposition \eqref{eqOD} of $Q$ by replacing $\nu_j$ with $\nu_j'=1-\nu_j$ for all $j$.
To distinguish between the local densities of $Q$ and $Q'$, we use $N_{m,Q}^{?}(p^n)$ and $\mu_{p,Q}^{?}(m,n)$ to emphasis the dependence on the quadratic form.
We now recall Hanke's inductive method with the simplification that $L$ is maximal (only used in (2)). 
\begin{lemme}[Hanke]\label{lem_hanke}
Let $n\in \Z_{>0}$ and set $\delta= 2 \val_2(p)+1$.
\begin{enumerate}
    \item For $n\geq \delta$ and all integers $m$, we have
    \[N_m^\good(p^n)=p^{(n-\delta)(r-1)}N_m^\good(p^\delta); \quad \mu_{p}^\good(m,n)=\mu_{p}^\good(m,\delta).\]
    \item For $m$ such that $p\mid m$, we have \[N_{m,Q}^{\bad}(p^{n+1})=p^{r-s_0} N^{\good}_{\frac{m}{p}, Q'}(p^{n}); \quad \mu_{p,Q}^{\bad}(m,n+1)=p^{1-s_0}\mu_{p,Q'}^{\good}\left(\frac{m}{p},n\right).\]
    \item For $m$ such that $p^2\mid m$, we have \[N_m^\zero(p^{n+2})=p^rN_{\frac{m}{p^2}}(p^n); \quad \mu_{p}^\zero(m,n+2)=p^{2-r}\mu_{p}\left(\frac{m}{p^2},n\right).\]
\end{enumerate}
\end{lemme}
\begin{proof}
All the assertions on $\mu_p$ follow from the assertions on $N_m$ by definition. The first assertion of the lemma is \cite[Lemma 3.2]{hanke}. The third assertion follows from the last two paragraphs on \cite[p.359]{hanke}. To recover the second assertion, note that since $L_p$ is maximal, we have $\nu_j\leq 1$ for all $j$. Hence Bad-type II points (see \cite[p.360]{hanke}) do not exist. Thus, the claim follows from the discussion on Bad-type I points in \emph{loc.~cit.}.
\end{proof}

\begin{corollaire}\label{explicit_alpha}
For $p\neq 2$, set $\delta_{p,\det(L)}=1$ if $p\mid \det(L)$ and $0$ otherwise.\footnote{In the formulas below, we do not need to introduce this term $
\delta_{p,\det(L)}$, because by definition, if $p\nmid \det(L)$, then $\nu'_j=1$ for all $j$ and hence $\mu^\good_{p,Q'}(m,n)=0$. Nevertheless, we put it here to emphasis that those terms are $0$.} Recall the quadratic form $Q'$ and the integer $s_0$ defined above. We have
\begin{enumerate}
    \item If $n\geq \val_p(m)+1$, then $\mu_{p,Q}(m,n)$ is equal to
    \[\sum_{u=0}^{\lfloor \frac{\val_p(m)}{2}\rfloor} p^{(2-r)u}\mu^\good_{p,Q}\left(\frac{m}{p^{2u}},1\right)+\delta_{p,\det(L)}p^{1-s_0}\sum_{u=0}^{\lfloor \frac{\val_p(m)-1}{2}\rfloor}p^{(2-r)u}\mu^\good_{p,Q'}\left(\frac{m}{p^{2u+1}},1\right).\]
    \item If $1\leq n \leq \val_p(m)$ and $n$ odd, then $\mu_{p,Q}(m,n)$ is equal to
    \[p^{\frac{(2-r)(n-1)}{2}}\mu_{p,Q}(mp^{1-n},1)+\sum_{u=0}^{\frac{n-3}{2}} p^{(2-r)u}\mu^\good_{p,Q}\left(\frac{m}{p^{2u}},1\right)+\delta_{p,\det(L)}p^{1-s_0}\sum_{u=0}^{\frac{n-3}{2}}p^{(2-r)u}\mu^\good_{p,Q'}\left(\frac{m}{p^{2u+1}},1\right).\]
    \item If $1\leq n \leq \val_p(m)$ and $n$ even, then $\mu_{p,Q}(m,n)$ is equal to
    \[p^{\frac{(2-r)(n-2)}{2}}\mu^\zero_{p,Q}(mp^{2-n},2)+\sum_{u=0}^{\frac{n-2}{2}} p^{(2-r)u}\mu^\good_{p,Q}\left(\frac{m}{p^{2u}},1\right)+\delta_{p,\det(L)}p^{1-s_0}\sum_{u=0}^{\frac{n-2}{2}}p^{(2-r)u}\mu^\good_{p,Q'}\left(\frac{m}{p^{2u+1}},1\right).\]
\end{enumerate}
\end{corollaire}
\begin{proof}
The base cases when $\val_p(m)\leq 1$ or $n\leq 2$ can be checked directly by definition and \Cref{lem_hanke}.
For $n>2$ and $p^2\mid m$, by \Cref{lem_hanke}, 
\begin{align*}
\mu_{p,Q}(m,n) & = \mu_{p,Q}^\good(m,n)+\mu_{p,Q}^\zero(m,n)+\mu_{p,Q}^\bad(m,n)\\
&= \mu_{p,Q}^\good(m,1)+p^{2-r}\mu_{p,Q}\left(\frac{m}{p^2},n-2\right)+p^{1-s_0}\mu^\good_{p,Q'}\left(\frac{m}{p},n-1\right)\\
&=\mu_{p,Q}^\good(m,1)+p^{2-r}\mu_{p,Q}\left(\frac{m}{p^2},n-2\right)+p^{1-s_0}\mu^\good_{p,Q'}\left(\frac{m}{p},1\right).
\end{align*}
Then we conclude by induction on $\val_p(m)$ and $n$.
\end{proof}

The next lemma gives a uniform bound on $|\mu_{p}(m,n)-\mu_p(m,w_p(m))|$ for primes $p$, integers $m$, and $n\in\{2,\ldots,w_p(m)-1\}$. 
\begin{lemme}\label{lem_uniform_alpha}
Let $p$ be prime and $m$ be any integer. For $n\in\{2+\val_2(p),\ldots,w_p(m)-1\}$, we have
\[|\mu_{p}(m,n)-\mu_p(m,w_p(m))|\ll \frac{1}{p^{3\lfloor(n/2)\rfloor-2}},\]
where the implied constant is absolute.
\end{lemme}
\begin{proof}
First consider the case when $p$ is odd.
By \Cref{explicit_alpha}, we have that for $n$ odd, 
\begin{equation*}
\begin{array}{rcl}
\displaystyle|\mu_{p}(m,n)-\mu_p(m,w_p(m))|&\leq &
\displaystyle\left|\frac{\mu_{p,Q}\left(\frac{m}{p^{n-1}},1\right)}{p^{\frac{(r-2)(n-1)}{2}}}\right|\\[.3in]
&+&\displaystyle
\left|\sum_{u=\frac{(n-1)}{2}}^{\lfloor \frac{\val_p(m)}{2}\rfloor}
\frac{\mu^\good_{p,Q}\left(\frac{m}{p^{2u}},1\right)+\delta_{p,\det(L)}p^{1-s_0}\mu^\good_{p,Q'}\left(\frac{m}{p^{2u+1}},1\right)}{p^{(r-2)u}}\right|\\[.3in]
&\leq& \displaystyle\left| \frac{p}{p^{\frac{3(n-1)}{2}}}\right| +\sum_{u=\frac{(n-1)}{2}}^{\lfloor \frac{\val_p(m)}{2}\rfloor}\left|
\frac{p+p^2}{p^{3u}}\right|\\[.2in]
&\leq& \displaystyle\frac{C_1 p^2}{p^{\frac{3(n-1)}{2}}}.
\end{array}
\end{equation*}
Here we use the trivial bound that all $|\mu_{p,Q}\left(\frac{m}{p^{n-1}},1\right)|, |\mu^\good_{p,Q}\left(\frac{m}{p^{2u}},1\right)|, |\mu^\good_{p,Q'}\left(\frac{m}{p^{2u+1}},1\right)|$ are less than $p$ by definition. The case when $n$ is even follows by a similar argument and the trivial bound that $\mu^\zero_{p,Q}(mp^{2-n},2)\leq p^r/p^{2(r-1)}=p^{2-r}$.

For $p=2$, by \Cref{lem_hanke}, we obtain analogous statements as \Cref{explicit_alpha} except that we can only reduce to $\mu^\good_{p}(?,3)$ (instead of $\mu^\good_{p}(?,1)$). The rest of the argument is the same as in \Cref{lem_uniform_alpha} and since $p=2$ is fixed, any trivial bound on density is absorbed in the absolute constant.
\end{proof}

We can actually show that all $\mu_{p}(m,n)$ are close to $1$ when $p\nmid 2\det(L)$.
\begin{lemme}\label{close1}
There exists an absolute constant $C_2>0$ such that for all $m,n\in \Z_{>0}$, all primes $p\nmid 2\det(L)$, we have
\[|\mu_{p}(m,n)-1|\leq \frac{C_2}{p}.\]
\end{lemme}
\begin{proof}
By \Cref{explicit_alpha}(1), \Cref{lem_uniform_alpha}, we only need to show the claim for $n=1$ and $n=\val_p(m)+1$. For $n=1$, we first consider the case when $p\mid m$. Then $Q(v)\equiv 0\bmod p$ defines a smooth projective hypersurface in $\P^{r-1}$; except the solution $v=0 \bmod p$, every $p-1$ solutions of $Q(v)\equiv 0 \bmod p$ (all these are of good type) correspond to a $\bF_p$-point in the hypersurface. 
Then by the Weil bound (see for instance \cite[Théorème 8.1]{weil1}),\footnote{In our specific case, namely that of a quadratic form, this result was in \cite{weil49}} there exists a constant $C_3>0$ independent of $p$ and $m$ such that
\begin{align*}
|N_{m}^{\good}(p)-p^{r-1}|\leq C_3 p^{r-2}.
\end{align*}
Therefore, $|\mu_{p}^\good(m,1)-1|\leq C_3/p$.

For $p\nmid m$, we consider the smooth projective hypersurface in $\P^r$ defined by $Q(v)=my^2$. In this case $N_m(p)=N^\good_m(p)$ is the number of $\bF_p$-points in the hypersurface such that $y\neq 0$ in $\bF_p$. 
By the Weil bound, the number of $\bF_p$-points in the hypersurface is $p^{r-1}+O_m(p^{r-2})$; the number of $y=0$ points on the hypersurface is $p^{r-2}+O(p^{r-3})$ by the Weil bound. 
Then we conclude that there exists a constant $C_4>0$ independent\footnote{Although the equation of the hypersurface depends on $m$, the number of solutions only depends on whether $m$ is a quadratic square in $\bF_p$. So we only need to apply the Weil bound for a fixed square $m$, and for some fixed non-square $m$, to obtain some $C_4$ independent of $m$.} of $p$ and $m$ such that $|\mu_{p}^\good(m,1)-1|\leq C_4/p$. In particular, for any $m$, we have $$\mu_p^{\good}(m,1)\leq 1+\max\{C_3,C_4\}.$$

For $n=\val_p(m)+1$ and $p|m$, by \Cref{explicit_alpha}(1) and note that $\delta_{p,\det(L)}=0$, we have
\begin{align*}
|\mu_{p}(m,\val_p(m)+1)-1|&=\left|\mu_p^{\mathrm{good}}(m,1)-1+\sum_{u=1}^{\floor*{\frac{\val_p(m)}{2}}}\frac{\mu_{p}^{\mathrm{good}}(\frac{m}{p^{2u}},1)}{p^{u(r-2)}}\right|\\
&\leq \frac{C_3}{p}+\sum_{u=1}^{\floor*{\frac{\val_p(m)}{2}}}\frac{\max\{C_3,C_4\}+1}{p^{u(r-2)}} \leq \frac{C_5}{p},
\end{align*}
where we take $C_2=\max\{C_3,C_4,C_5\}$.
\end{proof}


Due to our assumption that $r\geq 5$ and $L$ maximal, there is an absolute lower bound for $\mu_{p}(m,n)$. The following lemma is well known, but we include it for the convenience of the reader.
\begin{lemme}\label{unif_lower}
Recall that $r\geq 5$ and $L$ is maximal. Then for any $m,n\in \Z_{>0}$, any prime $p$, we have $\mu_{p}(m,n)\geq 1/2$.
\end{lemme}

\begin{proof}
Since $r\geq 5$ and $L$ is maximal, then by for instance \cite[Lemma 6.36]{gerstein}, 
for every prime $p$, there exists a basis of $L_p$ such that in the coordinate of this basis, $Q((x_1,\dots,x_r))=x_1x_2+Q_1((x_3,\dots,x_r))$, where $Q_1$ is a quadratic form in $(r-2)$ variables. 

Recall as in \Cref{lem_hanke}, that $\delta=3$ if $p=2$ and $\delta=1$ otherwise. Fix an integer $\delta'$ satisfying $1\leq \delta'\leq \delta$. For any $x_1\in (\Z/p^{\delta'})^\times$ and any $x_i\in\Z/p^{\delta'}, 3\leq i \leq r$, there exists a unique $x_2\in \Z/p^{\delta'}$ such that $Q(x_1,\dots,x_r))=m\bmod p^{\delta'}$. Therefore $\mu^\good_{p}(m,\delta')\geq \frac{p-1}{p}\geq 1/2$ and hence by \Cref{lem_hanke}(1), for $n\geq \delta$, $\mu_{p}(m,n)\geq \mu^\good_p(m,\delta)\geq 1/2$.
\end{proof}

\begin{corollaire}\label{cor_repD}
Every large enough $m\in \Z_{>0}$ is representable by $(L,Q)$.

\end{corollaire}
\begin{proof}
By \cite[Theorem 4]{HB} (recalled in Theorem 4.9 below from which we borrow the notations), we can choose a non-negative test function $\omega$ such that $\mu_\infty(Q,\omega)>0$. Then by the Proposition above, $\mu(Q,m)>0$ and hence $N(Q,m,\omega)$ is positive for $m$ large enough. Hence $m$ is representable by $(L,Q)$.
\end{proof}

Now we are ready to prove the main result of this subsection.

\begin{proof}[Proof of \Cref{count}]
For simplicity of notation, we denote $w_p(m)$ by $w_p$.

{\bf First case:} assume that $p\nmid 2\det(L)$. 
By \Cref{close1}, 
\begin{align}\label{one-w}
|\mu_{p}(m,w_p)-\mu_{p}(m,1)|\leq |\mu_{p}(m,w_p)-1|+|\mu_{p}(m,1)-1|\leq C_7/p.
\end{align}
Note that $\mu_{p}(m,0)=1$ by definition, see \ref{eq:localdensity}. Then by Lemmas \ref{lem_uniform_alpha},\ref{close1} and \eqref{one-w}, we get 
\begin{align*}
\left|w_p-\sum_{n=0}^{w_p-1}\frac{\mu_{p}(m,n)}{\mu_{p}(m,w_p)}\right|&\leq\frac{1}{\mu_{p}(m,w_p)}\left[\sum_{n\geq 2}\frac{C_1 p^2}{p^{3\lfloor n/2\rfloor}}+\frac{C_7}{p}+\frac{C_2}{p}\right] \leq \frac{C_8}{\mu_{p}(m,w_p)p}
\end{align*}
We conclude by the fact that $\mu_{p}(m,w_p)$ is uniformly bounded away from $0$ by \Cref{unif_lower}. 

\medskip 
{\bf Second case}: assume now that $p\mid 2\det(L)$. 
By \Cref{lem_uniform_alpha}, 
for any $n\geq 3$,\begin{align}\label{eq_alpha}
|\mu_{p}(m,w_p)-\mu_{p}(m,n)|\leq \frac{C_9p^2}{p^{3\lfloor n/2\rfloor}}.
\end{align} 
Then as in the first case we have \[\left|\sum_{n=3}^{w_p-1}\frac{\mu_{p}(m,w_p)-\mu_{p}(m,n)}{\mu_{p}(m,w_p)}\right|\leq \frac{C_{10}}{p}.\]
On the other hand, for $0\leq n\leq 2$, we have for all $p\mid 2\det(L)$
\[|\mu_{p}(m,w_p)-\mu_{p}(m,n)|\leq |\mu_{p}(m,w_p)-\mu_{p}(m,3)|+|\mu_{p}(m,3)-\mu_{p}(m,n)|\leq C_{11}/p\]
by \eqref{eq_alpha} and the trivial bound $|\mu_{p}(m,n)|,|\mu_{p}(m,3)|\leq p^3\leq (2\det(L))^3$.
Then we conclude as in the first case.
\end{proof}

\subsection{On the number of representations of quadratic forms}\label{circle_method}

Developing a new form of the circle method, Heath-Brown \cite{HB} proves a number of results pertaining to the representation of integers by quadratic forms. The purpose of this subsection is to describe the setup used in \cite{HB}, and recall those results necessary for us in the sequel. We do not entirely keep the notations of \cite{HB} since we will only be concerned with homogeneous quadratic forms, which allows us to make certain simplifications in the notation.

Let $F(\bfx)=F(x_1,\ldots,x_n)$ be an integral quadratic form with non-zero discriminant in $n\geq 5$ variables. A function $\omega:\R^n\to\C$ is said to be a {\it smooth weight function} if it is infinitely differentiable of compact support. Given a set $S$ of parameters (see \cite[Page 6]{HB} for what parameters are allowed to be), Heath-Brown defines a set of weight functions $\cC(S)$ in \cite[\S2]{HB}. 
\begin{remarque}\label{rmk_wt}
The following facts on weight functions will be used later (see for instance the observations on \cite[p.~162]{HB}). 
\begin{enumerate}
    \item There exists a function $\omega^{(n)}_0(\bfx):\R^n\rightarrow[0,2]$ with compact support in $[-1,1]^n$ which belongs to $\cC(n)$ such that $\omega^{(n)}_0(\bfx)\geq 2$ for $x\in [-1/2,1/2]^n$ (for instance, by rescaling the function defined by \cite[(2.1),(2.2)]{HB}). 
    \item Let $M$ be an invertible $n\times n$ matrix, such that the coefficients of both $M$ and $M^{-1}$ are bounded in absolute value by $K$. If $\omega$ is a weight function belonging to $\cC(S)$, then $\omega(Mx)$ belongs to $\cC(S,K)$.
\end{enumerate}
\end{remarque}

The reason for introducing the set $\cC(S)$ is the following. For the quadratic form $F$ fixed as above, an integer $m\neq 0$, and a weight function $\omega\in\cC(S)$ for some set of parameters $S$, we define
\begin{equation*}
N(F,m,\omega):=\sum_{\substack{\bfx\in\Z^n\\F(\bfx)=m}}\omega\Bigl(\frac{\bfx}{\sqrt{m}}\Bigr).
\end{equation*}
The quantity $N(F,m,\omega)$ then is a weighted sum of representations of $m$ by $F$, where the coordinates of these representations are bounded by $O_S(\sqrt{m})$, since $\omega$ has compact support. Then \cite{HB} gives asymptotics for the size of $N(F,m,\omega)$, where the error term only depends on $S$.

More precisely, define the {\it singular integral} by\footnote{As explained on \cite[p.154-155]{HB}, for weight functions $\omega\in \cC(S)$, this limit exists.}
\begin{equation*}
\mu_\infty(F,\omega):=\lim_{\epsilon\to 0}\frac{1}{2\epsilon}\int_{|F(\bfx)-1|\leq\epsilon}\omega(\bfx)d\bfx.
\end{equation*}
Recall the Siegel mass at $p$ of the quadratic form $F$ given by\footnote{Using the notation in \S\ref{sec_density}, $\mu_p(F,m)=\lim_{k\rightarrow\infty}\mu_{p,F}(m,k)$. From the discussion in \S\ref{sec_density}, given $F,m,p$, $\mu_{p,F}(m,k)$ stabilizes as $k\gg 1$ and hence the limit automatically exists.}
\begin{equation*}
\mu_p(F,m):=\lim_{k\to\infty}\frac{1}{p^{(n-1)k}}|\bigl\{
\bfx\!\!\!\!\!\pmod{p^k}: F(\bfx)\equiv m\!\!\!\!\!\pmod{p^k}\bigr\}|
\end{equation*}
and define the {\it singular series} by
\begin{equation*}
\mu(F,m):=\prod_p\mu_p(F,m).
\end{equation*}
In our situation with $n\geq 5$, the singular series always converges absolutely, see also for instance, \cite[\S 11.5]{iwaniec}. Now \cite[Theorem 4]{HB} states the following.
\begin{theorem}[{\cite[Theorem 4]{HB}}]\label{HB_thm4}
Let notation be as above, and let $\omega\in\cC(S)$ be a weight function for some set of parameters $S$. Then
\begin{equation*}
N(F,m,\omega)=\mu_\infty(F,\omega)\mu(F,m)m^{n/2-1}+O_{F,S,\epsilon}\bigl(m^{(n-1)/4+\epsilon}\bigr).
\end{equation*}
\end{theorem}

\noindent Note in particular that the error term depends only on $F$ and $S$, and not on the specific weight function $\omega$ or on $m$.

We also recall a corollary of the above theorem for positive definite quadratic forms, which will be used in \S\ref{sec_finite}.
\begin{corollaire}[{\cite[Corollary 1]{HB}}]\label{HB_cor}
Let notation be as above and assume further that $F$ is positive definite. Then
\[|\{\bfx\in \Z^n: F(\bfx)=m\}|=\mu_{\infty}(F,1)\mu(F,m)m^{n/2-1}+O_{F,\epsilon}(m^{(n-1)/4+\epsilon}).\]
\end{corollaire}

\subsection{An application of Heath-Brown's theorem and a result of Niedermowwe}\label{Vol-Nie}

Recall that $(L,Q)$ is an even quadratic lattice of signature $(b,2)$ with $b\geq 3$. We will apply Heath-Brown's result to $(L,Q)$ (here we identify $L$ with $\Z^{b+2}$) after we construct suitable smooth weight functions. Moreover, we recall Niedermowwe's result, which is analogous to Heath-Brown's result but with sharp weight functions given by characteristic functions of certain expanding domains $\Omega_T$ defined below and keeps track of the dependence of the error term on $T$.

Similar to the definition of the singular integral $\mu_\infty(F,\omega)$ in \S\ref{circle_method}, we define a measure $\mu_\infty$ on
\begin{equation*}
    L_{\R,1}:=\{\lambda\in L_\R:Q(\lambda)=1\}
\end{equation*}
as follows. For an open bounded subset $W$ of $L_\R$ we set
\begin{equation}\label{def_muinfty}
    \mu_\infty(W\cap L_{\R,1}):=
    \lim_{\epsilon\to 0}\frac{1}{2\epsilon}
    \mu_L\bigl(\{\lambda\in W:|Q(\lambda)-1|<\epsilon\}\bigr),
\end{equation}
where $\mu_{L}$ is the Lebesgue measure on $L_{\R}$ normalized so that $L$ has covolume $1$.\footnote{Let $\bfone_W$ be the characteristic function on $W$. Then by abuse of notation (since $\bfone_W$ is not smooth), we have $\mu_\infty(Q, \bfone_W)=\mu_\infty(W\cap L_{\R,1})$.} 

Let $x$ denote a fixed point in the period domain $D_L$ and let $P$ denote the negative definite plane (with respect to $Q$) in $L_{\R}$ associated to $x$. Let $P^\perp$ denote the orthogonal complement of $P$ in $L_{\R}$. Given a vector $\lambda\in L_\R$, we let $\lambda_x$ and $\lambda_{x^\bot}$ denote the projections of $\lambda$ to $P$ and $P^\perp$, respectively.

We introduce notation for the set of elements in $L_{\R,1}$ with bounded value of $Q(\lambda_x)$: for $T>0$, define
\begin{equation*}
    \Omega_{\leq T}:=\{\lambda\in L_{\R,1}:\,-Q(\lambda_x)\in[0,T]\}.
\end{equation*}
Then the following lemma computes the volume of the sets $\Omega_{\leq T}$. Recall that $k=1+\frac{b}{2}$.
\begin{lemme}\label{secondmeasure}
Let $T>0$ be a real number. 
Then $$\mu_{\infty}(\Omega_{\leq T})=\frac{(2\pi)^{k}\left((1+T)^{\frac{b}{2}}-1\right)}{\sqrt{|L^{\vee}/L|}\Gamma(k)}.$$ 
\end{lemme}
\begin{proof}[Proof]
For $\epsilon >0$, let $U_{T,\epsilon}:=\{x\in L_{\R}:|Q(x)-1|<\epsilon,\, -Q(\lambda_x)<T\}$. Then $\Omega_{\leq T}=U_{T,\epsilon}\cap L_{\R,1}$ and by definition $$\mu_{\infty}(\Omega_{\leq T})=\lim_{\epsilon\rightarrow 0}\frac{\mu_{L}(U_{T,\epsilon})}{2\epsilon}.$$
Let $\mathcal{E}$  be an orthogonal basis of $L_{\R}$ adapted to the decomposition $P\oplus P^{\bot}$ and in which the bilinear form associated to $Q$ has the following intersection matrix $$\begin{pmatrix} -1&0&0\\
0&-1&0\\
0&0&I_b\\
\end{pmatrix},$$
where $I_b$ denotes the $b\times b$ identity matrix.
Let $\mu_{\mathcal{E}}$ be the associated Lebesgue measure for which the $\Z$-span of $\mathcal{E}$ is of covolume $1$. By change of variables, we have   
\begin{align*}
\mu_{L}(U_{T,\epsilon})&=\frac{2^{1+\frac{b}{2}}}{\sqrt{|L^{\vee}/L|}}\mu_{\mathcal{E}}(U_{T,\epsilon}) \\
&=\frac{2^{1+\frac{b}{2}}}{\sqrt{|L^{\vee}/L|}}\int_{\underset{\underset{x^2_1+x_2^2<T}{|x_1^2+x_2^2-y_1^2-\cdots-y_b^2+1|<\epsilon}}{(x_1,x_2,y_1,\cdots,y_b)\in\R^{b+2}}} dx_1 dx_2dy_1\cdots dy_b\\
&=\frac{2^{2+\frac{b}{2}}\pi}{\sqrt{|L^{\vee}/L|}}\int_{0}^{\sqrt{T}}\left(\int_{1+r^2-\epsilon<y_1^2+\cdots+y_b^2<1+r^2+\epsilon}dy_1\cdots dy_b \right) rdr\\
&=\frac{2(2\pi)^{1+\frac{b}{2}}}{\sqrt{|L^{\vee}/L|}\Gamma\left(1+\frac{b}{2}\right)}\int_{0}^{\sqrt{T}}\left((1+r^2+\epsilon)^{\frac{b}{2}}-\left(1+r^2-\epsilon\right)^{\frac{b}{2}}\right)rdr\\
&=2\epsilon.\frac{(2\pi)^{1+\frac{b}{2}} \left((1+T)^{\frac{b}{2}}-1\right)}{\sqrt{|L^{\vee}/L|}\Gamma\left(1+\frac{b}{2}\right)}+O(\epsilon^2)
\end{align*}
Dividing by $2\epsilon$ and letting $\epsilon$ go to zero, we get the desired result.
\end{proof}

We now describe the desired estimates for $\bigl|\bigl\{
\lambda\in L:Q(\lambda)=m,\,\lambda/\sqrt{m}\in \Omega_{\leq T}\bigr\}\bigr|$ in two cases. 

\subsubsection{Assume $T\leq 1$}
In this case, we only need a good upper bound and hence we construct a suitable smooth weight function $\omega:L_{\R}\to\R$ as follows and apply Heath-Brown's theorem. 

\begin{corollaire}\label{HBforL}
Given $(L,Q)$ as above.
For any $m\in \Z_{>0}$, any $0<T\leq 1$, we have
\[\bigl|\bigl\{
\lambda\in L:Q(\lambda)=m,\,\lambda/\sqrt{m}\in \Omega_{\leq T}\bigr\}\bigr|=O_Q(m^{\frac{b}{2}}T)+O_{Q,T,\epsilon}(m^{(b+1)/4+\epsilon}).\]
\end{corollaire}
\begin{proof}
By \Cref{rmk_wt}, there exist smooth functions $\omega_P:P\to[0,2]$ and $\omega_{P^\perp}:P^\perp\to[0,2]$ such that
\begin{enumerate}
    \item $\omega_P(\lambda_x)\geq 1$ for elements $\lambda_x\in P$ with $Q(\lambda_x)<T$ and $\omega_P(\lambda_x)=0$ if $Q(\lambda_x)>2T$,
    \item $\omega_{P^\perp}(\lambda_{x^\perp})\geq 1$ if $Q(\lambda_{x^\perp})\leq 1$ and $\omega_{P^\perp}(\lambda_{x^\perp})=0$ if $Q(\lambda_{x^\perp})\geq 2$.
\end{enumerate}
We define $\omega(\lambda)=\omega_P(\lambda_x)\omega_{P^\perp}(\lambda_{x^\perp})$ and by construction, $\omega\in \cC(b,T)$.

By definition, $\bigl|\bigl\{
\lambda\in L:Q(\lambda)=m,\,\lambda/\sqrt{m}\in \Omega_{\leq T}\bigr\}\bigr|\leq N(Q,m,\omega)$. By definition and \Cref{secondmeasure}, the singular integral $\mu_\infty(Q,\omega)\ll \mu_\infty(\Omega_{\leq 2T})=O(T)$. Then the assertion follows by
applying \Cref{HB_thm4} to $\omega$ and the fact that $\mu(Q,m)=O_F(1)$ (since $b\geq 3$). 
\end{proof}

\subsubsection{Assume $T\geq 1$} In this case, we will need the exact main term along with an error term with explicit dependence on $T$ and we will apply Niedermowwe’s work \cite{Nieder}.

For the convenience of later use,
for an integer $m\geq 1$, we define the quantity
\begin{equation}\label{definitionam}
a(m)=\frac{-c(m)\Gamma(k)\sqrt{|L^\vee/L|}}{2(2\pi)^{k}},
\end{equation}
where $c(m)$ is the $m$-th Fourier coefficient of the Eisenstein series defined in \S\ref{Eisenstein}.
Note that $a(m)$ grows as $\asymp m^{\frac{b}{2}}$.
We have the following proposition, which follows from work of Niedermowwe \cite{Nieder}.

\begin{proposition}\label{effective}
Let $A>1$ be a positive real number.
For any $m\in \Z_{>0}$, $T\geq 1$, we have 
\begin{equation*}
\bigl|\bigl\{
\lambda\in L:Q(\lambda)=m,\,\lambda/\sqrt{m}\in \Omega_{\leq T}\bigr\}\bigr|=
a(m)\mu_\infty(\Omega_{\leq T})+O\bigl(m^{\frac{b}{2}}T^{\frac{b}{2}}\log(mT)^{-A}\bigr).
\end{equation*}
\end{proposition}
\begin{proof}
In \cite[Theorem 3.6]{Nieder}, Niedermowwe estimates the number of lattice points with fixed norm in homogenously expanding rectangular regions. His proof carries over without change for our region, yielding that\footnote{The definitions of singular series are the same in \cite{HB} and \cite{Nieder}. For the definitions of singular integral, it suffices to compare the definitions when $\omega$ is a smooth weight function (say a good approximation of $\bfone_{U_T}$, where $U_T\subset L_\R\cong \R^{b+2}$ with $U_T\cap L_{\R,1}= \Omega_{\leq T}$). In \cite{Nieder}, the singular integral, denoted by $I_\omega(m)$, is defined to be $\int_{-\infty}^\infty\int_{\R^{b+2}}\omega(\bfx/\sqrt{m})\exp(2\pi i z(Q(\bfx)-m))d\bfx dz$, which equals to $\int_{Q(\bfx)=m}\omega(\bfx/\sqrt{m})(\frac{dQ}{dx_1})^{-1}dx_2\cdots dx_{b+2}$ by applying the Fourier inversion theorem to $f:\R\rightarrow \C$, $f(y):=\int_{Q(\bfx)=y}\omega(\bfx/\sqrt{m})(\frac{dQ}{dx_1})^{-1}dx_2\cdots dx_{b+2}$. Then by \cite[Theorem 3]{HB} and by a change of variables $\bfx\mapsto \bfx/\sqrt{m}$, $I_\omega(m)=m^{\frac{b}{2}}\mu_\infty(Q,\omega)$. Therefore, the leading term in \cite{Nieder} $I_\omega(m)\mu(Q,m)$ coincides with the leading term $\mu_\infty(Q,\omega)\mu(Q,m)m^{\frac{b}{2}}$ in \cite{HB}.}
\[\bigl|\bigl\{
\lambda\in L:Q(\lambda)=m,\,\lambda/\sqrt{m}\in \Omega_{\leq T}\bigr\}\bigr|=
\mu_\infty(\Omega_{\leq T})m^{\frac{b}{2}}\mu(Q,m)+O\bigl(m^{\frac{b}{2}}T^{\frac{b}{2}}\log(mT)^{-A}\bigr).\]
We then deduce the desired formula by the explicit formula for $c(m)$ in \cite[(22),(23)]{bruinierkuss}, which asserts that $\displaystyle c(m)=-\frac{2(2\pi)^{\frac{b}{2}+1}m^{\frac{b}{2}}}{\sqrt{|L^\vee/L|}\Gamma(\frac{b}{2}+1)}\prod_p \mu_p(Q,m)$.
\end{proof}


We conclude this section by an integral computation similar to \Cref{secondmeasure} which will be used later.
For $s\in\R$, consider the function 
\begin{align*}
h_s:L_{\R}&\rightarrow \R^+\\
\lambda&\mapsto \left(\frac{1}{1-Q(\lambda_x)}\right)^{k-1+s}.
\end{align*}
\begin{lemma}\label{hintegral}
For $s>0$, we have
\begin{equation*}
\int_{L_{\R,1}}h_s(\lambda)d\mu_{\infty}(\lambda)=
\frac{b}{4}\cdot \frac{|c(m)|}{s\cdot a(m)}.
\end{equation*}
\end{lemma}
\begin{proof}
As in the proof of Lemma \ref{secondmeasure}, we define
$L_\epsilon=:\{x\in L_{\R}:|Q(x)-1|<\epsilon\}$. Then
\begin{align*}
\lim_{\epsilon\rightarrow 0} \frac{1}{2\epsilon}&\int_{L_{\epsilon}}\left(\frac{1}{1-Q(\lambda_x)}\right)^{k-1+s}d\mu_{\infty}(\lambda)\\&=\lim_{\epsilon\rightarrow 0}\frac{2^{\frac{b}{2}}}{\epsilon\sqrt{|L^\vee/L|}}\int_{(x_1,x_2)\in \R^2}\int_{\underset{|y_1^2+\cdots+y_b^2-x_1^2-x_2^2-1|<\epsilon}{(y_1,\dots, y_b)\in \R^b}}\frac{dx_1dx_2dy_1\cdots dy_b}{\left(1+x_1^2+x_2^2\right)^{k-1+s}}\\ 
&=\frac{2(2\pi)^{1+\frac{b}{2}}}{\Gamma\left(\frac{b}{2}\right).\sqrt{|L^\vee/L|}}\int_{0}^{+\infty}\left(\frac{1}{1+r^2}\right)^{s+1}rdr\\
&=\frac{(2\pi)^{1+\frac{b}{2}}}{\Gamma\left(\frac{b}{2}\right).\sqrt{|L^\vee/L|}}\frac{1}{s}.
\end{align*}
The lemma now follows from the definition of $a(m)$.
\end{proof}

\section{First step in archimedean estimate and uniform diophantine bounds}\label{mainproof}

We keep the notations from \S\ref{gspin} and \S\ref{harmonic}. Namely, $(L,Q)$ is an even maximal lattice of signature $(b,2)$ with $b\geq 3$. Recall that $\mathcal{M}$ is the integral model over $\Z$ of the associated GSpin Shimura variety $M$ and $\widehat{\mathcal{Z}}(m):=(\cZ(m),\Phi_m)$ are the arithmetic special divisor on $\cM$ for $m\in \Z_{>0}$. 
Throughout this section and the rest of the paper, we assume that the equation $Q(v)=m$ has a solution in $L$, i.e. $\cZ(m)\neq \emptyset$.
As in \Cref{main_sp_end}, $\cY\in\cM(\cO_K)$ such that $\cY_K$ is Hodge-generic and $\cY^\sigma\in M(\C)$ via $\sigma: K\hookrightarrow \C$.


The main goal of this section is to give a first estimate of the archimedean term in the height formula \eqref{intersectionformula2}; more precisely, we show that for a fixed $\sigma$ and for every $m$,
\begin{align}\label{Phibound1}
    \Phi_m(\cY^\sigma)\asymp -m^{\frac{b}{2}}\log m + A(m, \cY^\sigma) + o(m^{\frac{b}{2}}\log m),
\end{align}
where $A(m, \cY^\sigma)$ is a non negative real number (see \eqref{Aofm} and \Cref{arch1} for the precise statement).

An important consequence (\Cref{summary}) of this estimate is the following uniform diophantine bounds. For a fixed finite place $\fP$ and a fixed $\sigma$, we have
\begin{align}\label{Diopbound}
    (\cY.\cZ(m))_\fP=O(m^{\frac{b}{2}}\log m),\quad \Phi_m(\cY^\sigma)=O(m^{\frac{b}{2}}\log m).
\end{align}
This consequence is one of the key inputs for the estimates in \S\S\ref{sec_arch}-\ref{sec_finite}.




Throughout this section, $x$ will denote a $\C$-point of $M$, which is not contained in any special divisor. This section is organized as follows. First, in \S\ref{explicitPhi}, we follow Bruinier \cite{bruinier} and Bruinier--K\"uhn \cite{bruinierintegrals} to express  $\Phi_m(x)$ as a sum of two terms 
$$\Phi_m(x)=\phi_m(x)-b'_m(k/2),$$
where $\phi_m(x)$ and the function $b_m(s)$ are defined in \eqref{phi}, \eqref{limit} and \eqref{bfunction}. Then in \S\ref{Green_main}, we use results from \cite{bruinierintegrals} to prove that $b'_m(k/2)\asymp m^{\frac{b}{2}}\log m$. Next, in \S\ref{sec_Green_phi}, we prove that $\phi_m(x)= A(m,x)+O(m^{\frac{b}{2}})$, where $A(m,x)$, as above, is non-negative. In \S \ref{summaryofsecfive}, we put together the results of \S\S\ref{explicitPhi}-\ref{sec_Green_phi} to deduce \eqref{Phibound1} and \eqref{Diopbound}.






\subsection{Bruinier's explicit formula for the Green function $\Phi_m$}\label{explicitPhi}

There is an another expression for the Green function $\Phi_m$ introduced in \S\ref{arithmetic} due to Bruinier (see \cite[\S 2]{bruinier} and \cite[\S 4]{bruinierintegrals}); this expression will allow us later to make explicit computations. As in \S\ref{Eisenstein}, let $k=1+\frac{b}{2}$, and $s\in\C$ with $\mathrm{Re}(s)>\frac{k}{2}$. We pick a lift of $x\in M(\C)$ to the period domain $D_L$ and still use $x$ to denote the lift. Recall from \S\ref{gspin_Q} that $x$ defines a negative definite plane\footnote{Using the notation in \S\ref{gspin_Q}, for a point $[z]\in D_L$ with $z=u+iw, u,w\in L_{\R}$, we have a negative definite plane given by $\spn_{\R}\{u,w\}$.} $P_x$ of $L_{\R}$ and for $\lambda\in L_{\R}$, we denote by $\lambda_x$ the orthogonal projection of $\lambda$ on $P_x$. Let $$F(s,z)=H\left(s-1+\frac{k}{2},s+1-\frac{k}{2},2s;z\right),\text{ where } H(a,b,c;z)=\sum_{n\geq 0}\frac{(a)_n(b)_n}{(c)_n}\frac{z^n}{n!}$$ is the Gauss hypergeometric function as in \cite[Chapter 15]{handbook}, and $(a)_n=\frac{\Gamma(a+n)}{\Gamma(a)}$ for $a,b,c,z\in \C$ and $|z|<1$.
Finally, let\footnote{In \cite[Section 2.2, (2.15)]{bruinier}, $\phi_{m}(x,s)$ is defined as a regularized theta lift of $F_{0,m}$; here the regularization process is slightly different from Borcherds version.}  
\begin{align}\label{phi}
\phi_{m}(x,s)=2\frac{\Gamma(s-1+\frac{k}{2})}{\Gamma(2s)}\sum_{Q(\lambda)=m, \lambda\in L}\left(\frac{m}{m-Q(\lambda_x)}\right)^{s-1+\frac{k}{2}} F\left(s,\frac{m}{m-Q(\lambda_x)}\right).
\end{align}
 By \cite[Proposition 2.8, Theorem 2.14]{bruinier}, the function $\phi_{m}(x,s)$ admits a meromorphic continuation to $\mathrm{Re}(s)>1$ with a simple pole at $s=\frac{k}{2}$ with residue $-c(m)$, where $c(m)$ is the Fourier coefficient defined in \S\ref{Eisenstein}, see also \cite[Proposition 4.3]{bruinierintegrals} for the value of the residue. 

We regularize $\phi_m(x,s)$ at $s=k/2$ by defining $\phi_{m}(x)$ to be the constant term at $s=\frac{k}{2}$ of the Laurent expansion of $\phi_{m}(x,s)$. As in \cite[Prop.4.2]{bruinierintegrals}, for $x\in D_L$, we have  
\begin{align}\label{limit}
\phi_m(x)=\lim_{s\rightarrow \frac{k}{2}}\left(\phi_{m}(x,s)+\frac{c(m)}{s-\frac{k}{2}}\right).
\end{align}

To compare $\phi_m(x)$ with $\Phi_m(x)$, we recall that $C(n,s), n\in \Z, s\in \C, \re(s)>1-\frac{k}{2}$ is part of the Fourier coefficient of $E_0(\tau,s)$ defined in \S\ref{Eisenstein}.
For $s\in \C$ with $\mathrm{Re}(s)>1$, define\footnote{In the notation of \cite{bruinierintegrals}, it is $b(0,0,s)$ in Equation (4.12) {\it loc.cit.}. The comparison with the formula given above is given in \cite[(4.20)]{bruinierintegrals}. In \cite[Theorem 1.9]{bruinier},  $b(s)$ is  defined as the coefficient of $\gamma=0, n=0$ in the Fourier expansion of $F_{0,m}(\cdot,s)$.}
\begin{align}\label{bfunction}
b_m(s)=-\frac{C\left(m,s-\frac{k}{2}\right)\cdot\left(s-1+\frac{k}{2}\right)}{\left(2s-1\right)\cdot\Gamma\left(s+1-\frac{k}{2}\right)}.
\end{align}
By \cite[Theorem 1.9]{bruinier}, $b_m(s)$ is a holomorphic function of $s$ in the region $\mathrm{Re}(s)>1$.

\begin{proposition}[{\cite[Proposition 2.11]{bruinier}}]\label{comparison}
For $x\in D_{L}$, we have: 
$$\Phi_m(x)=\phi_m(x)-b'_m(k/2).$$
\end{proposition}
This proposition shows in particular that $\phi_{m}$ is also a Green function for the arithmetic cycle $\mathcal{Z}(m)$.



\subsection{Estimating $b'_m(k/2)$}\label{Green_main}
The main result of this subsection is the following.
\begin{proposition}\label{logterm}
Let $m\in\Z_{>0}$ such that $\cZ(m)\neq \emptyset$ and if $b$ is odd, we further assume that for a fixed $D$, $\sqrt{m/D}\in \Z$ as in \Cref{main_sp_end}. Then as $m\rightarrow +\infty$, we have:
$$b'_m(k/2)=|c(m)|\log m+o\left(c(m)\log m\right).$$
In particular, $b'_m(k/2)\asymp m^{\frac{b}{2}}\log m$.
\end{proposition}

\Cref{lem_BK} below reduces the proposition into computations of certain local invariants of the lattice $L$ at primes $p$.


In order to state the lemma, we recall some notations from \cite{bruinierintegrals} and \S\ref{sec_density}. Recall $r=b+2$ and $D$ be the fixed integer in \Cref{main_sp_end}; let $d$ be \begin{align*}
&(-1)^{\frac{r}{2}}\det(L),\, \text{if $r$ is even};\\
&2(-1)^{\frac{r+1}{2}}D\det(L),\, \text{otherwise},
\end{align*}
where $\det(L)$ denote the Gram determinant of $L$.
Let $d_0$ denote the fundamental discriminant of number field $\Q(\sqrt{d})$ and let $\chi_{d_{0}}$ be the quadratic character associated to $d_0$. The polynomial $\mathrm{L}_m^{(p)}(t)$ is defined by 
$$\mathrm{L}_m^{(p)}(t)=N_{m}(p^{w_p})t^{w_p}+(1-p^{r-1}t)\sum_{n=0}^{w_p-1}N_{m}(p^n)t^n\in\Z[t],$$
where, as in \S\ref{sec_density}, $N_{m}(p^n)=\#\{v\in L/p^nL;\,Q(v)\equiv m\pmod{p^n}\}$ and $w_p:=w_p(m)=1+\val_{p}(m)$ for $p\neq 2$ and $w_2:=w_2(m)=1+2\val_2(2m)$.\footnote{In \cite[(3.18), (3.20)]{bruinierintegrals}, $w_p$ is defined to be $1+2\val_p(2m)$ for every prime $p$. Our definition of $w_p$ for odd prime $p$ and the definition in \cite{bruinierintegrals} give the same definition of $L^{(p)}_m(t)$ by the fact that $N_m(p^{n+1})=p^{r-1}N_m(p^n)$ for all $n\geq 1+ \val_p(m)$ (\Cref{explicit_alpha}(1)) and a direct computation (see for instance \cite[(21), (22)]{bruinierkuss}).} 

\begin{lemme}[Bruinier--K\"uhn]\label{lem_BK}
Let $D\in \Z_{>0}$ be the fixed integer in \Cref{main_sp_end}, for all $m\in \Z_{>0}$ such that $\sqrt{m/D}\in \Z$ (and representable by $(L,Q)$), we have
\[\frac{b'_m\left(\frac{k}{2}\right)}{b_m\left(\frac{k}{2}\right)}=\log(m)+2\frac{\sigma_{m}'(k)}{\sigma_{m}(k)}+O(1),\]
where for $s\in\C$, $\mathrm{Re}(s)>0$ the function $\sigma_{m}$ is given by:   
\begin{align}\label{sigma}
\sigma_{m}(s)=\left\{\begin{array}{ll}
\underset{p\mid 2m\det(L)}{\prod}\frac{\mathrm{L}_m^{(p)}\left(p^{1-\frac{r}{2}-s}\right)}{1-\chi_{d_{0}}(p)p^{-s}},\,\textrm{if}\, r\,\textrm{is even},\\
\underset{p\mid 2m\det(L)}{\prod}\frac{1-\chi_{d_{0}}(p)p^{\frac{1}{2}-s}}{1-p^{1-2s}}\mathrm{L}_m^{(p)}\left(p^{1-\frac{r}{2}-s}\right),\, \textrm{if}\, r\,\textrm{is odd}.
\end{array}
\right.
\end{align}
\end{lemme}
\begin{proof}
Taking logarithmic derivatives in \eqref{bfunction} at $s=\frac{k}{2}$ yields:
\begin{align*}
\frac{b'_m\left(\frac{k}{2}\right)}{b_m\left(\frac{k}{2}\right)}=\frac{C'(m,0)}{C(m,0)}-\frac{2}{b}-\Gamma'(1)
\end{align*}
Then we conclude by \cite[Theorem 4.11, (4.73), (4.74)]{bruinierintegrals}, since both $d_0$ and $k$ are independent of $m$. (Our definition of $d$ above differs from the definition in \cite{bruinierintegrals} by $m/D$, which is a square, and hence yields the same $d_0$.)
\end{proof}

\begin{proof}[Proof of \Cref{logterm}]
By definition, $b_m(k/2)=-c(m)=|c(m)|$. Therefore, by \Cref{lem_BK},
it is enough to show that $$\frac{\sigma_{m}'(k)}{\sigma_{m}(k)}=o(\log(m)).$$ Taking the logarithmic derivative in \eqref{sigma} at $s=k$, we get for $r$ even
\begin{align*}
\frac{\sigma_{m}'(k)}{\sigma_{m}(k)}=-\sum_{p\mid 2m\det(L)}\left(\frac{p^{1-r}\mathrm{L}_m^{(p)'}\left(p^{1-r}\right)}{\mathrm{L}_m^{(p)}\left(p^{1-r}\right)}+\frac{\chi_{d_0}(p)}{p^k-\chi_{d_0}(p)}\right)\log (p),
\end{align*} 
and for $r$ odd 
\begin{align*}
\frac{\sigma_{m}'(k)}{\sigma_{m}(k)}=-\sum_{p\mid 2m\det(L)}\left(\frac{p^{1-r}\mathrm{L}_m^{(p)'}\left(p^{1-r}\right)}{\mathrm{L}_m^{(p)}\left(p^{1-r}\right)}-\frac{\chi_{d_0}(p)}{p^{k-\frac{1}{2}}-\chi_{d_0}(p)}+\frac{2}{p^{2k-1}-1}\right) \log(p).
\end{align*}

Since $k=1+\frac{b}{2}\geq \frac{5}{2}$, we have 
\begin{align*}
\left|\sum_{p\mid 2m\det(L)}\frac{\chi_{d_0}(p)\log(p)}{p^k-\chi_{d_0}(p)}\right|\leq \sum_{p}\frac{\log(p)}{p^{5/2}-1}<+\infty,
\end{align*}
\begin{align*}
\left|\sum_{p\mid 2m\det(L)}\frac{\chi_{d_0}(p)\log(p)}{p^{k-\frac{1}{2}}-\chi_{d_0}(p)}\right|\leq \sum_{p}\frac{\log(p)}{p^{2}-1}<+\infty,
\end{align*}
\begin{align*}
\left|\sum_{p\mid 2m\det(L)}\frac{2\log p}{p^{2k-1}-1}\right|\leq \sum_{p}\frac{2\log(p)}{p^{4}-1}<+\infty.
\end{align*}
Hence it remains to treat the $L^{(p)}_m$ term.
We have $\mathrm{L}^{(p)}_{m}(p^{1-r})=N_{m}(p^{w_p})p^{(1-r)w_p}$ and $$\mathrm{L}^{(p)'}_{m}(p^{1-r})=w_pN_{m}(p^{w_p})p^{(1-r)(w_p-1)}-\sum_{n=0}^{w_p-1}N_{m}(p^n)p^{(n-1)(1-r)}.$$
Hence 
\begin{align*}
\left|\frac{p^{1-r}\mathrm{L}_m^{(p)'}\left(p^{1-r}\right)}{\mathrm{L}_m^{(p)}\left(p^{1-r}\right)}\right|=\left|w_p-\sum_{n=0}^{w_p-1}\frac{N_{m}(p^n)}{N_{m}(p^{w_p})}p^{(n-w_p)(1-r)}\right|=\left|w_p-\sum_{n=0}^{w_p-1}\frac{\mu_{p}(m,n)}{\mu_{p}(m,w_p)}\right|\leq \frac{C}{p},
\end{align*}
where $\mu_{p}(m,n)=p^{-n(r-1)}N_{m}(p^n)$ as in \S\ref{sec_density} and the last inequality follows from \Cref{count} with constant $C$ only depends on $(L,Q)$ (i.e., is independent of $m,p$).
Thus we have 
\begin{align*}
\left|\sum_{p\mid 2m\det(L) }\frac{p^{1-r}\mathrm{L}_m^{(p)'}\left(p^{1-r}\right)}{\mathrm{L}_m^{(p)}\left(p^{1-r}\right)}\right|\leq C\sum_{p\mid 2m\det(L)}\frac{\log(p)}{p}=O(\log\log(m)).
\end{align*}
Here we use the fact that for $N\geq 2$,  $\sum_{p\mid N}\frac{\log(p)}{p}=O(\log\log(N))$.
Indeed, let $X=\log(N)$ and use Mertens' first theorem to write 
\begin{align*}
\sum_{p\mid N}\frac{\log(p)}{p}&=\sum_{p\mid N,p<X}\frac{\log(p)}{p}+\sum_{p\mid N,p\geq X}\frac{\log(p)}{p}\leq \log(X)+\frac{1}{X}\sum_{p\mid N}\log(p)+O(1)\\
&\leq \log(X)+\frac{\log(N)}{X}+O(1)\leq \log(\log(N))+O(1).
\end{align*}
This concludes the proof of the proposition.
\end{proof}


\subsection{Estimates on $\phi_m(x)$}\label{sec_Green_phi}
Recall that $x\in M(\C)$ is a Hodge-generic point and we pick a lift of $x$ to $D_L$. We will associate the quantity $A(m,x)$ to $x$, which is independent of the choice of the lift. 
Thus, we will also denote this lift by $x$.
Recall from \S\ref{explicitPhi}, for $\lambda\in L_{\R}$, we use $\lambda_x$ to denote the orthogonal projection of $\lambda$ onto the negative definite plane in $L_{\R}$ associated to $x$. Define
\begin{equation}\label{Aofm}
    A(m,x):=-2\sum_{\underset{|Q(\lambda_x)|\leq 1,Q(\lambda)=1}{\sqrt{m}\lambda\in L}}\log (|Q(\lambda_x)|).
\end{equation}
Note that since $x$ is Hodge-generic, for any $\lambda\in L$, $\lambda_x\neq 0$. Hence for any $\lambda\in L_{\R}$ such that $\R\lambda \cap L\neq \{0\}$, we also have $\lambda_x\neq 0$. On the other hand, the conditions $|Q(\lambda_x)|\leq 1,Q(\lambda)=1$ cut out a compact region in $L_{\R}$ and hence for a fixed $m$, $A(m,x)$ is the sum of finitely many terms. Therefore $A(m,x)$ is well-defined and non-negative.

The main purpose of this subsection is to prove the following result.
\begin{proposition}\label{twoterms}
For $m\in\Z_{>0}$, we have
$$\phi_{m}(x)=A(m,x) +O(m^{\frac{b}{2}}).$$
\end{proposition}

Recall $F(s,t)$ from \S\ref{explicitPhi}.
Since $F(s,0)=1$, for $z\in \C$ with $|z|<1$, we may write $F(s,z)=zG(s,z)+1$. Recall that we set $k=1+\frac{b}{2}$.
From the definitions, we obtain the following decomposition of $\phi_{m}(x)$.
\begin{equation}\label{decomposition}
\begin{array}{rcl}
\phi_{m}(x)&\overset{\eqref{limit}}{=}&\displaystyle
\lim_{s\to\frac{k}{2}}\Bigl(\phi_m(x,s)+\frac{c(m)}{s-\frac{k}2}
\Bigr)
\\[.2in]&\overset{\eqref{phi}}{=}&\displaystyle
\lim_{\underset{\re s>0}{ s\to 0}}\left(\frac{c(m)}{s}+\frac{4}{b}\sum_{\underset{Q(\lambda)=m}{\lambda\in L}}\Bigl(\frac{m}{m-Q(\lambda_x)}\Bigr)^{k-1+s}
F\Bigl(\frac{k}{2}+s,\frac{m}{m-Q(\lambda_x)}\Bigr)\right)
\\[.3in]&=&\displaystyle \widetilde{\phi}_m(x,0)+
\lim_{\underset{\re s>0}{ s\to 0}}R_x(s,m),
\end{array}
\end{equation}
where for $s\in \C$ with $\re s>0$, we define
\begin{equation*}
\begin{array}{rcl}
\displaystyle\widetilde{\phi}_{m}(x,s)
&=&
\displaystyle\frac{4}{b}\sum_{\underset{Q(\lambda)=1}{\sqrt{m}\lambda\in L}}
\Bigl(\frac{1}{1-Q(\lambda_x)}\Bigr)^{k+s}G\Bigl(\frac{k}{2}+s,\frac{1}{1-Q(\lambda_x)}\Bigr), 
\\[.2in]
\displaystyle R_x(s,m)&=&
\displaystyle\frac{c(m)}{s}+\frac{4}{b}\sum_{\underset{Q(\lambda)=1}{\sqrt{m}\lambda\in L}}\Bigl(\frac{1}{1-Q(\lambda_x)}\Bigr)^{k-1+s};
\displaystyle 
\end{array}
\end{equation*}
in \S\ref{tildephi}, we will prove that the above series defining $\widetilde{\phi}_m(x,s)$ indeed converges uniformly absolutely in a small compact neighborhood of $s=0$ in $\C$ and hence it defines a function holomorphic at $0$ and we still denote this function by $\widetilde{\phi}_m(x,s)$. In particular, $\widetilde{\phi}_m(x,0)$ is well defined and equals to $\lim_{\re s>0, s\rightarrow 0}\widetilde{\phi}_m(x,s)$. Therefore, the last equality in \eqref{decomposition} is valid; moreover, since $\phi_m(x,s+\frac{k}{2})+\frac{c(m)}{s}$ admits a holomorphic continuation to $s=0$ (see \S\ref{explicitPhi}), then $R_x(s,m)$ admits a holomorphic continuation to $s=0$. We use $R_x(0,m)$ to denote 
$\lim_{\re s>0, s\to 0}R_x(s,m)=\lim_{s\in \R_{>0}, s\to 0}R_x(s,m)$ and rewrite \eqref{decomposition} as
\begin{equation}\label{decom}
\phi_m(x)=\widetilde{\phi}_{m}(x,0)+R_x(0,m).
\end{equation}
Note that the second equality in \eqref{decomposition} also use the fact that the ratio of $\Gamma$-functions in \eqref{phi} is holomorphic and has limit to be $2/b$.

In what follows next, we estimate $R_x(0,m)$ and $\widetilde{\phi}_{m}(x,0)$ using results from \S\ref{Vol-Nie}, where we use the work of Heath-Brown and Niedermowwe to estimate the number of the lattice points $\lambda$ in certain regions in $L$ with $Q(\lambda)=m$.

\subsubsection{Bounding $R_x(0,m)$}
We only consider $s\in\R_{\geq 0}$. Recall from \S\ref{Vol-Nie}, for $\lambda\in L_{\R,1}=\{\lambda\in L_\R:Q(\lambda)=1\}$, we set $\displaystyle h_s(\lambda)=\left(\frac{1}{1-Q(\lambda_x)}\right)^{k-1+s}$;
$\displaystyle a(m)=\frac{-c(m)\Gamma(k)\sqrt{|L^\vee/L|}}{2(2\pi)^{k}}>0$ and $\mu_\infty$ denotes the measure on $L_{\R,1}$ defined in \eqref{def_muinfty}.

For $s>0$, \Cref{hintegral} yields the equality $$R_x(s,m)=\frac{4}{b}\sum_{\underset{Q(\lambda)=1}{\sqrt{m}\lambda\in L}}h_s(\lambda)-\frac{4a(m)}{b}\int_{L_{\R,1}}h_s(\lambda)d\mu_{\infty}(\lambda).$$
Our next result provides the required bound on $R_x(0,m)$.

\begin{proposition}\label{proprxsm}
For a given $x\in D_L$ Hodge-generic, we have
$$R_x(0,m)=\lim_{s\in \R_{>0}, s\rightarrow 0} R_x(s,m) =O(m^{\frac{b}{2}}).$$
\end{proposition}

\begin{proof}
Fix $\epsilon>0$ and we only consider $s\in [0,\epsilon]$.
As in \S\ref{Vol-Nie}, $\Omega_T=\{\lambda\in L_{\R,1}:-Q(\lambda_x)\in[0,T]\}$. For an integer $N\geq 0$, define the set $\Theta_N:=\Omega_{N+1}\backslash\Omega_N$, and note that for $\lambda\in\Theta_N$, we have 
$$
h_s(\lambda)=\frac{1}{(N+1)^{s+k-1}}+O\Bigl(\frac{1}{(N+1)^{s+k}}\Bigr).
$$ 
To obtain a uniformly absolutely convergent expression for $R_x(s,m)$ for $s\in ]0, \epsilon]$, we write
\begin{equation*}
\frac{b}{4}R_x(s,m)=
\sum_{N=0}^\infty
\Bigl(\sum_{\substack{\lambda\in\Theta_N \\\sqrt{m}\lambda\in L}} h_s(\lambda)-a(m)\int_{\Theta(N)}h_s(\lambda)d\mu_{\infty}(\lambda)\Bigr).
\end{equation*}
We use the above estimate of $h_s(\lambda)$, for $\lambda\in\Theta_N$, and bound the associated error term using \Cref{effective} and \Cref{secondmeasure}. More precisely, let $C>0$ be an absolute constant such that for all $\lambda\in \Theta_N$, for all $s\in[0,\epsilon]$, we have $|h_s(\lambda)-(N+1)^{-(s+k-1)}|<C(N+1)^{-(s+k)}$. Then the error term is bounded by 
\begin{equation}\label{partialestimate}
    \begin{array}{rcl}
& & \displaystyle C \sum_{N=0}^{\infty} (N+1)^{-(s+k)} \Bigl(\sum_{\substack{\lambda\in\Theta_N \\\sqrt{m}\lambda\in L}}1 + a(m)\mu_\infty(\Theta_N)\Bigr) 
\\& = & \displaystyle C\sum_{N=1}^\infty (N^{-(s+k)}-(N+1)^{-(s+k)})\Bigl(\sum_{\substack{\lambda\in\Omega_N \\\sqrt{m}\lambda\in L}}1 + a(m)\mu_\infty(\Omega_N)\Bigr) 
\\  & \ll & \displaystyle a(m) \sum_{N=1}^\infty N^{-(s+k+1)}\cdot N^{\frac{b}{2}} = a(m) \sum_{N=1}^\infty N^{-s-2} \ll a(m)=O(m^{\frac{b}{2}}),
\end{array}
\end{equation}
where the first equality is partial summation and the implicit constants above are absolute.

To bound the main term, we pick some $A>1$ in \Cref{effective}.
By partial summation, we write

\begin{equation*}
\begin{array}{rcl}
&&\displaystyle
\sum_{N=0}^{\infty}\frac{1}{(N+1)^{s+k-1}}
\Bigl(\sum_{\substack{\lambda\in\Theta_N \\\sqrt{m}\lambda\in L}}1
-a(m)\mu_\infty(\Theta_N)\Bigr)
\\[.2in]&=& \displaystyle
\sum_{N=1}^{\infty}
\Bigl(\frac{1}{N^{s+k-1}}-\frac{1}{(N+1)^{s+k-1}}\Bigr)
\Bigl(\sum_{\substack{\lambda\in\Omega_N \\\sqrt{m}\lambda\in L}}1
-a(m)\mu_\infty(\Omega_N)\Bigr)
\\[.2in]&\ll&\displaystyle
\displaystyle
\sum_{N=1}^{\infty}N^{-(s+k)}
\cdot N^{\frac{b}{2}}m^{\frac{b}{2}}(\log mN)^{-A}\leq m^{\frac{b}{2}}\sum_{N=1}^\infty N^{-1}(\log N)^{-A}\ll m^{\frac{b}{2}},
\end{array}
\end{equation*}
which is again sufficient. The proposition follows.
\end{proof}

\subsubsection{Estimating $\widetilde{\phi}_{m}(x,0)$}\label{tildephi}
We fix an $\epsilon\in ]0,1/2]$ and consider $s\in \C$ such that $|s|\leq \epsilon$. Since $$\displaystyle G\left(s+\frac{k}{2},z\right)=\sum_{n=1}^\infty\frac{\Gamma(s+k-1+n)\Gamma(s+1+n)\Gamma(2s+k)}{\Gamma(s+k-1)\Gamma(s+1)\Gamma(2s+k+n)}\frac{z^{n-1}}{n!}$$ converges uniformly absolutely for $|s|\leq \epsilon, |z|\leq 1/2$, then $G\left(s+\frac{k}{2},z\right)$ is absolutely bounded for such $s$ and $z$.

To show that $\widetilde{\phi}_m(x,s)$ is holomorphic in $|s|<\epsilon$, 
we write
\begin{equation*}
\begin{array}{rcl}
\displaystyle\frac{b}{4}\widetilde{\phi}_m(x,s)&=&
\displaystyle
\sum_{\substack{\sqrt{m}\lambda\in L\\Q(\lambda)=1\\ |Q(\lambda_x)|>1}}
\Bigl(\frac{1}{1-Q(\lambda_x)}\Bigr)^{k+s}G\Bigl(\frac{k}{2}+s,\frac{1}{1-Q(\lambda_x)}\Bigr)
\\[.2in]&+&\displaystyle
\sum_{\substack{\sqrt{m}\lambda\in L\\Q(\lambda)=1\\|Q(\lambda_x)|\leq 1}}
\Bigl(\frac{1}{1-Q(\lambda_x)}\Bigr)^{k+s}G\Bigl(\frac{k}{2}+s,\frac{1}{1-Q(\lambda_x)}\Bigr),
\end{array}
\end{equation*}
where, similar to the definition of $A(m,x)$ in \eqref{Aofm}, the second term is a finite sum of holomorphic functions. Thus we only need to show the first term converges absolutely uniformly on $|s|\leq \epsilon$.
The uniform absolute convergence follows from the boundedness of $G(s+k/2,z)$ in conjunction with an argument identical to \eqref{partialestimate}, which indeed implies that 

\[\sum_{\substack{\sqrt{m}\lambda\in L\\Q(\lambda)=1\\ |Q(\lambda_x)|>1}}
\Bigl(\frac{1}{1-Q(\lambda_x)}\Bigr)^{k+s}G\Bigl(\frac{k}{2}+s,\frac{1}{1-Q(\lambda_x)}\Bigr)=O(m^{\frac{b}{2}}).\]
Therefore, 
\[\widetilde{\phi}_m(x,0)=\frac{4}{b}
\sum_{\substack{\sqrt{m}\lambda\in L\\Q(\lambda)=1\\|Q(\lambda_x)|\leq 1}}
\Bigl(\frac{1}{1-Q(\lambda_x)}\Bigr)^{k}G\Bigl(\frac{k}{2},\frac{1}{1-Q(\lambda_x)}\Bigr)+O(m^{\frac{b}{2}}).\]

Obtaining an estimate on the terms with $|Q(\lambda_x)|\leq 1$ is significantly more difficult since the function
$$\lambda\mapsto \left(\frac{1}{1-Q(\lambda_x)}\right)^{k}G\left(\frac{k}{2},\frac{1}{1-Q(\lambda_x)}\right)$$ has a logarithmic singularity along $\{\lambda\in L_{\R,1},\, x\in \lambda^{\bot}\}$.
Since $\displaystyle G(k/2,z)=\sum_{n=1}^\infty\frac{k-1}{n+k-1}z^{n-1}$, it follows that there exists an absolute constant $C>0$ such that for $z\in [1/2, 1[$,
\begin{align*}
\left|z^{k}G\left(\frac{k}{2},z\right)+\frac{b}{2}\log(1-z)\right| \leq C.
\end{align*}
Hence, noting that $\displaystyle\sum_{\substack{|Q(\lambda_x)|\leq 1 \\\sqrt{m}\lambda\in L}}1=O(m^{\frac{b}{2}})$, which follows from either \Cref{HBforL} or \Cref{effective}, we have
\begin{align*}
\frac{4}{b}\sum_{\substack{\sqrt{m}\lambda\in L\\Q(\lambda)=1\\|Q(\lambda_x)|\leq 1}}
\Bigl(\frac{1}{1-Q(\lambda_x)}\Bigr)^{k}G\Bigl(\frac{k}{2},\frac{1}{1-Q(\lambda_x)}\Bigr)&=-2\sum_{\substack{\sqrt{m}\lambda\in L\\Q(\lambda)=1\\|Q(\lambda_x)|\leq 1}}\log\Bigl(\frac{-Q(\lambda_x)}{1-Q(\lambda_x)}\Bigr)+O(m^{\frac{b}{2}}) \\
&=-2\sum_{\substack{\sqrt{m}\lambda\in L\\Q(\lambda)=1\\|Q(\lambda_x)|\leq 1}}\log (|Q(\lambda_x)|)+O(m^{\frac{b}{2}}).\\
\end{align*}

Therefore, we have proved the following proposition.
\begin{proposition}\label{propwpmx}
We have
$$\widetilde{\phi}_m(x,0)=A(m,x)+O(m^{\frac{b}{2}}).$$
\end{proposition}
Proposition \ref{twoterms} follows immediately from \eqref{decom}, and Propositions \ref{proprxsm} and \ref{propwpmx}.

\subsection{Conclusions.}\label{summaryofsecfive}
 The following theorems summarize the results proved in the previous subsections. First, we remind the reader that $|c(m)| \asymp m^{b/2}$. 
 
 \begin{theorem}\label{arch1}
 For every $m$ representable by $(L,Q)$, we have
 \[\Phi_m(\cY^\sigma)=c(m)\log m + A(m, \cY^\sigma)+o(|c(m)|\log m).\]
 \end{theorem}
 \begin{proof}
 Combine \Cref{comparison,logterm,twoterms}.
 \end{proof}

\begin{theorem}\label{summary}
For every positive integer $m$, we have the following bounds: 
\begin{enumerate}
    \item[(i)] $0\leq A(m,\cY^\sigma)\ll m^{\frac{b}{2}} \log m$.
    
    \item[(ii)] $(\cY.\cZ(m))_{\fP}\log|\mathcal{O}_{K}/\mathfrak{P}| \ll m^{\frac{b}{2}}\log m$. 
    
\item[(iii)] $\displaystyle \sum_{\sigma:K\hookrightarrow\C}\frac{\Phi_{m}(\cY^\sigma)}{|\mathrm{Aut}(\cY^\sigma)|} = O(m^{\frac{b}{2}}\log m)$.
\end{enumerate}

\end{theorem}
\begin{proof}
Recall that $\Phi_m(\cY^{\sigma}) = \phi_m(\cY^{\sigma}) - b_m'(k/2)$. For every integer $m>0$, Equation \eqref{intersectionformula2}, Proposition \ref{eq1} and Proposition \ref{logterm} yields 
$$\displaystyle  \sum_{\fP} (\cY.\cZ(m))_{\fP}\log|\mathcal{O}_{K}/\mathfrak{P}| + \sum_{\sigma: K \hookrightarrow \C} \phi_m(\cY^{\sigma}) \asymp m^{\frac{b}{2}} \log m. $$

By Proposition \ref{twoterms}, we have $\phi_m(\cY^{\sigma}) = A(m, \cY^\sigma) + O(m^{\frac{b}{2}})$, where $A(m, \cY^\sigma)$ is a sum of positive quantities, and is therefore non-negative. Further, note that $(\cY.\cZ(m))_{\fP}\log|\mathcal{O}_{K}/\mathfrak{P}|$ is also non-negative, for every prime $\fP$. Therefore, it follows that $A(m, \cY^\sigma) \ll m^{\frac{b}{2}} \log m$, $(\cY.\cZ(m))_{\fP}\log|\mathcal{O}_{K}/\mathfrak{P}| \ll m^{\frac{b}{2}}\log m$ and \emph{a fortiriori} that $\Phi_m(\cY^{\sigma}) = O(m^{\frac{b}{2}}\log m)$. 
\end{proof}




\section{Second step in bounding the archimedean contribution }\label{sec_arch}
We keep the notations from \S\ref{sec_Green_phi}. Namely $(L,Q)$ is a quadratic lattice of signature $(b,2)$. We are given a Hodge-generic point $x = \cY^{\sigma}$ in the Shimura variety $M(\C)$ and choose a lift of $x$ to the period domain $D_L$, which corresponds to a $2$-dimensional negative definite (with respect to $Q$) plane $P\subset L_{\R}$. Let $P^\perp$ denote the orthogonal complement of $P$ in $L_{\R}$. Then $P^\perp$ is a $b$-dimensional positive definite space. Given a vector $\lambda\in L_\R$, we let $\lambda_x$ and $\lambda_{x^\bot}$ denote the projections of $\lambda$ to $P$ and $P^\perp$, respectively. For $m\in\Z_ {>0}$, recall that we defined the quantity $A(m,x)$ in \eqref{Aofm} by 
\begin{equation*}
A(m,x) =-2\sum_{\underset{\underset{|Q(\lambda_x)|\leq1}{Q(\lambda)=1}}{\sqrt{m}\lambda\in L}}\log (|Q(\lambda_x)|)= \displaystyle2\sum_{\substack{\lambda\in L\\Q(\lambda)=m\\|Q(\lambda_x)|\leq m}}
  \log\Bigl(\frac{m}{|Q(\lambda_x)|}\Bigr).
\end{equation*}
As $x$ is fixed, we denote $A(m,x)$ simply by $A(m)$. Also, for a subset $S\subset \Z_{>0}$, the \emph{logarithmic asymptotic density} of $S$ is defined to be $\displaystyle \limsup_{X\rightarrow\infty}\frac{\log |S_X|}{\log X}$, where $S_X:=\{a\in S \mid X\leq a < 2X\}$. The main result of this section is the following bound on $A(m)$, for positive integers $m$ outside a set of zero logarithmic asymptotic density. Combined with \Cref{arch1}, we obtain the estimate for the archimedean term $\Phi_m(\cY^\sigma)\asymp - m^{\frac{b}{2}}\log m$ for such $m$'s.
\begin{theorem}\label{thAbound}
There exists a subset $S_{\bad} \subset \Z_{>0}$ of logarithmic asymptotic density zero such that for every $m\notin S_{\bad}$, we have
\begin{equation*}
    A(m)=o(m^{\frac{b}{2}}\log(m)).
\end{equation*}
\end{theorem}

To prove Theorem \ref{thAbound}, we write $A(m)=A_\mt(m)+A_\er(m)$, where we define the {\it main term} $A_{\mt}(m)$ and the {\it error term} $A_\er(m)$ to be
\begin{equation*}
\begin{array}{rcl}
\displaystyle A_\mt(m) &=& \displaystyle2\sum_{\substack{\lambda\in L\\Q(\lambda)=m\\1\leq |Q(\lambda_x)|\leq m}}
  \log\Bigl(\frac{m}{|Q(\lambda_x)|}\Bigr),\\[.3in]
\displaystyle A_\er(m) &=& \displaystyle2\sum_{\substack{\lambda\in L\\Q(\lambda)=m\\0<|Q(\lambda_x)|<1}}
  \log\Bigl(\frac{m}{|Q(\lambda_x)|}\Bigr).
%
\end{array}
\end{equation*}
In \S\ref{arch2main} and \S\ref{arch2error}, we obtain bounds for the two terms $A_\mt(m)$ and $A_\er(m)$, respectively. \Cref{thAbound} follows directly from \Cref{circlemethod,propdio}. Since we only consider the fixed quadratic form $Q$ in this section, we will not specify the dependence of the implicit constants on $Q$ when we apply the circle method results by Heath-Brown recalled in \S\S\ref{circle_method}-\ref{Vol-Nie}.



\subsection{Bounding the main term using the circle method}\label{arch2main}

In this section, we prove the following proposition.
\begin{proposition}\label{circlemethod}
We have
\begin{equation*}
\lim_{m\to\infty}
\frac{A_\mt(m)}{m^{\frac{b}{2}}\log m}=0.
\end{equation*}
\end{proposition}
\begin{proof}
Fix a real number $T>1$ and let $m>T$ be a positive integer. We break up $A_\mt(m)$ into two terms $A_1(m)$ and $A_2(m)$, where $A_1(m)$ is the sum over those $\lambda\in L$ such that $|Q(\lambda_x)|\geq m/T$, and $A_2(m)$ is the sum over those $\lambda\in L$ such that $|Q(\lambda_x)|< m/T$. We start by bounding $A_1(m)$.
First note that if $\lambda\in L$ with $Q(\lambda)=m$ and $m/T\leq |Q(\lambda_x)|\leq m$, then we have
$$\log\Bigl(\frac{m}{|Q(\lambda_x)|}\Bigr)\leq \log T.$$
Furthermore, by \Cref{HBforL},
\begin{equation*}
|\{\lambda \in L: Q(\lambda)=m,Q(\lambda_x)\leq m\}|\ll m^{\frac{b}{2}}.
\end{equation*}
Therefore, we obtain the bound
\begin{equation}\label{eqA1bound}
A_1(m)\ll m^{\frac{b}{2}}\log T.
\end{equation}

Next, we consider $A_2(m)$.
Once again, we apply \Cref{HBforL} to obtain the bound (take $\epsilon=1/4$)
\begin{equation}\label{eqA2bound}
\left|\Bigl\{\lambda\in L, Q(\lambda)=m, |Q(\lambda_x)|<\frac{m}{T}\Bigr\}\right|\ll \frac{m^{\frac{b}{2}}}{T} +O_T(m^\frac{b+2}{4}).
\end{equation}

Equations \eqref{eqA1bound}, \eqref{eqA2bound}, and the trivial bound $\log(m/|Q(\lambda_x)|)\leq\log m$
for $\lambda\in L$ with $1\leq |Q(\lambda_x)|\leq m$, yield
the following estimate on $A_\mt(m)=A_1(m)+A_2(m)$:
\begin{equation*}
A_\mt(m)\ll m^{\frac{b}{2}}\log T+m^{\frac{b}{2}}(\log m)T^{-1}+O_{T}(m^{(b+2)/4}\log m).
\end{equation*}
Dividing by $m^{\frac{b}{2}}\log(m)$ and letting $m$ tend to infinity, we see that
\begin{equation*}
    \limsup_{m\to\infty}
\frac{A_\mt(m)}{m^{\frac{b}{2}}\log m}\ll \frac{1}{T}.
\end{equation*}
Since this is true for every $T$, then $\displaystyle\lim_{m\rightarrow\infty} \frac{A_\mt(m)}{m^{\frac{b}{2}}\log m}$ exists and equals to $0$.
\end{proof}

\subsection{Bounding the error term using the diophantine bound}\label{arch2error}
We start with the following (entirely lattice theoretic) lemma.

\begin{lemma}\label{lemsphere}
Let $C>2$ be a fixed constant, $X>1$ be a real number, and let $N\geq 2$ be a positive integer.
Suppose $S$ is a set of $N$ vectors in $L_\R$ such that $C\leq Q(v)\ll X$ 
and $|Q(v_x)|< e^{-C}$ for all $v\in S$. Then there exist two distinct vectors $v$ and $v'$ in $S$ such
that their difference $w:=v-v'$ satisfies the following two
properties: 
\begin{enumerate}
\item $-\log(|Q(w_x)|)\gg \min\bigl(-\log(|Q(v_x)|),-\log(|Q(v'_x)|)\bigr)$
\item $Q(w_{x^\bot})\ll \frac{X}{N^{\frac{2}{b}}}$.
\end{enumerate}
All implicit constants here are absolute; in particular, they are independent of $C,X,N$.
\end{lemma}
\begin{proof}
The first property is immediate and is satisfied for every pair $v$ and $v'$. Indeed, we have $w_x=v_x-v'_x$ and using the
triangle inequality, it follows that
\begin{equation*}
|Q(w_x)|^{1/2}\leq 2\max(|Q(v_x)|^{1/2},|Q(v'_x)|^{1/2})<1,
\end{equation*}
since $C>2$. Hence we obtain the first claim. 

To obtain the second claim, remark that if $v\neq v'$, then $v_{x^\bot}\neq v_{x^\bot}'$ since otherwise $|Q(v-v')|=|Q(v_x-v'_x)|\in]0,1[$ and $v-v'\in L$. Thus, by considering the projections $v_{x^\bot}$ for $v\in S$, we obtain $N$ vectors in the $b$-dimensional real vector space
$P^\perp$. Let $2T$ be the smallest distance between the vectors $v_{x^\perp}$ for $v\in S$, where distance is taken with respect to the positive definite form $Q$ on $P^\perp$. By the triangle inequality, we have the trivial bound $T=O(X^{1/2})$.
Then the $N$ balls of radii $T$ around the points $v_{x^\bot}$ are disjoint, and all lie within the ball of radius $C_0\sqrt{X}$
around the origin, where $C_0$ is an absolute constant depending only on the absolute implicit constant in $Q(v)\ll X$. By comparing volumes, we obtain
\begin{equation*}
NT^b\ll X^{\frac{b}{2}},
\end{equation*}
from which it follows that $T\ll X^{\frac{1}{2}}/N^{\frac{1}{b}}$. Therefore, there exist two points $v$ and $v'$ in $S$ such that $$Q(v_{x^\bot}-v'_{x^\bot})\leq T^2\ll \frac{X}{N^{\frac{2}{b}}},$$ concluding the proof of the lemma.
\end{proof}

The following result controls the error term in $A(m)$ for most $m$.
\begin{proposition}\label{propdio}
Let $S_{\bad}\subset \N^\times$ be the set of integers $m$ such that \begin{equation}\label{eqbl1b}
A_\er(m)>m^{\frac{b}{2}}.
\end{equation}
Then $S_\bad$ has logarithmic asymptotic density zero.
\end{proposition}

\begin{proof}
A crucial ingredient in the proof is the ``uniform diophantine bound'' in \Cref{summary}~(i). Since $A(m)$ is a sum of the positive terms $\log\left(\frac{m}{|Q(\lambda_x)|}\right)$, each such term must also satisfy the same bound, i.e., for all $\lambda\in L$ such that $Q(\lambda)=m, |Q(\lambda_x)|\leq m$, we have $$\log\left(\frac{m}{|Q(\lambda_x)|}\right)\ll m^{\frac{b}{2}}\log m.$$

Let $\epsilon\in ]0,1[$ and $X>1$ and let $S_{\bad,X}=]X,2X]\cap S_\bad$. We pick a fixed constant $C\in [2,4]$ and break up the interval $]C,X^{\frac{b}{2}}]$ into the disjoint union of dyadic intervals $\cup_{i\in I}]Z_i,2Z_i]$ such that $|I|=O(\log(X))$. Define the three following subsets of $]X,2X]$.
\begin{enumerate}
    \item The set $B_{1,X}$ of $m$ such that 
    $$|\{\lambda\in L, Q(\lambda)=m, |Q(\lambda_x)|<1\}|\geq X^{\frac{b}{2}-\epsilon}.$$
    \item The set $B_{2,X}$ of $m$ such that there exists at least one element $\lambda\in L$ with $Q(\lambda)=m$ and $-\log(|Q(\lambda_x)|)\geq X^{\frac{b}{2}}.$
    \item The set $B_{3,X}$ of $m$ for which there exists an index $i_m\in I$ such that 
    $$|\{\lambda\in L, Q(\lambda)=m, -\log(|Q(\lambda_x)|)\in]Z_{i_m},2Z_{i_m}]\}|\geq \frac{X^{\frac{b}{2}-\epsilon}}{Z_{i_m}}.$$
\end{enumerate}
Notice that if $m\in ]X,2X]\backslash (B_{1,X}\cup B_{2,X} \cup B_{3,X})$, then we can write
\begin{align*}
    A_\er(m) &=2\sum_{\substack{\lambda\in L\\Q(\lambda)=m\\|Q(\lambda_x)|<1}}
  \log m + 2\sum_{\substack{\lambda\in L\\Q(\lambda)=m\\ -\log(|Q(\lambda_x)|)\in [0, C]}}
  |\log(|Q(\lambda_x)|)| + 2\sum_{\substack{\lambda\in L\\Q(\lambda)=m\\-\log(|Q(\lambda_x)|)\in ]C, X^{\frac{b}{2}}]}}
  |\log(|Q(\lambda_x)|)\\
    &\leq 2X^{\frac{b}{2}-\epsilon}\cdot(\log(m)+C)+2\sum_{i}\sum_{\underset{-\log|Q(\lambda_x)|\in ]Z_i,2Z_i]}{\lambda\in L,Q(\lambda)=m}}(|\log(|Q(\lambda_x)|)|)\\
    &\leq 2 X^{\frac{b}{2}-\epsilon}\log(2X)+ 8 X^{\frac{b}{2}-\epsilon}+4\sum_{i}\frac{X^{\frac{b}{2}-\epsilon}}{Z_i}\cdot Z_i\\
    &\leq 14 X^{\frac{b}{2}-\epsilon}\log(2X)\\
\end{align*}
One can thus find $X_\epsilon>1$ such that for $X>X_\epsilon$, we have $14X^{\frac{b}{2}-\epsilon}\log(2X)<X^{\frac{b}{2}}$. Hence for all $X>X_\epsilon$, for all $m\in ]X,2X]\backslash(B_{1,X}\cup B_{2,X}\cup B_{3,X})$, we get $A_\er(m)\leq m^\frac{b}{2}$.  
In other words, for $X>X_\epsilon$, $S_{\bad,X}\subset B_{1,X}\cup B_{2,X}\cup B_{3,X}$. We will obtain upper bounds on the cardinality of $B_{1,X}$, $B_{2,X}$ and $B_{3,X}$. 

\begin{enumerate}
    \item The volume of the region of elements $\lambda\in L_\R$ such that $X<Q(\lambda)\leq 2X$ and $|Q(\lambda_x)|<1$ is bounded by $O(X^{\frac{b}{2}})$, hence a geometry-of-numbers argument implies that the number of elements in $L$ in this region is bounded by $O(X^{\frac{b}{2}})$. More precisely, we break the set of such $\lambda$ into a finite union of subsets such that for any two elements $\lambda, \lambda'$ in each subset, we have $|Q(\lambda_x-\lambda'_x)|<1$; as $Q$ is negative definite on $P$ and $|Q(\lambda_x)|<1$, we may choose the subsets using $\lambda_x$ such that the total number of subsets is a constant only depending on $Q$. Therefore, as in the proof of \Cref{lemsphere}, for each subset, we count the $\lambda$ by counting $\lambda_{x^\perp}\in P^\perp$ and then apply a geometry-of-numbers argument on $P^\perp$.
    
    It thus follows that 
\begin{align*}
    |B_{1,X}|=O(X^\epsilon).
\end{align*}

\item Let $Y:=|B_{2,X}|\geq 1$ and for each $m\in B_{2,X}$, let $\lambda(m)$ be an element of $L$ such that  $Q(\lambda(m))=m$ and $-\log(|Q(\lambda(m)_x)|)\geq X^{\frac{b}{2}}$. By Lemma \ref{lemsphere}, we obtain a nonzero integer vector $\lambda$ in $L$ such that $-\log(|Q(\lambda_x)|)\gg X^{\frac{b}{2}}$ and $Q(\lambda_{x^{\bot}})\ll \frac{X}{Y^{\frac{2}{b}}}$.\footnote{Although \Cref{lemsphere} assumes $Y\geq 2$, the above statement is trivial when $Y=1$.} Let $M=Q(\lambda)$, and note that $M=Q(\lambda_x)+Q(\lambda_{x^\bot})\ll \frac{X}{Y^{\frac{2}{b}}}$.  \Cref{summary}(i)  implies
\begin{equation*}
X^{\frac{b}{2}}\ll -\log (|Q(\lambda_x)|)\ll A(M)\ll M^{\frac{b}{2}}\log M\ll \frac{X^{\frac{b}{2}}\log(X)}{Y}.
\end{equation*}
Therefore, we obtain $$|B_{2,X}|\ll \log(X).$$

\item
The set $B_{3,X}$ is included in the union of the subsets $B_{3,Z_i}, i\in I$  formed by the elements $m\in]X,2X]$ such that $$\left|\{\lambda\in L,\, Q(\lambda)=m,\, -\log(|Q(\lambda_x)|)\in ]Z_i,2Z_i]\}\right|\geq \frac{X^{\frac{b}{2}-\epsilon}}{Z_i}.$$ 
Suppose that $Y:=|B_{3,Z_i}|\geq 1$ for some $i\in I$. Then there at least $\lceil\frac{YX^{\frac{b}{2}-\epsilon}}{Z_i}\rceil$ vectors $\lambda\in L$ such that $Q(\lambda)\in ]X,2X]$ and $-\log(|Q(\lambda_x)|)\in ]Z_i,2Z_i]$. We use again Lemma \ref{lemsphere} to construct an integral nonzero vector $\lambda\in L$ such that $$-\log(|Q(\lambda_x)|)\gg Z_i\,\textrm{and}\, Q(\lambda_{x^\bot})\ll \frac{X^{\frac{2\epsilon}{b}} Z_i^{\frac{2}{b}}}{Y^{\frac{2}{b}}}.$$ Let $M$ denote again $Q(\lambda)$ and notice that $M\ll \frac{X^{2\epsilon/b} Z_i^{2/b}}{Y^{2/b}}$. 
Theorem \ref{summary}(i) implies that
\begin{equation*}
Z_i\ll -\log (|Q(\lambda_x)|)\ll A(M)\ll M^{\frac{b}{2}}\log M\ll \frac{X^{\epsilon}Z_i}{Y}.
\end{equation*}
Thus for every $i\in I$, we have $|B_{3,Z_i}|\ll X^{\epsilon}$. Summing over all $i\in I$ yields $$|B_{3,X}|\ll X^{\epsilon} \log X.$$ 

\end{enumerate}

Hence we conclude that $\log|S_{\bad,X}|\ll \epsilon \log X + \log \log X$. Thus $\displaystyle\limsup_{X\rightarrow\infty} \frac{\log|S_{\bad,X}|}{\log X}\leq \epsilon$. As the equality holds for every $\epsilon >0$, we get the desired result.
\end{proof}

\section{Bounding the contribution from a finite place with good reduction}\label{sec_finite}

We keep the notations from the beginning of \S\ref{mainproof}, namely $\cM$ is the integral model over $\Z$ of the GSpin Shimura variety associated to an even maximal quadratic lattice $(L,Q)$ with signature $(b,2), b\geq 3$; $\cY$ is an $\cO_K$-point in $\cM$ such that $\cY_K$ is Hodge-generic; $\cZ(m)$ denotes the special divisor over $\Z$ associated to an integer $m\in\Z_{>0}$ and is defined in \S\ref{special}. We denote the Kuga--Satake abelian scheme over $\cO_K$ associated to $\cY$ by $\cA$, and let $A$ denote $\cA_K$. The assumption on $\cY_K$ being Hodge-generic implies that the lattice of special endomorphisms $V(A_{\overline{K}})$ (see \S\ref{special} for the definition) is just $\{0\}$. Fix a prime $\fP$ of $\cO_K$ and let $p$ denote the characteristic of the residue field $\bF_{\fP}$, and let $e$ denote the ramification index of $\fP$ in $K$. We use $\cY_{\overline{\fP}}, \cA_{\overline{\fP}}$ denote the geometric special fibers of $\cY,\cA$ at $\fP$.

Recall the intersection multiplicity $(\cY.\cZ(m))_{\fP}$ from \eqref{int_formula_finite1}: let $v \in V(\cA_{\overline{\fP}})$ be a special endomorphism of $\cA_{\overline{\fP}}$ satisfying $v\circ v = [m]$. We denote by $\cO_{\cY\times_{\cM}\cZ(m),v}$ the \'etale local ring of $\cY\times_{\cM}\cZ(m)$ at $v$. Then we have
\begin{align}\label{int_formula_finite}
(\cY.\cZ(m))_{\fP}=\sum_{\substack{v\in V(\cA_{\overline{\fP}})\\ v\circ v =[ m]}}\length(\cO_{\cY\times_{\cM}\cZ(m),v}) .
\end{align}
In this section, we prove the following result, which controls the above local intersection number on average over $m$.



\begin{theorem}\label{finite-square}
Let $D\in \Z_{\geq 1}$. For $X\in \Z_{>0}$, let $S_{D,X}$ denote the set \[\{m\in \Z_{>0}\mid X \leq m<2X,\, \frac{m}{D}\in \Z \cap (\Q^\times)^2\}.\] Then we have
\[\sum_{m\in S_{D,X}}(\cY . \cZ(m))_\fP=o(X^{\frac{b+1}{2}}\log X).\]
\end{theorem}
This section is organized as follows. First in \S7.1, we express $ (\cY.\cZ(m))_{\fP}$ as a sum of lattice point counts over a family of $p$-adically shrinking lattices. We then prove some preliminary properties of these lattices. Finally in \S7.2, we evaluate these lattice counts to prove Theorem \ref{finite-square}. Needless to say, crucial to the proof of Theorem \ref{finite-square} is the fact that we have the global height bound (\Cref{summary}) for every individual $m$. In \S \ref{sec_trans_ex}, we will illustrate this with an example of a $W(\bF_q)$-valued point $\cY' \in \cM$ (which is \emph{a posteori} not defined over a number field) with the property that $(\cY.\cZ(m))_p$ is exponential in $m$ for an infinite sequence of positive integers $m \in \{n\in \Z_{>0}: \frac{n}{D} \in \Z\cap (\Q^\times)^2\}$.



\subsection{The lattices of special endomorphisms}
Let $K_{\fP}$ denote the completion of $K$ at $\fP$; let $K_{\fP}^{\nr}$ denote a maximal unramified extension of $K_{\fP}$ and let $\cO^{\nr}_{\fP}$ denote its ring of integers.
For every $n \in \Z_{\geq 1}$, let $L_n$ denote the lattice of special endomorphisms $V(\cA_{\cO^\nr_\fP/\fP^n})$. By definition, $L_{n+1}\subset L_n$ for all $n$. Since $\cY_K$ is Hodge-generic, then $\displaystyle \cap_{n=1}^\infty L_n=\{0\}$.
Recall from \S\ref{special} that all $L_n$'s are  equipped with compatible positive definite quadratic forms $Q$ given by $$v\circ v =Q(v)\cdot \Id_{\cA \bmod \fP^n}$$ for every $v\in L_n$.


The next lemma is a direct consequence of the moduli interpretation of $\cZ(m)$ in \S\ref{special}.\footnote{See for instance \cite[Theorems 4.1, 5.1]{conrad} for similar formulas.}

\begin{lemma}\label{loc_int_nb}
The local intersection number is given by
\[(\cY. \cZ(m))_\fP=\sum_{n=1}^\infty |\{v\in L_n \mid Q(v)=m\}|.\]
\end{lemma}
\noindent Note that the right hand side above is indeed a finite sum since there are only finitely many vectors $v$ in $L_1$ with $Q(v)=m$ and for each vector $v$, there exists $n_v\in \Z_{>0}$ such that $v\notin L_{n_v}$.

The following proof uses \eqref{int_formula_finite} and gives a direct description of the length of the \'etale local rings. Alternatively, one may pick a finite covering of $\cM$ over $\Z_p$ and pick a section of $\cY$ to the covering space and then deduce the intersection number via the projection formula. 
\begin{proof}
For the proof of this lemma, we may assume (without loss of generality) that $\cY=\spec(\cO_\fP^\nr)$. Consider $v\in V(\cA_{\overline{\fP}})=L_1$ satisfying $v\circ v=[m]$.
As $\cY$ is a scheme, we use the moduli interpretation of $\cZ(m)$ to see that the \'etale local ring $\cO_{\cY\times_{\cM}\cZ(m),v}$ represents following deformation problem.

Consider any local Artin $\cO^\nr_\fP$-algebra $S$ with residue field $\overline{\bF}_\fP$. We may view $\spec S$ as a scheme over $\cM$ factoring through $\cY$, and we refer to this scheme as $\cY_S$. The deformation problem represented by $\cO_{\cY\times_{\cM}\cZ(m),v}$ associates to the $\cO^\nr_\fP$-algebra $S$ the data $(\cY_S, w)$, where $w$ is a special endomorphism of $\cA_S$ such that the reduction of $w$ (via $ S\rightarrow \overline{\bF}_\fP$) in $V(\cA_{\overline{\fP}})$ equals to $v$. Thus the length of  $\cO_{\cY\times_{\cM}\cZ(m),v}$ equals to the largest $n$ such that $v$ lifts to a special endomorphism of $\cY_{\cO^\nr_\fP/\fP^n}$. 

Let $\bfone_{L_n}:L_1\rightarrow\{0,1\}$ denote the characteristic function of $L_n$. 
Then by \eqref{int_formula_finite}, we have
\begin{align*}
    (\cY.\cZ(m))_{\fP} 
    &=\sum_{\substack{v\in V(\cA_{\overline{\fP}})\\ v\circ v = [m]}}\sum_{n=1}^\infty \bfone_{L_n}(v)=\sum_{n=1}^\infty \sum_{\substack{v\in L_n\\ v\circ v = [m]}}1.\qedhere
\end{align*}
\end{proof}

The following proposition generalizes \cite[Thm.~4.1.1, Lem.~4.1.3, Lem.~4.3.2]{st}.


\begin{proposition}\label{GMdecay}
Let $\Lambda$ denote the $\Z_p$-lattice of special endomorphisms of the $p$-divisible group $\cA[p^\infty]$ over $\cO^\nr_{\fP}$ $($see \Cref{{def_sp_pdiv}}$)$.
Then the $\Z$-rank of $L_n$ is at most $b+2$ and the $\Z_p$-rank of $\Lambda$ is at most $b$. Moreover, there exists a constant $n_0$ such that for $n_0'\geq n_0$ \[L_{n_0'+ke}=(\Lambda + p^k L_{n_0'}\otimes \Z_p)\cap L_{n_0},\]
for $k\geq 1$.
In particular, the rank of $L_n$ is independent of $n$ and we denote it by $r$.
\end{proposition}
\begin{proof}
For the claim on ranks, by \cite[Lemma 4.5.2]{agmp2}, we reduce to the case when $L$ is self-dual at $p$. In this case, by the Dieudonn\'e theory, $L_n\otimes_{\Z}\Z_p\subseteq L_1\otimes_{\Z}\Z_p\subseteq \bfV_{cris,\cY_{\overline{\fP}}}^{\varphi=1}$, which is a $\Z_p$-lattice of rank at most $b+2$. Hence $\rank_{\Z}L_n\leq b+2$. For $\rank_{\Z_p}\Lambda$, as in \cite[Lemma 4.3.2]{st}, we make use of the filtration on $\bfV_{dR,\cY}$. By  Grothendieck--Messing theory, $\Lambda\subseteq \cF^0\bfV_{dR,\cY_{K^\nr}}\cap \bfV_{cris,\cY_{\overline{\fP}}}^{\varphi=1}$ so $\rank_{\Z_p}\Lambda\leq b+1$ and the equality holds if and only if $\cF^0\bfV_{dR, \cY_{K^\nr}}=\spn_{K^\nr}\Lambda$. If so, since $\Lambda\subseteq  \bfV_{cris,\cY_{\overline{\fP}}}^{\varphi=1}$, then $\spn_{K^\nr}\Lambda$ admits trivial filtration by Mazur's weak admissibility theorem. This contradicts $\cF^{1}\bfV_{dR, \cY_{K^\nr}}\neq 0$. We conclude that $\rank_{\Z_p}\Lambda\leq b$.

As in \cite[Lemma 4.1.3]{st}, by Serre--Tate theory, \[\cap_{n=1}^\infty(L_n\otimes \Z_p)=\End_{C(L)}(\cA[p^\infty]_{\cO^\nr_\fP})\cap (L_1 \otimes \Z_p)=\Lambda.\]
To prove the last equality above, by \cite[Lemma 4.5.2]{agmp2}, we reduce to the self-dual case. Then by \Cref{def_sp_pdiv}, a $C(L)$-endomorphism of a $p$-divisible group is special if its crystalline realization lies in $\bfV_{cris,\cY_{\overline{\fP}}}$, which by Dieudonn\'e theory, is equivalent to that it lies in $L_1 \otimes \Z_p$. 

By \cite[Lemma 5.9]{madapusiintegral}, \cite[Lemma 4.5.2]{agmp2} and the N\'eron mapping property, an endomorphism of $\cA_{\cO^\nr_\fP/\fP^n}$ is special if and only if its induced endomorphism in $\End(\cA_{\overline{\bF}_\fP})$ is special. Therefore, a vector $v\in L_1$ lies in $L_n$ if and only if $v$ deforms to an endomorphism of $\cA_{\cO^\nr_\fP/\fP^n}$. Then the rest of the argument is the same as in the proof of \cite[Theorem 4.1.1]{st}.
\end{proof}

We now define the successive minima of a lattice following \cite{EK}, and discuss the asymptotics of the successive minima of $L_n$.

\begin{definition}
\begin{enumerate}
    \item  For $1\leq i\leq r$, the successive minimum $\mu_i(n)$ of $L_n$ is defined as: \[\inf\{y\in \R_{>0}:\exists v_1,\cdots, v_i\in L_n \text{ linearly independent, and } Q(v_j)\leq y^2, 1\leq j\leq i\}.\] 
    
    \item For $n\in \Z_{\geq 1}, 1\leq i \leq r$, define $a_i(n)=\prod_{j=1}^i \mu_j(n)$; define $a_0(n)=1$.
\end{enumerate}
\end{definition}

We have the following consequence of Proposition \ref{GMdecay}.

\begin{corollaire}\label{GMdecay-cor}
Every successive minimum $\mu_j(n)$ satisfies $\mu_j(n)\ll p^{n/e}$. If we further assume that $r = b+2$, then $a_{b+1}(n) \gg p^{n/e}$ and $a_{b+2}(n) \gg p^{2n/e}$.
\end{corollaire}

\begin{proof}
By \Cref{GMdecay}, there exists an absolute bounded $n_0\equiv n\pmod{e}$ such that
\begin{equation*}
    L_n=\bigl(\Lambda+p^{(n-n_0)/e}L_{n_0}\otimes \Z_p\bigr)\cap L_{n_0}.
\end{equation*}
Denote $(n-n_0)/e$ by $k$, and note that $k\gg n/e$.
Since $p^kL_{n_0}\subset L_n$, it follows that $\mu_j(n)\ll p^k$ for every $j$, proving the first claim.

Now if the rank $r$ of $L_n$ is $b+2$, then (since the rank of $\Lambda$ is at most $b$) we clearly have
$[L_1:L_n] \gg p^{2k}$. Thus $\Disc(L_n)^{\frac{1}{2}} \gg p^{2k}$. By \cite[Equations (5),(6)]{EK} this implies that $a_{b+2}(n) \gg p^{2k}$ as required. 
In conjunction with $\mu_{b+2}(n)\ll p^k$, we also immediately obtain that $a_{b+1}(n)\gg p^k$.
\end{proof}

\begin{lemma}\label{first-min}
For every $\epsilon >0$, we have $a_1(n)\gg_\epsilon n^{\frac{1}{b + \epsilon}}$. Moreover, $a_i(n)\gg_\epsilon n^{\frac{i}{b+\epsilon}}$.
\end{lemma}
\begin{proof}
Let $\epsilon >0$. By \Cref{summary}(ii), we have 
\begin{align}\label{eq_ht_bd}
    (\cY . \cZ(m))_{\fP} \ll m^{\frac{b}{2}}\log m \ll_\epsilon m^{\frac{b+\epsilon}{2}}.
\end{align}

Let $w_0\in L_n$ denote a vector such that $Q(w_0)=a_1(n)^2$. By taking $m=a_1(n)^2$ in \eqref{eq_ht_bd}, we get
\[n \ll (\cY. \cZ(a_1(n)^2))_\fP \ll_\epsilon a_1(n)^{b+\epsilon}\]
where the first bound follows from \Cref{loc_int_nb} and the observation that $w_0\in L_k$ for all $k\leq n$. The second assertion follows directly from the first, in conjunction with the bound $a_i(n)\geq a_1(n)^i$.
\end{proof}

\subsection{Proof of \Cref{finite-square}}
We first introduce some notations. For any positive integers $a<b$ and  $D,X$ as in \Cref{finite-square},  define 
\begin{equation*}
   \bfN_{D}(a,b;X) =  \sum_{n = a}^{b}|\{v \in L_n: Q(v) \in S_{D,X} \}|.
\end{equation*}

It is known that $\rank L_n=b+2$ if and only if $\cY$ has supersingular reduction at $\fP$.\footnote{Indeed, $\rank L_n=b+2$ if and only if the Frobenius $\varphi$ is isoclinic on $\bfV_{cris, \cY_{\overline{\fP}}}$. The later claim is equivalent to that $\cY_{\fP}$ lies in the basic (i.e., supersingular) locus in $\cM_{\overline{\fP}}$. Note that we do not need this fact in the proof of \Cref{finite-square}.} 
When $\fP$ is a prime of supersingular reduction, we write $\bfN_D(1,\infty;X)$ as a sum $\bfN_D(1,\lfloor \frac{e}{4} \log_p X \rfloor;X)+\bfN_D(\lceil \frac{e}{4} \log_p X \rceil,\infty;X)$.
In the following proposition, we follow \cite[\S 4.2, \S 4.3]{st} to bound the finite contributions
for primes $\fP$ modulo which $\cY$ does not have supersingular reduction, and also bound $\bfN_D(\lceil \frac{e}{4} \log_p X \rceil,\infty;X)$ for primes $\fP$ modulo which $\cY$ does have supersingular reduction.


\begin{proposition}\label{nonsscontribution}
Let notation be as above. Then we have:
\begin{enumerate}
    \item If $r=\rank L_n \leq b+1$, then \[\sum_{n=1}^{\infty}|\{v\in L_n\backslash\{0\}: Q(v)<X\}| =O(X^{\frac{b+1}{2}}).\]
    \item If $r=b+2$, then \[\sum_{n=\lceil\frac{e}{4}\log_p X\rceil}^{\infty}|\{ v\in L_n\backslash\{0\}: Q(v)<X\}=O(X^{\frac{b+1}{2}}) .\]
\end{enumerate}
\end{proposition}

\begin{proof}
Let $\epsilon \in ]0,1[$. By \Cref{first-min}, there exists a constant $C_{0, \epsilon}$ such that $a_1(n)\geq C_{0,\epsilon} n^{1/(b+\epsilon)}$. Let $C_{1,\epsilon}=C_{0,\epsilon}^{-(b+\epsilon)}$. If $n> (X^{1/2}C_{0,\epsilon}^{-1})^{b+\epsilon}=C_{1,\epsilon} X^{\frac{b+\epsilon}{2}}$, then $a_1(n)>X^{1/2}$ and hence \[\{v \in L_n\backslash\{0\}: Q(v)<X\}=\emptyset.\]

Therefore, for (1), we have 
\begin{align*}
\sum_{n=1}^{\infty}|\{v\in L_n\backslash\{0\}: Q(v)<X\}|&=\sum_{n=1}^{C_{1,\epsilon}X^{\frac{b+\epsilon}{2}}}|\{v\in L_n\backslash\{0\}: Q(v)<X\}|\\
 &\overset{(i)}{\ll} \sum_{n=1}^{C_{1,\epsilon}X^{\frac{b+\epsilon}{2}}} \sum_{i=0}^r \frac{X^{\frac{i}{2}}}{a_i(n)}\\ &\overset{(ii)}{\ll_\epsilon} \sum_{n=1}^{C_{1,\epsilon}X^{\frac{b+\epsilon}{2}}} \sum_{i=0}^r \frac{X^{\frac{i}{2}}}{n^{i/(b+\epsilon)}}\\ 
 &=\sum_{i=0}^r\sum_{n=1}^{C_{1,\epsilon}X^{\frac{b+\epsilon}{2}}}  \frac{X^{i/2}}{n^{i/(b+\epsilon)}},
\end{align*}
where (i) follows from \cite[Lemma 2.4, Equations (5),(6)]{EK},\footnote{The authors refer to \cite{schmidt} for a proof of their Lemma 2.4} and (ii) follows from \Cref{first-min}. For $0\leq i\leq b$, note that \[\sum_{n=1}^{C_{1,\epsilon}X^{\frac{b+\epsilon}{2}}}  \frac{X^{\frac{i}{2}}}{n^{i/(b+\epsilon)}} \ll_\epsilon X^{\frac{i}{2}}\cdot (X^{\frac{b+\epsilon}{2}})^{1-\frac{i}{b+\epsilon}}=O(X^{\frac{b+\epsilon}{2}}).\]
For $i=b+1$, since $\sum_{n=1}^\infty n^{-(b+1)/(b+\epsilon)}$ converges, we have \[\sum_{n=1}^{C_{1,\epsilon}X^{\frac{b+\epsilon}{2}}}  \frac{X^{\frac{i}{2}}}{n^{i/(b+\epsilon)}}=O_\epsilon(X^{\frac{b+1}{2}}).\]

For (2), similarly, we have that the left hand side is bounded by
\[\sum_{n=\lceil\frac{e}{4}\log_p X\rceil}^{C_{1,\epsilon}X^{\frac{b+\epsilon}{2}}}  \frac{X^{(b+2)/2}}{a_{b+2}(n)}+\sum_{i=0}^{b+1}\sum_{n=\lceil\frac{e}{4}\log_p X\rceil}^{C_{1,\epsilon}X^{\frac{b+\epsilon}{2}}}  \frac{X^{\frac{i}{2}}}{n^{i/(b+\epsilon)}}.\]
As in (1), the second term is $O_\epsilon(X^{\frac{b+1}{2}})$. For the first term, by \Cref{GMdecay-cor}, we have $a_{b+2}(n)\gg p^{\frac{2n}{e}}$. Since the series $\sum_{n=1}^\infty \frac{1}{p^{2n/e}}$ converges and $p^{n/e}\geq X^{1/4}$ when $n\geq \lceil\frac{e}{4}\log_p X\rceil$, then the first term is bounded by $O(\frac{X^{\frac{b+2}{2}}}{\sqrt{X}})=O(X^{\frac{b+1}{2}})$, hence the result. 
\end{proof}

\begin{remarque}
In order to prove \Cref{main}, we do not have to restrict ourselves to sets like $S_{D,X}$ and can sum over all $m$. The bounds that the proof of  \Cref{nonsscontribution} yields are therefore sufficient, even in the case when $\fP$ is a prime of supersingular reduction. 
\end{remarque}

We are now ready to finish the proof of \Cref{finite-square}.
\begin{proof}[Proof of \Cref{finite-square}]
By Proposition \ref{nonsscontribution}(1), we may restrict ourselves to the case where the rank of $L_n$ is equal to $b+2$. By Proposition \ref{nonsscontribution}(2), it suffices to prove that $\bfN_D(1,\lfloor \frac{e}{4} \log_p X \rfloor,X) = o(X^{\frac{b+1}{2}}\log X)$. 

To that end, let $1\leq T\leq \lfloor \frac{e}{4} \log_p X \rfloor/e$ be an integer and let $m$ be an integer satisfying $X \leq m < 2X$. For brevity, let \[\bfN_1(m) =|\{v \in  L_1: Q(v) = m\}|, \quad \bfN_T(m) = |\{v \in  L_{eT}: Q(v) = m\}|.\] 
Then we have the trivial bound 
\begin{align}\label{eq1_pf_fin}
    \sum_{n =1 }^{\lfloor \frac{e}{4}\log_p X \rfloor}|\{v \in L_n: Q(v) = m\}|&= \sum_{n=1}^{eT}|\{v \in L_n: Q(v) = m\}|+\sum_{n=eT+1}^{\lfloor \frac{e}{4}\log_p X \rfloor}|\{v \in L_n: Q(v) = m\}|\\
    &\leq  eT\bfN_1(m) + \frac{e\log_p X}{4}\bfN_T(m).
\end{align}

By \Cref{HB_cor}, $eT \bfN_1(m) \ll eTm^{\frac{b}{2}} \ll eTX^{\frac{b}{2}}$ 
and \[\bfN_T(m) = \mu_{\infty}(Q_T,1)\mu(Q_T,m)m^{\frac{b}{2}} + O_{T,\epsilon}(m^{(b+1)/4 + \epsilon}),\]
where $Q_T$ is the positive definite quadratic form on $L_{eT}$.
By Lemma \ref{lemdencomp1} below, 

$$\mu_{\infty}(Q_T,1)\mu(Q_T,m) \ll p^{-3/5T},$$ and so we obtain $$\bfN_T(m) \ll p^{-3/5T}m^{\frac{b}{2}} + O_{T,\epsilon}(m^{(b+1)/4 + \epsilon}).$$ 

Therefore, by summing \eqref{eq1_pf_fin} over $m \in S_{D,X}$ and by the above bounds on $\bfN_1(m), \bfN_T(m)$, we have 
\begin{equation*}
\frac{\bfN_D(1,\lfloor e\log_pX \rfloor,X)}{X^{\frac{b+1}{2}}\log X} \ll \frac{eT}{\log X}+ p^{-3/5T} + \frac{O_{T,\epsilon}(X^{(b+3)/4 + \epsilon} )}{X^{\frac{b+1}{2}}\log X}.
\end{equation*}

Therefore, 
\begin{equation*}
\limsup_{X\rightarrow\infty} \frac{\bfN_D(1,\lfloor e\log_pX \rfloor,X)}{X^{\frac{b+1}{2}}\log X} \ll p^{-3/5T}.
\end{equation*}
As the above inequality is true for every value of $T$, we have 
\begin{equation*}
\bfN_D(1,\lfloor e\log_pX \rfloor,X)= o(X^{\frac{b+1}{2}}\log X),
\end{equation*}
whence the theorem follows.
\end{proof}

\begin{lemma}\label{lemdencomp1}
Let $Q$ denote an integral positive definite quadratic form of rank $r\geq 5$, let $m\geq 1$ be any integer and let $p$ denote a prime. Then, we have 
$$ \mu_{\infty}(Q,1) \mu_p(Q,m) \ll \frac{p^r}{|\Disc(Q)|^{3/20}},$$
where the implicit constant above only depends on $r$.
In particular, for $T$ and $p$ as above, we have $$\mu_{\infty}(Q_T,1)\mu(Q_T,m) \ll p^{-3/5T}.$$
\end{lemma}
\begin{proof}
The definition of $\mu_{\infty}(Q,1)$ in \S\ref{circle_method} yields that $\mu_{\infty}(Q,1) \asymp \textrm{Disc}(Q)^{-1/2}$, where the implicit constant only depends on the rank $r$. The assertion about $Q_T$ follows from the first assertion, by the fact that $\textrm{Disc}(Q_T)^{-1/2}\ll p^{-2T}$ (from \Cref{GMdecay-cor}), and the fact that for a prime $\ell \neq p$, $\mu_{\ell}(Q_T,m)$ is independent of $T$ for $T\gg 1$ (from \Cref{GMdecay}). Therefore, it suffices to prove the first assertion. 

 Recall from \S\S\ref{sec_density}-\ref{circle_method} that we have $\mu_p(Q,m)=\mu_p(m,n)$ for some sufficiently large integer $n$. Moreover, following \cite[\S 3, pp.~359-360]{hanke}, we have
\begin{equation*}
    \mu_p(m,n)=\mu_p^{{\rm good}}(m,n)+\mu_p^{{\rm bad1}}(m,n)+\mu_p^{{\rm bad2}}(m,n)+\mu_p^{{\rm zero}}(m,n),
\end{equation*}
where the summands come from elements of reduction type good, bad1, bad2, and zero, respectively.
To prove the first estimate of the lemma, we once again (as in \S\ref{sec_density}) use the reduction maps from \cite[\S3]{hanke}. These immediately yield the following inequalities: 
\begin{equation*}
\begin{array}{rcl}
\mu_p^{{\rm good}}(m,n)&\leq& p^3,\\[.1in]
\mu_p^{{\rm bad1}}(m,n)&\leq& p^4,\\[.15in]
\displaystyle\frac{\mu_p^{{\rm bad2}}(m,n)}{\Disc(Q)^{1/2}}&\leq& \displaystyle p^{2-r}\frac{\mu''_p(m/p^2,n-2)}{\Disc(Q'')^{1/2}},\\[.2in]
\mu_p^{{\rm zero}}(m,n)&\leq& p^{2-r}\mu_p(m/p^2,n-2),
\end{array}
\end{equation*}
where in the third line, $Q''$ is a quadratic form of rank $r$ constructed from $Q$ (see \cite[p.360]{hanke} for the definition of $Q''$) and $\mu_p''$ is the density corresponding to $Q''$. Furthermore, it is easy to check that we have $\Disc(Q'')\geq \Disc(Q)/p^{2r}$. Then we obtain the bound $\displaystyle\frac{\mu_p(m,n)}{\Disc(Q)^{1/2}}\leq \frac{p^r}{\Disc(Q)^{3/20}}$ by induction on $n$, along with the observation that in each step of the induction, $\Disc(Q'')$ decreases by at most $p^{2r}$, and that $r-2\geq 3$. 
\end{proof}

\subsection{A transcendental example.}\label{sec_trans_ex}
We now demonstrate an example of a point  $\cY'\in \cM(W(\bF_q))$ with the property that $(\cY'.\cZ(m))_p$ is exponential in $m$ for an infinite sequence of $m \in S$ where $S = \{\Z_{>0} \cap (\Q^\times)^2\}$ and ${\rm char}\, \bF_q =p$. Similar examples can be analogously constructed with $m\in S_D= \{\Z_{>0} \cap D\cdot  (\Q^\times)^2$. As this example isn't necessary for any of the main results, and only serves to highlight the importance of the global height bound, we will be brief in our exposition. 

For simplicity, assume that the rank of the lattice defining the Shimura variety is even and is at least $12$. Let $b = 2c$. We may choose our Shimura variety $\cM$ such that it admits a map from the modular curve in characteristic 0, and hence contains a CM point $\cY$ whose lattice of special endomorphisms $L$ also has rank $b = 2c$. We may also assume that $L$ represents every large enough positive integer. 

Let $p$ be a large enough prime of ordinary reduction for this point. We also assume that $p$ doesn't divide the discriminant of the lattice of special endomorphisms of $\cY$. Let $y$ denote the mod $p$ reduction of $\cY$ -- we may assume (by increasing $p$ if necessary) that $\cY$ is the canonical lift of $y$, and therefore $L$ is also the lattice of special endomophrisms of $y$. Let $q$ be such that $y \in \cM(\bF_q)$.  

The lemma below follows from Serre--Tate theory and \cite[Corollaries 5.17,5.19]{madapusiintegral}:

\begin{lemma}\label{padiclatticelift}
Let $\Lambda\subset L\otimes \Z_p$ denote a self-dual\footnote{We mean self dual with respect to the quadratic form on $L$.} $\Z_p$-submodule. Then there exists a point $\cY_{\Lambda} \in \cM(W(\bF_q))$ lifting $y$ whose $\Z_p$-lattice of special endomorphisms is $\Lambda$. Further, we have that the $\Z$-lattice of special endomorphisms of $\cY_{\Lambda} \bmod p^{n+1}$ is $(\Lambda + p^nL\otimes \Z_p) \cap L$.
\end{lemma}

Our example $\cY'$ will equal $\cY_{\Lambda}$ for an appropriate choice of $\Lambda$. A choice of $\Lambda$ which is ``very well approximated'' by $\Z$-sublattices of $L$ will have the property that $(\cY_{\Lambda}.\cZ(m))$ is exponential in terms of $m$ for infinitely many $m$. Indeed, the following lemma makes this precise. 
\begin{lemma}\label{wellapproximatedlattice}
Suppose that a rank $r\geq 5$ self-dual sublattice $\Lambda \subset L\otimes \Z_p$ has the following properties: 
\begin{enumerate}
    \item $\Lambda$ contains no integeral elements, i.e. $\Lambda \cap L = \{0\}$. \label{Property 1}
    \item There exists an increasing sequence of rank $r$ sublattices $L_i\subset L$ with discriminants $D_i$ such that $\Lambda \equiv L_{n_i} \mod p^{N_i}$, where $N_i \geq e^{e^{D_i}}$. \label{Property 2}
\end{enumerate}
Then, there exists an increasing sequence of perfect squares $m_i$ such that $(\cY_{\Lambda},\cZ(m_i))_p$ is exponential in $m_i$.

\end{lemma}
\begin{proof}

We first note that property \ref{Property 1} implies that $\cY_{\Lambda}$ isn't contained in any special divisor $\cZ(m)$, and so the intersection $(\cY_{\Lambda}.\cZ(m))_p$ is well defined for every positive integer $m$. Property \ref{Property 2} implies that the lattice of special endomorphisms of $\cY_{\Lambda} \bmod p^{N_i+1}$ is $L_i + p^{N_i}L$. The lattice $L_i$, having rank at least 5, must locally represent perfect squares satisfying congruence conditions modulo a power of $D_i$, where the power depends only on the rank $r$ and nothing else. Therefore, $L_i$ must represent perfect squares (again satisfying these congruence conditions) having size bounded by $e^{D_i}$, as if a quadratic form having rank at least 4 locally represents integers which are large relative to the discriminant, it must actually represent such integers too. 

Pick some perfect square $m_i < e^{D_i}$ which $L_i$ represents. Then, we have (by Lemma \ref{loc_int_nb}) $(\cY_{\Lambda}.\cZ(m_i))_p > N_i$. The lemma follows. 
\end{proof}

We now only need to find a $\Z_p$-sublattice of $L$ that satisfies the properties of Lemma \ref{wellapproximatedlattice}. To that end, we may assume that $L$ has an integral orthogonal basis  $e_1,f_1,\hdots e_c,f_c$. For elements $\mu_1\hdots \mu_c \in \Z_p$, we consider the rank $c$ $\Z_p$-sublattice defined by $\Lambda = \textrm{Span}_{\Z_p}\{e_i + \mu f_i \}$. The fact that $b\geq 10$ implies that $\Lambda$ has rank $\geq 5$. Insisting that the $\mu_i$ are irrational implies that $\Lambda$ has property \ref{Property 1}. 
We now pick the $\mu_i$ so that $\Lambda$ satisfies property \ref{Property 2}. We will write $\mu = \sum_{i\geq 0} a_i p^i$, with $0\leq a_i \leq p-1$ and $a_0 = 1$. We define the $a_i$ as follows: suppose that $a_n \neq 0$, we then define $a_i = 0$ for $n+1 \leq i \leq e^{e^{p^n}}$, and then define $a_{i} = 1$ for $i = \lceil e^{e^{p^n}}\rceil$. This implies that $\sum_{i\le n} a_i p^i \equiv \mu \bmod p^N$ with $N\geq e^{e^{p^n}}$, i.e. there is an integer $x_n$ of size $O(p^n)$ which approximates $\mu$ $\bmod p^N$, with $N$ double-exponential in the size of $x_n$. Having defined $\mu$ in this way, we pick $\mu_i = \mu$ for all $i$, and the lattice $\Lambda$ created in this way must also have property \ref{Property 2}. 

Choosing $\cY' = \cY_{\Lambda}$, we obtain our example by Lemma \ref{wellapproximatedlattice}. 

\section{Proof of the main theorem}\label{finalmvt}
Let $(L,Q)$ be a maximal integral quadratic even lattice of signature $(b,2)$ with $b\geq 3$, and let $\mathcal{M}$ denote the integral model of the Shimura variety  associated to $(L,Q)$ defined in \S 2.
We recall the statement of the main theorem:
\begin{thm}[Theorem \ref{main_sp_end}]
Let $K$ be a number field and let $D\in \Z_{>0}$ be a fixed integer represented by $(L,Q)$. Let $\cY\in \mathcal{M}(\mathcal{O}_K)$ and assume that $\cY_K\in M(K)$ is Hodge-generic. Then there are infinitely many places $\fP$ of $K$ such that $\cY_{\overline{\fP}}$ lies in the image of $\cZ(Dm^2)\rightarrow \cM$ for some $m\in \Z_{>0}$. Equivalently, for a Kuga--Satake abelian variety $\cA$ over $\cO_K$ parameterized by $\cM$ such that $\cA_{\overline{K}}$ does not have any special endomorphisms, there are infinitely many $\fP$ such that $\cA_{\overline{\fP}}$ admits a special endomorphism $v$ such that $v\circ v=[Dm^2]$ for some $m\in \Z_{>0}$.
\end{thm}

In this section, we will prove Theorem \ref{main_sp_end} using results proved in the previous sections. First, we recall results and definitions that we will need to prove the main theorem.

We have the expression
\begin{equation}\label{localglobal}
h_{\widehat{\cZ}(m)}(\cY)=\sum_{\sigma:K\hookrightarrow \C}\Phi_m(\cY^\sigma)+\sum_{\fP} (\cY. \cZ(m))_\fP \log |\cO_K/\fP|.
\end{equation}
In \S\S\ref{mainproof}, \ref{sec_arch} and \ref{sec_finite}, we proved results bounding the terms in \eqref{localglobal} which we restate below for the convenience of the reader.

\begin{description}
    \item[Theorem \ref{arch1}]
For every $m$ representable by $(L,Q)$, we have
 \[\Phi_m(\cY^\sigma)=c(m)\log m + A(m, \cY^\sigma)+o(|c(m)|\log m).\]
 \item [Theorem \ref{thAbound}]
There exists a subset $S_{\bad} \subset \Z_{>0}$ of logarithmic asymptotic density zero such that for every $m\notin S_{\bad}$, we have
\begin{equation*}
    A(m,\cY^{\sigma})=o(m^{\frac{b}{2}}\log(m)).
    \end{equation*}
 \item[Theorem \ref{finite-square}]
Given $D, X\in \Z_{>0}$, let $S_{D,X}$ denote the set \[\{m\in \Z_{>0}\mid X \leq m<2X, \sqrt{m/D}\in \Z \}.\]
For a fixed prime $\fP$ of $K$ and a fixed $D$, we have
\[\sum_{m\in S_{D,X}}(\cY . \cZ(m))_\fP=o(X^{\frac{b+1}{2}}\log X).\] 
\end{description}

\begin{proof}[Proof of \Cref{main_sp_end}]
Assume for contradiction that there exists $D\in\Z_{>0}$ represented by $L$ such that there are only finitely many $\fP$ for which $\cY_{\overline{\fP}}$ lies in the image of $\cZ(m)$ where $m/D$ is a perfect square. Therefore, for such $m$, $(\cY,\cZ(m))_{\fP}=0$ for all but finitely many $\fP$.

For $X\in \Z_{>0}$, let $S^\good_{D,X}$ denote the set $\{m\in S_{D,X}\mid m\notin S_\bad\}$, where $S_{D,X}$ is defined in \Cref{finite-square} and $S_\bad$ is union of the sets of log asymptotic density $0$ in \Cref{thAbound} by taking $x=\cY^\sigma$ for all $\sigma:K\hookrightarrow \C$. Then $S_\bad$ is also of log asymptotic density $0$ and $|S^\good_{D,X}|\asymp X^{1/2}$ as $X\rightarrow \infty$. On the other hand, by assumption, $D$ is representable by $(L,Q)$, then each $m\in S_{D,X}$ is representable by $(L,Q)$ and hence $\cZ(m)\neq \emptyset$.

We sum (\ref{localglobal}) over $m\in S^\good_{D,X}$ and note that for each $m\in S^\good_{D,X}$, $m\asymp X$. For the archimedean term, by \Cref{arch1} we have
\begin{align}\label{eqn_arch}
    \sum_{m\in S^\good_{D,X}}\sum_{\sigma:K\hookrightarrow\C}\frac{\Phi_{m}(\cY^\sigma)}{|\mathrm{Aut}(\cY^\sigma)|}\asymp -X^{\frac{b+1}{2}}\log X.
\end{align}

For a fixed $\fP$, since $(\cY.\cZ(m))_{\fP}\geq 0$ for all $m$, then by \Cref{finite-square},
\[\sum_{m\in S^\good_{D,X}}(\cY . \cZ(m))_\fP \log|\cO_K/\fP| \leq\sum_{m\in S_{D,X}}(\cY . \cZ(m))_\fP \log|\cO_K/\fP|=o(X^{\frac{b+1}{2}}\log X).\]
Since $(\cY,\cZ(m))_{\fP}=0$ for all but finitely many $\fP$, we have 
\begin{align}\label{eqn_finite}
\sum_{m\in S^\good_{D,X}}\sum_{\fP}(\cY . \cZ(m))_\fP \log|\cO_K/\fP|=o(X^{\frac{b+1}{2}}\log X).
\end{align}

By \Cref{eq1}, we have $\displaystyle \sum_{m\in S^\good_{D,X}} h_{\widehat{\cZ}(m)}(\cY) = O(X^{\frac{b+1}{2}})$, which contradicts \eqref{localglobal}.
\end{proof}

\section{Applications: Picard rank jumps and exceptional isogenies}\label{applicationsch}
In this section, we will elaborate on a number of applications : K3 surfaces and rational curves on them, exceptional splittings of Kuga-Satake abelian varieties and abelian varieties parametrized by unitary Shimura varieties. Then we set our results in context with past work of Charles \cite{charles1} and Shankar--Tang \cite{st} that deals with Shimura varieties associated to quadratic lattices of signature $(2,2)$.

\subsection{Picard rank jumps in families of K3 surfaces and rational curves}\label{sec_K3}
For background on K3 surfaces, we refer to \cite{huybrechts}. Let $X$ be a K3 surface over a number field $K$. By replacing $K$ with a finite extension if necessary, we may assume that $\mathrm{Pic}(X_{\overline{K}})=\mathrm{Pic}(X)$. For any embedding $\sigma: K \hookrightarrow \C$, the $\Z$-module $H^2(X_{\sigma}^{an},\Z)$ endowed with the intersection form $Q$ given by Poincaré duality is an unimodular even lattice of signature $(3,19)$. The first Chern class map $$c_1:\mathrm{Pic}(X)\rightarrow H^2(X_{\sigma}^{an},\Z)$$ is a primitive embedding. By the Hodge index theorem, $\mathrm{Pic}(X)$ has signature $(1,\rho(X)-1)$, where $\rho(X)$ is the Picard rank of $X$. Let $(L,Q)$ be a maximal orthogonal lattice to $\mathrm{Pic}(X)$ in $H^2(X_{\sigma}^{an},\Q)$. Then $(L,-Q)$ is an even lattice, whose genus is independent of the choice of $\sigma$ and $L$, and has signature $(b,2)$ where $b=20-\rho(X)$. Let $\mathcal{M}$ be the GSpin Shimura variety associated to $(L,-Q)$. By \cite[Main Lemma 1.7.1]{andretate}, up to extending $K$, the Kuga-Satake abelian variety associated to $X$, denoted by $A$, is defined over $K$ and corresponds to a $K$-point $x\in\mathcal{M}(K)$.



\begin{proof}[Proof of \Cref{k3ar}]
Since $X$ has everywhere good reduction, up to extending $K$, by \cite[Lemma~9.3.1]{andretate}, the corresponding Kuga--Satake abelian variety $A$ has also potentially good reduction everywhere. Then by \Cref{arithmeticstack}(4), it gives rise to an $\mathcal{O}_K$-point $\cY$ of $\mathcal{M}$.
It is Hodge-generic by construction of the Shimura variety $\mathcal{M}$ and by the Lefschetz theorem on $(1,1)$ classes. By Theorem \ref{main}, there exist infinitely many places $\mathfrak{P}$ of $K$ and $m>0$, such that the geometric fiber of the reduction $A$, denoted by $A_{\overline{\mathfrak{P}}}$, has a special endomorphism $s$ such that $s\circ s=[m]$. By \cite[Main Lemma 1.7.1 (ii)]{andretate}, a special endomorphism of $A_{\overline{\mathfrak{P}}}$ corresponds to a line bundle $L$ on $X_{\overline{\mathfrak{P}}}$ such that $Q(L)=-m$ and which is orthogonal to the image of $\mathrm{Pic}(X)$ inside $\mathrm{Pic}(X_{\overline{\mathfrak{P}}})$ under the specialization map. This proves \Cref{k3ar}.
\end{proof}


\begin{proof}[Proof of \Cref{rational}]
Let $X$ be a K3 surface defined over a number field which has potentially everywhere good reduction. Then by \Cref{k3ar}, $X$ has infinitely many specializations where the geometric Picard rank jumps. If $X$ has finitely many unirational specializations, then the strategy of \cite{ll} can be applied, more precisely the statements \cite[Proposition 4.2]{ll} are satisfied, and we can thus conclude by the proof of \cite[Theorem 4.3]{ll}. 
\end{proof}

\subsection{Kuga--Satake abelian varieties}\label{sec_KSsplit}

Via the exceptional isomorphism between $\GSp_4$ and $\GSpin(V)$ with $b=3$, as in \cite{KR00}, the moduli space $\cS_2$ of principally polarized abelian surfaces\footnote{They do not need the polarization degree to be one; here we work with the principally polarized case for simplicity. Indeed, to prove \Cref{thm_absurf}, we may enlarge $K$ and work with a principally polarized abelian surface isogenous to the one in question.} is a GSpin Shimura variety. 
In this case, let $B$ be a principally polarized abelian surface; then as in \cite{KR00}, the special endomorphisms are $s\in \End(B)$ such that $s^\dagger =s$ and $\tr s =0$, where $\dagger$ denotes the Rosati involution. Indeed, let $A$ denote the Kuga--Satake abelian variety (of dimension $2^{2+3-1}=16$) at the point $[B]\in \cS_2$ and $A=A^+\times A^-$ given in \S\ref{sec_KS}; Kudla and Rapoport gave a moduli interpretation of special divisors by defining special endomorphisms to be $s\in \End_{C^+(L)}(A^+)$ such that $s^\dagger =s$ and $\tr s =0$ (see \cite[\S 1, Definition 2.1]{KR00}). By \cite[\S 1]{KR00}, we have $C^+(V)\cong M_4(\Q)$ and hence $A^+$ is isogenous to $B^4$; moreover, the special endomorphisms induces $s_B\in \End(B)\otimes \Q$ such that $s_B^\dagger =s_B$ and $\tr s_B =0$.

The Kudla--Rapoport version of special endomorphisms allows us to deduce \Cref{thm_absurf} from \Cref{main_sp_end}. We now work with the general setting as in \Cref{ass_KSsplit} since the argument is the same.
Recall that $b=2n-1$ for $n\in \Z_{>0}$, and we assume that $C^+(V)\cong M_{2^n}(\Q)$, then $A^+$ is isogenous to $B^{2^n}$, where $B$ is an abelian variety with $\dim B= 2^n$. By \cite[\S 5.2]{vanGeemen2}, if $[A]$ is a Hodge generic point in $\cM$, then $\End(A^+_{\bar{K}})\otimes \Q=C^+(V)$ and in particular, $\End(B_{\bar{K}})=\Z$ and $B$ is geometrically simple.

In order to translate a special endomorphism of the Kuga--Satake abelian variety $A$ to a special endomorphism of $A^+$, we choose an element $\delta_0\in Z(C(L))\cap C(L)^-$ such that $\delta_0^*=\delta_0$, where $Z(C(L))$ denote the center of $C(L)$ and $(-)^*$ denote the unique involution on $C(V)$ which acts trivially on $V$ (see for instance \cite[\S 2.1]{agmp1} for a concrete definition). Indeed, let $e_1,\dots, e_{b+2}\in L$ be a basis of $V$ such that $Q(v)=d_1x_1^2+\cdots + d_{b+2}x_{b+2}^2$ for $v=x_1e_1+\cdots + x_{b+2}e_{b+2}$. Since $b\equiv 3 \bmod 4$, we may take $\delta_0=e_1\cdots e_{b+2}$ and note that $\delta_0^2=\prod_{i=1}^{b+2} d_i$. Via the usual $C(L)$-action on $A$, the element $\delta_0\in C(L)^-$ induces an endomorphism $\delta_0:A^-\rightarrow A^+$ and hence for any special endomorphism $v\in \End_{C(L)}(A)$ defined in \S\ref{special}, since $v:A^+\rightarrow A^-$ and $\delta_0\in Z(C(L))$, we have $s:=\delta_0\circ v\in \End_{C^+(L)}(A^+)$. Since $C^+(V)\cong M_{2^n}(\Q)$ and $A^+$ is isogenous to $B^{2^n}$, then we obtain $s_B\in \End(B)\otimes \Q$. Since $s$ is not a scalar multiplication on $A^+$, then $s_B$ is not a scalar multiplication on $B$.

\begin{proof}[Proof of \Cref{thm_absurf} and \Cref{cor_gspin}]
Notation as above, let $D=\prod_{i=1}^{b+2} d_i$. By \Cref{cor_repD}, without loss of generality, we may multiple $D$ by a square number such that $D$ is representable by $(L,Q)$.
For a finite place $\fP$, if $v$ is a special endomorphism of $A_{\overline{\fP}}$ such that $v\circ v=[Dm^2]$, then $\delta_0\circ v$ induces a quasi-endomorphism $s_B$ on $B_{\overline{\fP}}$ such that $s_B\circ s_B=[Q(\delta)Q(v)]=[D^2m^2]$.
Since $s_B$ is not a scalar multiplication, then $\ker(s_B-[Dm])$ is a non-trivial simple factor of $B_{\overline{\fP}}$ and hence $B_{\overline{\fP}}$ is non simple. 
We conclude by \Cref{main_sp_end} that there are infinitely many such $\fP$.
\end{proof}

Via the algorithm in the proof of \cite[Thm. 7.7]{vanGeemen}, here is an example when $C^+(V)=M_{2^n}(\Q)$: assume $b\equiv 3 \bmod 8$, consider $Q(x)=-x_1^2-x_2^2+\sum_{i=3}^{b+1} x_i^2 +d x_{b+2}^2$.


\subsection{Abelian varieties parametrized by unitary Shimura varieties}\label{sec_unitary}
We recall the moduli interpretation of the Shimura varieties attached to $\GU(r,1)$ following \cite[\S 2]{KR14} (see also \cite[\S 2.2]{BHKRY}). Recall that $E$ is an imaginary quadratic field. Consider the moduli problem which associates to a locally noetherian $\cO_E$-scheme $S$ the groupoid of triples $(B,\iota, \lambda)$, where $B$ is an abelian scheme over $S$, $\iota: \cO_E\hookrightarrow \End_S(B)$, and $\lambda:B\rightarrow B^\vee$ is a principal polarization such that 
\begin{enumerate}
    \item $\iota(a)^\dagger=\iota (a^\sigma)$, where $\dagger$ is the Rosati involution and $\sigma$ is the non-trivial element in $\Gal(E/\Q)$; and
    \item $\iota(a)$ acts on $\Lie A$ with characteristic polynomial $(T-\varphi(a))^r(T-\varphi(a)^\sigma)$, where $\varphi:\spec\cO_E\rightarrow S$ is the structure morphism.
\end{enumerate}
 This moduli space $\cM(r,1)$ is a Deligne--Mumford stack over $\cO_E$ such that $\cM(r,1)_E$ is a disjoint union of Shimura varieties attached to $\GU(r,1)$ (see for instance \cite[Prop.~2.19, Prop.~4.4]{KR14}). Similarly, we define $\cM(1,0)$. In particular, after enlarging $K$ by a finite extension which contains $E$, the abelian variety $A$ in \Cref{cor_unitary} gives a $K$-point on $\cM(r,1)$.


In order to relate $\cM(r,1)$ to the GSpin Shimura variety $\cM$ defined in \S\ref{gspin}, we pick an auxillary elliptic curve $A_0$ defined over a finite extension of $E$ such that $\cO_E\subset \End(A_0)$ and the action of $\cO_E$ on $\Lie A_0$ is given by the embedding of $\cO_E$ into the definition field of $A_0$ and hence $A_0$ is a point on $\cM(1,0)$. As in \cite[\S\S 2.1, 2.2]{BHKRY}, pick an embedding of the definition field of $A_0$ (resp.~$A$) into $\C$, and let $W_0$ (resp.~$W$) denote the $E$-vector space $H_{1,B}(A_0(\C), \Q)$ (resp.~$H_{1,B}(A(\C), \Q)$), where the $E$-vector space structure is induced by the $\cO_E$-action on $A_0$ (resp.~$A$). There exists a unique Hermitian form $\psi$ of signature $(r,1)$ on $W$ such that the symplectic form on $H_{1,B}(A(\C),\Q)$ induced by the polarization equals to $\tr_{E/\Q}((\disc E)^{-1/2}\psi)$. Similarly, there exists a Hermitian form $\psi_0$ of signature $(1,0)$ on $W_0$ such that $\tr_{E/\Q}((\disc E)^{-1/2}\psi_0)$ induces the polarization on $A_0$. By \cite[eqns (2.1.4), (2.1.5)]{BHKRY}, $\psi_0$ and $\psi$ induce a Hermitian form $\phi$ on the $E$-vector space $\Hom_{\cO_E}(H_{1,B}(A_0(\C), \Q), H_{1,B}(A(\C),\Q))$ of signature $(r,1)$. Let $V$ denote the $\Q$-vector space $\Hom_{\cO_E}(H^1_B(A_0(\C), \Q), H^1_B(A(\C),\Q))$ endowed with the quadratic form $\tr_{E/\Q}\phi$ and $V$ is of signature $(2r,2)$. 

Let $G'$ denote the subgroup of $\GU(W_0,\psi_0)\times \GU(W, \psi)$ given by pairs whose similitude factors are equal. By \cite[\S\S 2.1, 6.2]{BHKRY} and \cite[\S 4]{Hof}, the induced action of $G'$ on $V$ gives a group homomorphism $G'\rightarrow \SO(V)$ and this group homormorphism is indeed a map between Shimura data (with the Hodge cocharacters given by the ones induced by the Hodge cocharacters of $A_0$ and $A$); hence we have a map between Shimura varieties $\Sh(G')\rightarrow \Sh(\SO(V))$ (with the maximal compact open subgroups of $G'(\A_f)$ and $\SO(V)(\A_f)$ defined by lattices in $W_0$ and $W$ given by $H_{1,B}(A_0(\C), \Z)$ and $H_{1,B}(A(\C), \Z)$). By \cite[Prop.~2.2.1]{BHKRY}, $\Sh(G')\subset \cM(1,0)\times \cM(r,1)$ and $(A_0,A)$ gives a $\overline{\Q}$-point on $\Sh(G')$ and hence a $\overline{\Q}$-point on $\Sh(\SO(V))$. 

Note that $\GSpin(V)\rightarrow \SO(V)$ induces an open and closed morphism $M\rightarrow \Sh(\SO(V))$ (with a suitable choice of maximal compact subgroups, which does not affect the rest of the argument), where $M$ is the $\GSpin$ Shimura variety defined in \S\ref{gspin_Q}; therefore, by applying a suitable Hecke translate on $\Sh(\SO(V))$, the image of the point $(A_0,A)$ under the Hecke translate lies in a connected component of $\Sh(\SO(V))$ which lies in the image of $M$. In particular, there exists a point $Y\in M(\overline{\Q})$ such that $Y$ maps to the Hecke translate of $(A_0,A)$ and hence as $\Q$-Hodge structures, 
\[\bfV_{B,Y}\otimes \Q\cong \Hom_{\cO_E}(H_{1,B}(A_0(\C),\Q), H_{1,B}(A(\C),\Q)),\]
where $\bfV_{B,Y}$ denotes the fiber at $Y$ of the local system $\bfV_B$ defined in \S\ref{sec_KS}. We enlarge $K$ by a finite extension so that $Y, A_0$ and $A$ are all defined over $K$. Since all Hodge cycles in the category of absolute Hodge motives generated by abelian varieties are absolute Hodge by Deligne's theorem, then after enlarging $K$ by a finite extension, we have 
\begin{equation}\label{isom_Gal}
    \bfV_{\ell, \textrm{\'et},Y}\cong \Hom_{\cO_E}(H_{1,\textrm{\'et}}(A_0, \Q_\ell), H_{1,\textrm{\'et}}(A, \Q_\ell))
\end{equation}
as $\Gal(\overline{K}/K)$-modules. We use $A^\ks$ to denote the Kuga--Satake abelian variety corresponding to $Y\in M(K)$.





\begin{proof}[Proof of \Cref{cor_unitary}]
Since $A$ and $A_0$ have potentially good reduction everywhere, then after enlarging by $K$ by a finite extension such that both $A$ and $A_0$ have good reduction over $K$, the Galois representation $\Hom_{\cO_E}(H_{1,\textrm{\'et}}(A_0,\Q_\ell), H_{1,\textrm{\'et}}(A,\Q_\ell))$ is unramified away from $\ell$. By \cite[Lemma~9.3.1]{andretate} and eqn.~\eqref{isom_Gal}, the Kuga--Satake abelian variety $A^{\ks}$ has potentially good reduction everywhere. Then by \Cref{arithmeticstack}(4), $Y$ extends to an $\mathcal{O}_K$-point $\cY$ of $\mathcal{M}$. By \Cref{main} and the definition of special endomorphisms (\Cref{def_sp_end}), there are infinitely many places $\fP$ such that $\bfV_{\ell, \textrm{\'et},\cY_\fP}$ admits a Tate cycle (after possible finite extension of the residue field) for $\ell$ not equals to the residue characteristic of $\fP$. 
For such a prime $\fP$, by eqn.~\eqref{isom_Gal}, there exists an $n\in \Z_{>0}$ such that $\Hom_{\cO_E}(H_{1,\textrm{\'et}}(A_0, \Q_\ell), H_{1,\textrm{\'et}}(A,\Q_\ell))^{\Frob^n_\fP=1}\neq \emptyset$ . In particular, $A_{0,\overline{\fP}}$ is an isogeny factor of $A_{\overline{\fP}}$ by Tate's theorem.
\end{proof}

\bibliographystyle{alpha}
\bibliography{bibliographie}
\end{document}